\documentclass[12pt]{amsart}
\usepackage[margin=1.5cm]{geometry}
\usepackage{setspace}
\usepackage{hyperref}
\usepackage{graphicx}
\usepackage{amsmath,amsthm}
\usepackage{amsfonts,amssymb}
\usepackage{amscd}
\usepackage{amsthm}
\usepackage{quiver}
\usepackage{verbatim}
\usepackage[cp1251]{inputenc}
\usepackage[english]{babel}
 \usepackage{tikz}
\usepackage[all,cmtip]{xy}
\usepackage{tikz-cd}
\usepackage{makecell,xcolor}
\usepackage{float}
\usepackage{subcaption} 
\usepackage{eucal}
\usepackage{mathalfa}
\usepackage{mathrsfs}
\usepackage{fancyhdr}
\theoremstyle{plain}

\usetikzlibrary{matrix,arrows,decorations.pathmorphing}

\tikzstyle{slice}=[text=gray]

\usepackage{xcolor}

\usepackage{enumitem}
\newtheorem{lemma}{Lemma}%[section]
\newtheorem{theorem}{Theorem}%[section]
\theoremstyle{remark}
\newtheorem*{remark}{Remark}
\newtheorem*{theorem*}{{\bf Theorem}}

\theoremstyle{definition}
\newtheorem{definition}{Definition}[section]

\newcommand{\End}{\mathop{\mathrm{End}}\nolimits}
\def\N{\mathbb N} \def\Z{\mathbb Z}

\newcommand{\proj}{{\mathrm{proj}}}
\newcommand{\Hom}{{\mathrm{Hom}}}

\newcommand{\RHom}{\mathop{\mathrm{RHom}}\nolimits}

\newcommand{\Aut}{{\mathrm{Aut}}}

\newcommand{\add}{{\mathrm{add}}}
\newcommand{\Rad}{\mathop{\mathrm{Rad}}\nolimits}
\newcommand{\Out}{{\mathrm{Out}}}
\newcommand{\Inn}{\mathop{\mathrm{Inn}}\nolimits}
\newcommand{\Pic}{{\mathrm{Pic}}}
\newcommand{\TrPic}{{\mathrm{TrPic}}}
\newcommand{\StPic}{{\mathrm{StPic}}}
\newcommand{\thick}{\mathop{\mathrm{thick}}\nolimits}
\newcommand{\Imm}{\mathop{\mathrm{Im}}\nolimits}
\newcommand{\D}{{\mathrm{D}}}
\newcommand{\Cmp}{{\mathrm{C}}}
\newcommand{\K}{{\mathrm{K}}}
\newcommand{\Li}{{\mathrm{L}}}
\newcommand{\charr}{{\mathrm{char}}}

\def\kk{\mathrm{k}}

\makeatletter
\renewcommand{\email}[2][]{%
  \ifx\emails\@empty\relax\else{\g@addto@macro\emails{,\space}}\fi%
  \@ifnotempty{#1}{\g@addto@macro\emails{\textrm{(#1)}\space}}%
  \g@addto@macro\emails{#2}%
}
\makeatother

\def\@evenhead{\hfil\sc A. Nordskova, \hfil} 
\def\@oddhead{\hfil\sc  The derived Picard groups of symmetric representation finite algebras of type D \hfil} 
\def\@oddfoot{\hfil\sc \thepage \hfil} 
\def\@evenfoot{\hfil\sc \thepage \hfil} 

\def\ps@firstpage{\ps@plain
  \def\@oddfoot{\hfil\sc \thepage \hfil}%

}
\makeatother

\author{Anya Nordskova}
\address{L. Euler International Mathematical Institute, Saint Petersburg State University, 14th Line V.O., 29B, St. Petersburg 199178, Russia.}

\thanks{\emph{email:} anya.nordskova@uhasselt.be   } 
\begin{document}

\title{Derived Picard groups of symmetric representation-finite algebras of type $D$}
\maketitle

\begin{abstract}
   We explicitly describe the derived Picard groups of symmetric representation-finite algebras of type $D$. In particular, we prove that these groups are generated by spherical twists along collections of $0$-spherical objects, the shift and autoequivalences which come from outer automorphisms of a particular representative of the derived equivalence class. The arguments we use are based on the fact that symmetric representation-finite algebras are tilting-connected. To apply this result we in particular develop a combinatorial-geometric model for silting mutations in type $D$, generalising the classical concepts of Brauer trees and Kauer moves. Another key ingredient in the proof is the faithfulness of the braid group action via spherical twists along $D$-configurations of $0$-spherical objects. \end{abstract}

\tableofcontents 

\newpage 

\section{Introduction}

Let $\kk$ be an algebraically closed field and let $A$ and $B$ be $\kk$-algebras. Classical Morita theory studies the relation of Morita equivalence. Namely $A$ and $B$ are Morita equivalent if their module categories are equivalent. One can go further and consider the weaker relation of derived equivalence instead. We say that two algebras $A$ and $B$ are derived equivalent if their bounded derived categories $\D^b(A)$ and $\D^b(B)$ are equivalent as triangulated categories. Rickard \cite{R}, \cite{R2} answered to the question when two algebras are derived equivalent, using the language of tilting complexes. His results were then generalised to the differential graded setting by Keller \cite{K}. The seminal work of Bondal \cite{Bo}, which most importantly jump-started the study of equivalences between derived categories of algebraic and geometric origin, also can be applied to the case of two algebras. The derived Picard group $\TrPic(A)$ of $A$, introduced independently by Rouquier-Zimmermann \cite{RZ} and Yekuteli \cite{Y}, is the group of so-called standard autoequivalences of the bounded derived category $\D^b(A)$ modulo natural isomorphisms. Equivalently, $\TrPic(A)$ can be defined as the group of isomorphism classes of two-sided tilting complexes over $A$ with respect to the tensor product. The derived Picard group is an important derived invariant of an algebra. It is closely related to the Hochschild cohomology \cite{K2}, another derived invariant. Derived Picard groups are known to be locally algebraic \cite{Y}. 

The derived Picard group of an algebra is a natural analogue of the classical Picard group in the context of derived Morita theory. However, while the classical Morita theory for finite-dimensional $\kk$-algebras is very well understood (i.e. the classical Picard group of a basic algebra is the group of its outer automorphisms), there is no general recipe known for computing derived Picard groups. So far, complete descriptions of derived Picard groups have been obtained only for several very nice classes of algebras, e.g. for commutative, for local, for finite-dimensional hereditary algebras, for derived discrete algebras, for affine Azumaya algebras, for preprojective algebras of types $A$ and $D$, and for self-injective Nakayama algebras (see \cite{RZ}, \cite{MY}, \cite{BPP}, \cite{Ne}, \cite{Miz}, \cite{VZ1}). On the other hand, the study of derived Picard groups may both contribute to the understanding of algebras in question and shed some light on the structure of their derived categories. For instance, classifying all tilting complexes in a derived category, which is a problem closely related to the problem of describing derived Picard groups, may give some insights into $t-$structures in this category.  

Self-injective algebras play a prominent role in the representation theory of algebras. It is a wide class of algebras that includes many interesting examples frequently appearing ``in nature'' (e.g. group algebras of finite groups, finite-dimensional Hopf algebras). At the same time, it is manageable enough to study, since self-injective algebras share quite a lot of non-trivial properties and structure (e.g. their stable categories are triangulated). The simplest self-injective algebras are those of finite representation type. Such algebras were classified up to derived equivalence by Asashiba \cite{A}. The derived equivalence classes correspond to triples $(\Gamma, q, t)$, where $\Gamma$ is a simply-laced Dynkin diagram, reflecting the form of the stable Auslander-Reiten quiver of an algebra, $q$ is a rational number called ``frequency'' and $t$ is a number called ``torsion'' (which is either $1, 2$ or $3$). For type $A$ and torsion $1$ the derived Picard groups of such algebras were described by Volkov and Zvonareva \cite{VZ1}, \cite{Alexandra}. 

In this paper we explicitly describe the derived Picard groups of symmetric representation-finite algebras of type $D$. In terms of the classification mentioned above, these are precisely the algebras of type $D$ with frequency $1$ and torsion $1$. We prove that the derived Picard groups of such algebras are isomorphic to extensions of certain braid groups quotients by $ \Z/2\Z \times k^* \times \Z$. More precisely, the result we obtain is as follows. Let $m \geq 4$ and let $B_{D_m}$ denote the Artin group or the generalised braid group of type $D_m$ (Definition \ref{def:braid}). By $\{\sigma_i\}_{i=1}^m$ we denote the set of standard generators of the braid group $B_{D_m}$, i.e. it is always assumed that $\sigma_i$ satisfy braid relations of type $D_m$. Let $c =\sigma_1\dots\sigma_m \in B_{D_m}$. When we write $\kappa_a$ for $a \in \kk^*$, we mean that $\kappa_a$ satisfy the same relations which hold for the corresponding elements of the multiplicative group of the field. 

\newpage

\begin{theorem*}[Theorem \ref{thm:main}]  Let $c =\sigma_1\dots\sigma_m \in B_{D_m}$.
 {\it 
 Define the following group by generators and relations 

$$ \widetilde{G}_{m} = \bigg\langle \sigma_i \in B_{D_{m}},\{\kappa_a\}_{a \in \kk^*}, \tau, s \bigg| \begin{array}{c}
\tau \sigma_i = \sigma_i \tau \ (i\neq 1,2); \ \tau \sigma_1 = \sigma_2 \tau; \  \tau^2 = e; \  \\  s \text{ and }  \kappa_a \ \forall a \in \kk^* \text{ are central}
\end{array} \bigg\rangle $$

Let $G_m = \widetilde{G}_m/ \langle c^{m-1}s^{-2m+3} \kappa_{-1} \rangle$ if $m$ is even and $G_m = \widetilde{G}_m/ \langle c^{m-1}s^{-2m+3} \kappa_{-1}\tau \rangle$ if $m$ is odd. Then the derived Picard group of any algebra of type $(D_m,1,1)$ is isomorphic to $G_m$. }
\end{theorem*}

Since all algebras of type $(D_m,1,1)$ are derived equivalent, it is sufficient to describe the derived Picard group of just one of them. However, our method in fact forces us to deal with not one, but all (up to Morita equivalence) algebras in this derived equivalence class. Two algebras $\Lambda_m$ and $\mathcal{R}_m$ (see Definitions \ref{def:Lambda} and \ref{def:R}) play a particularly important role in the proof. 

The group $G_m$ consists of three building blocks. The first block is $\langle s \rangle \cong \Z$, which corresponds to the subgroup generated by the shift functor. The second block is the braid group $\langle \sigma_1, \dots, \sigma_m \rangle \cong B_{D_m}$, which corresponds to the subgroup generated by spherical twists along a $D_m$-configuration of $0$-spherical objects. The fact that these spherical twists generate the braid group $B_{D_m}$ inside the derived Picard group follows from a non-trivial result on the faithfulness of braid group actions obtained in \cite{NV}. Although the setup of \cite{NV} is much more general, it is precisely the case of 0-spherical objects which appears to be the most involved, often requiring special attention. For one of our two ``favourite'' algebras $\mathcal{R}_m$ the collection of $0$-spherical objects in question is formed by indecomposable projective $\mathcal{R}_m$-modules. Finally, the third building block is $\langle \tau, \{\kappa_a\}_{a \in k*} \rangle \cong \Z/2\Z \times k^*$. This piece is given by the classical Picard group $\Pic(\Lambda_m)$ of our other ``favourite'' algebra $\Lambda_m$, or, equivalently, the group of outer automorphisms of $\Lambda_m$. Note that $\Pic$ is not a derived invariant, so other algebras in the same derived equivalence class may have different classical Picard groups.

When these three blocks are understood separately, two questions need to be answered to obtain the desired description of the derived Picard groups. Namely, (1) what are the relations between elements of different blocks and (2) why do they generate the whole derived Picard groups? While the first question boils down to a rather straightforward computation (Section \ref{sec:2alg}), the second requires much more work. To answer it we apply the technique used by Zvonareva for Brauer star algebras \cite{Alexandra}. The key ingredient there is the result of Aihara \cite{Aih} stating that every symmetric representation-finite algebra is tilting-connected. This allows us, although not immediately, to reduce the problem to computing all sequences of silting mutations (in the sense of Aihara and Iyama \cite{AI}) of a particular form and then expressing them as compositions of twists, shifts and outer automorphisms (Sections \ref{sec:mainres}-\ref{sec:app}). In turn, for this we first need to develop a good understanding of silting mutations for the class of algebras we are considering. We therefore start by introducing a combinatorial-geometric model for computing silting mutations in our case, generalising the classical concepts of Brauer trees and Kauer moves to suit our needs (Section \ref{sec:mutmodel}).
\subsection{Possible generalisations} One obvious way to continue this work would be to complete the description of the derived Picard groups of all symmetric representation-finite algebras by considering the remaining case of type $E$. It seems to us that similar methods can be used here, although they are likely to face more technical complications. However, it would be especially interesting to see if the result can be extended to a more general categorical context and proved without the specific machinery of the representation theory of finite-dimensional algebras. Let $\mathcal{D}$ be an enhanced $k-$linear triangulated category, whose full subcategory $\mathcal{D}^c$ of compact objects is classically generated by a collection of $0$-spherical objects, forming an ADE-configuration. Instead of the derived Picard group of an algebra, one can consider the group of enhanced autoequivalences of $\mathcal{D}$, i.e. those induced by autoequivalences of a fixed dg-enhancement, modulo natural isomorphisms. Then, in the light of our result, as well as the result of Zvonareva \cite{Alexandra} for Brauer tree algebras, it is natural to ask if this group is always generated by spherical twists along the spherical objects of the collection, modulo its identity component, automorphisms of the quiver corresponding to the spherical objects and the shift. There are several subtleties that one has to overcome in order to make this formal (e.g. defining the identity component), so we are not formulating a precise conjecture yet. Note that our result for the case of $D_m$ with even $m$ shows that there indeed might be autoequivalences induced by automorphisms of the Dynkin diagram that cannot be expressed via other elements. Naturally, one can go further and ask the same question about categories generated by collections of $n$-spherical objects for $n \neq 0$. For $n > 0$ it would be curious to see if there are geometric examples that can provide some insights. On the other hand, one can also consider non-Dynkin tree-like configurations of spherical objects. While the braid group action can be unfaithful in this situation, a priori this does not seem to be an immediate obstruction to describing a set of generators. 
\subsection{Notations and conventions}
Throughout the paper $\kk$ denotes an algebraically closed field. All algebras are assumed to be basic. Let $A$ be a $\kk-$algebra. By an $A$-module we mean a right $A$-module unless stated otherwise, and we will write $M \otimes N$ for $M \otimes_\kk N$. Given a quiver $Q$ we denote the product of arrows $\xrightarrow{\alpha}\xrightarrow{\beta}$ in its path algebra $\kk Q$ by $\alpha\beta$. By $A^{op}$ we denote the opposite algebra of $A$. For $\sigma \in \Aut(A)$ we denote by  ${}_\sigma A$ the $A$-bimodule which is regular regarded as a right module and has left multiplication twisted by $\sigma$, i.e. $a \cdot m = \sigma(a)m$ for any $m \in {}_\sigma A$, $a \in A$. By $\Cmp(A)$ we denote the category of complexes of $A$-modules, by $\K(A)$ the homotopy category of $\Cmp(A)$, and by $\K^b(\proj-A)$ the subcategory of $K(A)$ of bounded complexes of finitely generated projective modules. By $\D^b(A)$ we denote the bounded derived category of finitely generated $A$-modules. We refer to objects of $\D^b(A)$ quasi-isomorphic to bounded complexes of finitely generated projective modules as {\it perfect complexes}. If $T \in \mathcal{T}$ is an object of a triangulated category $\mathcal{T}$, we will denote by $\add(T)$ the smallest subcategory of $\mathcal{T}$ containing $T$, closed under direct summands, finite direct sums, and isomorphisms. We say that $T$ generates $\mathcal{T}$ if $\mathcal{T}$ is the smallest triangulated subcategory of $\mathcal{T}$ containing $T$ and closed direct summands and isomorphisms. In any triangulated category the shift functor will be denoted by $[1]$. When we write a chain complex it is always assumed that the leftmost term is placed in degree zero, unless another term is underlined. For instance, $X$ is in degree $0$ in $X \xrightarrow{} Y \xrightarrow{} \dots$, but in degree $-1$ in $X \xrightarrow{} \underline{Y} \xrightarrow{} \dots$. 

\subsection{Acknowledgements} 
This work is based on the master's thesis of the author completed under the supervision of Yury Volkov and Alexandra Zvonareva. I thank them both for their guidance and attention. I would particularly like to express my gratitude to Alexandra for her invaluable help with the text of the paper.  

This project has received funding from the European Research Council (ERC) under the European Union's Horizon 2020 research and innovation programme (grant agreement No 885203).  It was also supported by Ministry of Science and Higher Education of the Russian Federation, agreement \textnumero 075-15-2019-1619 and by RFBR grant 20-01-00030. As a winner of the M\"{o}bius Contest, the author was also in part supported by the M\"{o}bius Contest Foundation for Young Scientists. 

\section{Preliminaries}\label{sec:prelim} 

\subsection{Derived Morita theory and mutations of symmetric algebras}

\begin{definition} A $\kk$-algebra $A$ is {\bf self-injective} if it is injective as a right module over itself, or, equivalently, if every
projective $A$-module is injective.  A $\kk$-algebra $A$ is {\bf symmetric} if $A$ is isomorphic to $\Hom_\kk(A,\kk)$ as an $A-A$-bimodule. \end{definition}

\begin{remark}  It is easy to see that every symmetric algebra is self-injective.  \end{remark}

\begin{theorem}[Rickard, \cite{R}, \cite{R2}]\label{thm:R}
Let $A$ and $B$ be $\kk-$algebras. The following assertions are equivalent. 
\begin{enumerate}
    \item The categories $\D^b(A)$ and $\D^b(B)$ are equivalent as triangulated categories. 
    \item The categories $\K^b(\proj-A)$ and $\K^b(\proj-B)$ are equivalent as triangulated categories. 
    \item There is a complex $T \in \K^b(\proj-A)$ such that $T$ generates $\K^b(\proj-A)$, $B$ is isomorphic to $\End_{\D^b(A)}(T) $ as a $\kk-$algebra, and $\Hom_{\D^b(A)}(T,T[i]) = 0$ for $i\neq 0$. 
    \item There are a bounded complex $X$ of $B^{op} \otimes A$-modules whose restrictions to $A$ and $B^{op}$ are perfect and a bounded complex $Y$ of $A^{op} \otimes B$-modules whose restrictions to $B$ and $A^{op}$ are perfect. Moreover, $Y \otimes_B^{{\bf \Li}}X  \cong A$ in $\D^b(A^{op} \otimes A)$ and $X\otimes_A^{{\bf \Li}} Y \cong B$ in $\D^b(B^{op} \otimes B)$. 
\end{enumerate}
\end{theorem}

\begin{remark} It is worth remarking that the equivalence between $2)$ and $3)$ in Theorem \ref{thm:R} appeared already in the paper \cite{Bo} of Bondal, a little earlier. Although there he considered a collection of generating objects rather than one generator (tilting complex), for finite-dimensional algebras this gives the same result. \end{remark}

When $A$ and $B$ satisfy the conditions above, we say that they are \emph{derived equivalent}. In this case a complex $T$ is called a \emph{tilting complex} and complexes $X$ and $Y$ are called \emph{two-sided tilting complexes}, inverse to each other. The restrictions of two-sided tilting complexes are tilting complexes. On the other hand, if $T$ is a tilting complex with $\End_{\D^b(A)}(T) \cong B$, then there exists a two-sided tilting complex $X$ of $B^{op} \otimes A$-modules which is isomorphic to $T$ in $\D^b(A)$ when considered as a complex of $A$-modules. Moreover, such a complex $X$ is unique up to an automorphism of $B$, namely, any two-sided tilting complex whose restriction to $A$ is $T$ has the form $_\sigma B \otimes_B X$ for some $\sigma \in \Aut(B)$ (see \cite{RZ}).

If $X$ and $Y$ are two-sided tilting complexes as above, then $X \otimes_A^{{\bf \Li}} - $ and $Y \otimes_B^{{\bf \Li}} - $ are quasi-inverse equivalences between $\D^b(A)$ and $\D^b(B)$. Equivalences of this form are called {\bf standard}.

\begin{definition}[Rouquier, Zimmermann \cite{RZ}, Yekutieli \cite{Y}] The {\bf derived Picard group} $\TrPic(A)$ of a $\kk-$algebra $A$ is the group of isomophism classes of two-sided tilting complexes of $A^{op} \otimes A$-modules, where the product of classes of $X$ and $Y$ is given by the class of $X \otimes_A Y$. In other words, $\TrPic(A)$ is the group of standard autoequivalences of $\D^b(A)$ modulo natural isomorphisms. 
\end{definition}

A complex $S \in \K^b(\proj-A)$ is called {\bf silting} if $\Hom_{D^b(A)}(S, S[i]) = 0$ for $i>0$ and $S$ generates $\K^b(\proj-A)$ as a triangulated category.  Let $S$ be a basic silting complex, meaning that it is isomorphic to a direct sum of pairwise non-isomorphic indecomposable complexes in $\K^b(\proj-A)$. Let $S = M \oplus X$. A {\bf minimal left $\add(M)$-approximation} of $X$ is a morphism $X \xrightarrow{f} M'$, where $M'$ is an object of $\add(M)$ such that $\Hom(M',M'') \xrightarrow{f^*} \Hom(X,M'')$ is surjective for every $M'' \in \add(M)$ and any morphism $M' \xrightarrow{g} M'$ with $gf = f$ is an isomorphism. Now let $Y$ be the cone of $f$ in $\K^b(\proj-A)$. Note that $Y$ is defined uniquely up to an isomorphism. Define the {\bf left mutation of $S$ with respect to $X$} to be the object $\mu_X^+(S) = M \oplus Y$. Then $\mu_X^+(S)$ is again a basic silting complex (see \cite{AI} for this and other details regarding silting mutations). Dually, one can define the right mutations of $S$ with respect to $X$, which is denoted by $\mu_X^-(S)$. One can show that $\mu_Y^-(\mu_X^+(S)) = S$. A mutation with respect to an indecomposable object is called {\bf irreducible}. Note that if $A$ is a symmetric algebra, then every silting complex is automatically a tilting complex, because $\Hom_{\D^b(A)}(S,S[i]) = D\Hom_{\D^b(A)}(S[i],S)$, where $D$ denotes the $\kk-$dual.

\subsection{Brauer trees and Brauer tree algebras}

Now we are going to briefly recall the concepts of Brauer trees and Brauer tree algebras. Although in this paper we mainly deal with a different class of algebras, the theory of Brauer tree algebras is crucial for understanding some of our constructions and seeing the intuition behind them. We must warn the reader that what we define below as a Brauer tree is classically referred to as a Brauer tree with exceptional vertex of multiplicity $1$ (or simply without an exceptional vertex). But since in the paper only this special case is used, we drop all mentions of exceptional vertices and simply write ``Brauer trees'' and ``Brauer tree algebras''. 

\begin{definition}
 A {\bf Brauer tree} $G$ is a finite tree with a fixed cyclic ordering of half-edges incident to every vertex. It is customary to fix an embedding of $G$ into the plane and assume that the ordering of half-edges is given by the clockwise orientation of the plane. 
\end{definition}

\begin{definition} A {\bf Brauer tree algebra} of the Brauer tree $G$ is an algebra isomorphic to $A_G = \kk Q_G/I_G$ where the quiver $Q_G$ and the ideal of relations $I_G$ are defined in the following way. The vertices of $Q_G$ are in one to one correspondence with the edges of $G$, and we will denote them by the same letters. Now let $a$ be a vertex of $G$ and let $x$ be an edge of $G$ incident to $a$. If $y$ is an edge of $G$ immediately following $x$ in the cyclic ordering ``around'' their common vertex $a$, then there is an arrow $(x \to y)$ in $Q_G$. One can see that in this way vertices of $G$ correspond to cycles in $Q_G$. Note that if $a$ is a leaf, we automatically have $x = y$ and the resulting arrow in $Q_G$ is a loop. We refer to these loops as {\bf formal} and do not draw them when depicting the quiver. It is easy to see that the arrows of $Q_G$ split into a union of cycles, where every vertex belongs to exactly two cycles. The two-sided ideal $I_Q$ is generated by the following relations: 

\begin{enumerate}
    \item The two cycles starting at any fixed vertex are equal. 
    \item Any product of two arrows not belonging to the same cycle is zero. 
\end{enumerate}
\end{definition}
\begin{remark} The ideal $I_G$ defined above is not admissible, but in our context this presentation is the most convenient. 
\end{remark}

Brauer tree algebras are always symmetric. A lot, not to say everything, is known about mutations of tilting complexes over Brauer tree algebras. More precisely, let $A = \bigoplus_{i=1}^n P_i = \bigoplus_{i=1}^n e_i A$ be a Brauer tree algebra considered as the direct sum of its indecomposable projective modules. Let $G$ be the Brauer tree of $A$. For any $j = 1, \dots, n$ it is easy to compute the tilting complex $\mu_{j}^{\pm}(A) := \mu_{P_j}^{\pm}(A) = \bigoplus_{i \neq j} P_i \oplus P_j'$. Moreover, $\End_{D^b(A)}(\mu_{j}^{\pm}(A))$ is again a Brauer tree algebra. Its Bauer tree $G'$ can be described explicitly. Namely, $G'$ differs from $G$ only in the edge corresponding to the vertex $j$ of the quiver of $A$. The transformation changing the position of this one edge is usually referred to as a {\bf Kauer move} or a {\bf flip}. For further details we refer the reader to, for instance, \cite{Kauer}, \cite{Aih2}, \cite{Aih3}. 

\subsection{Symmetric representation-finite algebras}

\begin{theorem}[Aihara, \cite{Aih}]\label{thm:tiltconn} Let $A$ be a symmetric representation-finite algebra. Then any tilting complex $T \in \K^b(\proj-A)$ concentrated in non-positive degrees can be obtained via a sequence of irreducible left mutations from $A$ regarded as a tilting complex over itself. 
\end{theorem}

Continuing the work of Riedtmann \cite{Ried2}, Asashiba \cite{A} classified self-injective algebras of finite representation type up to derived equivalence. Riedtmann established that the stable AR quiver of such an algebra has the form $\Z \Gamma/G$, where $\Gamma$ is a simply-laced Dynkin diagram and $G$ is a cyclic group generated by an element $\tau^{ql_{\Gamma}}\varphi$. Here $\tau$ is the AR translation, $l_\Gamma$ is the Loewy lengh of the mesh category of $\Z\Gamma$, $q$ is some rational number called {\it frequency} and $\varphi$ is an automorphism of the diagram $\Gamma$. If $\varphi$ is an automorphism of order $r \in \N$, we say that in this case a representation-finite self-injective algebra is of type $(\Gamma, q, r)$. Asashiba found the complete list of possible types and proved that derived equivalence classes of representation-finite self-injective algebras are, with one minor exception, classified by such triples. More precisely, there are only algebras of types $(A_m, \frac{k}{m}, 1)$, $(A_{2m+1}, k, 2)$, $(D_m, \frac{k}{gcd(m,3)}, 1)$, $(D_m, k, 2) $, $(D_4, k, 3)$, $(E_6, k, 2)$, and $(E_d, k, 1)$, for ${k,m \in \N, d = 6,7,8}$. Moreover, algebras of the same type are derived equivalent, except for the case $(D_{3m}, \frac{1}{3}, 1)$ and $\charr \kk = 2$, where there are two derived equivalence classes.

In this paper we are interested in symmetric representation-finite algebras of type $D$. In terms of the classification we discussed above these are algebras of types $(D_m, 1, 1)$, $m \in \N$.  We denote the class of all algebras of type $(D_m, 1, 1)$ by $\mathfrak{C}(m)$. Equivalently, $\mathfrak{C}(m)$ can be defined as the class of all algebras derived equivalent to the trivial extension algebra of the $D_m$ Dynkin diagram with alternating orientation (see Definition \ref{def:R}). Let us point out that Brauer tree algebras are derived equivalent to the trivial extension algebra of the $A_m$ Dynkin diagram with alternating orientation.

\begin{definition}\label{def:braid} The {\bf Artin group} or the {\bf generalised braid group} $B_{D_m}$ of type $D_m$ is a group given by the following generators and relations

\[\begin{tikzcd}
	&&&&& {\small{1}} \\
	{\small{m}} & {\small{m-1}} && {\small{4}} & {\small{3}} && { B_{D_n} = \bigg\langle \{\sigma_i\}_1^m \bigg| \begin{array}{c} \sigma_i \sigma_{j} \sigma_i = \sigma_{j} \sigma_i \sigma_{j} \text{ if } \exists \text{ an edge } (i,j) \\ \  \sigma_i\sigma_j = \sigma_j\sigma_i \ \text{ if } \not\exists \text{ an edge } (i,j) \end{array} \bigg\rangle  } & {} \\
	&&&&& {\small{2}}
	\arrow[no head, from=2-1, to=2-2]
	\arrow[no head, from=2-5, to=1-6]
	\arrow[no head, from=2-5, to=3-6]
	\arrow[no head, from=2-4, to=2-5]
	\arrow[dashed, no head, from=2-2, to=2-4]
\end{tikzcd}\]

%$$ B_{D_n} = \bigg\langle \{\sigma_i\}_1^n \bigg| \begin{array}{c} \sigma_i \sigma_{i+1} \sigma_i = \sigma_{i+1} \sigma_i \sigma_{i+1} \ (3 \leq i \leq n) \\ \sigma_1\sigma_3\sigma_1 = \sigma_3\sigma_1\sigma_3, \  \sigma_2\sigma_3\sigma_2 = \sigma_3\sigma_2\sigma_3 \\ \sigma_1\sigma_2 = \sigma_2\sigma_1, \ \sigma_i\sigma_j = \sigma_j\sigma_i \ (|i-j| \geq 2) \end{array} \bigg\rangle  $$
 
 In other words, generators of $B_{D_n}$ correspond to vertices of the $D_n$ Dynkin diagram, every two generators corresponding to non-adjacent vertices commute, and every two generators corresponding to adjacent vertices satisfy the braid relation. 
\end{definition}

\section{Modified Brauer trees and mutations for type \texorpdfstring{$D$}{Lg}}\label{sec:mutmodel}
The aim of this section is to introduce a geometric model for tilting mutations of symmetric representation-finite algebras of type $D$, in a way generalising and adapting the classical concepts of Brauer trees and Kauer moves. This construction will be our key tool in the proof of the main result of the paper.  

Let $m \geq 4$ be an integer. Recall that by $\mathfrak{C}(m)$ we denote the class of all algebras derived equivalent to the trivial extension algebra of the $D_m$ Dynkin diagram with alternating orientation (see Definition \ref{def:R}). Equivalently, $\mathfrak{C}(m)$ consists of all algebras of type $(D_m, 1, 1)$ in the sense of Asashiba \cite{A}. 

\subsection{Algebras}

Let us begin by showing how to conveniently depict the algebras in $\mathfrak{C}(m)$ themselves, before we discuss mutations. Namely, we are going to define combinatorial ribbon graph-like objects corresponding to symmetric representation-finite algebras of type $D$ in a similar way as Brauer trees correspond to Brauer tree algebras. The newly introduced objects will be referred to as {\bf modified Brauer trees}. The class $\mathfrak{C}(m)$ splits into two subclasses which we are going to describe below and refer to as {\bf Double Edge} algebras and {\bf Triple Tree} algebras. These two subclasses correspond to two different shapes of modified Brauer trees. Such a classification is not new, for instance, it can be derived from the works of Riedtmann and Asashiba (see \cite{Ried}, \cite{A}). However, we will naturally rediscover it along the way. Namely, using the language of modified Brauer trees we will describe (up to Morita equivalence) all algebras that can be reached by tilting mutations from one particular algebra $\Lambda \in \mathfrak{C}(m)$. Then the fact that symmetric representation-finite algebras are tilting-connected (Theorem \ref{thm:tiltconn}) automatically gives us a classification of all algebras in $\mathfrak{C}(m)$. 

First we define {\bf Double Edge algebras}. Up to Morita equivalence they can all be obtained by the following procedure. Let $G$ be a Brauer tree with $m-2$ edges and let $Q_G$ be the quiver of the Brauer tree algebra $A_G$ corresponding to $G$. Fix an arrow ${(\gamma \colon x_1 \to x_2) \in (Q_G)_1}$, where $x_1$ and $x_2$ are two not necessarily distinct vertices of $Q_G$, i.e. we allow $\gamma$ to be a formal loop. Let $(x_1 \xrightarrow{\gamma} x_2 \xrightarrow{\gamma_2} x_3 \dots x_n \xrightarrow{\gamma_n} x_1)$ be the unique cycle $\gamma$ belongs to. We define a new quiver $Q_{G^\gamma}$ by putting $$(Q_{G^\gamma})_0 = (Q_G)_0 \sqcup \{a,b\}$$ $$(Q_{G^\gamma})_1 = \bigg((Q_{G^\gamma})_1 \setminus \{\gamma\} \bigg) \sqcup \{(x_1 \xrightarrow{\alpha_1} a), (a \xrightarrow{\alpha_2 } x_2), (x_1 \xrightarrow{\beta_1} b),  (b \xrightarrow{\beta_2} x_2), (a \xrightarrow{\alpha} a), (b \xrightarrow{\beta} b) \}$$

Let $I_{G^\gamma}$ be the ideal of $kQ_{G^\gamma}$ generated by $\alpha_1\alpha_2 - \beta_1\beta_2$, $\alpha_2 \gamma_2 \dots \gamma_n\beta_1$, $\beta_2\gamma_2\dots\gamma_n\alpha_1$, $\alpha - \alpha_2\gamma_2 \dots \gamma_n \alpha_1$, $\beta - \beta_2\gamma_2\dots \gamma_2\beta_1$ and all generators of $I_G$ in which $\gamma$ is replaced by $\alpha_1\alpha_2$.

\begin{figure}[H]
\[\begin{tikzcd}
	& \cdot && \cdot && {} && \cdot & \textcolor{rgb,255:red,25;green,0;blue,250}{a} & \cdot \\
	\cdot & {x_1} && {x_2} & \cdot && \cdot & {x_1} && {x_2} & \cdot \\
	{x_n} &&&& {x_3} && {x_n} && \textcolor{rgb,255:red,42;green,141;blue,27}{b} & {} & {x_3} \\
	\cdot &&&& \cdot && \cdot &&&& \cdot \\
	\\
	&& {{\huge Q_G}} && {} &&& {} & {{\huge Q_{G^\gamma}}}
	\arrow["\gamma"{description}, curve={height=-6pt}, from=2-2, to=2-4]
	\arrow["{\gamma_2}"{description}, curve={height=-6pt}, from=2-4, to=3-5]
	\arrow["{\gamma_n}"{description}, curve={height=-6pt}, from=3-1, to=2-2]
	\arrow[curve={height=-6pt}, from=4-1, to=3-1]
	\arrow[curve={height=-6pt}, from=3-5, to=4-5]
	\arrow[curve={height=-18pt}, dashed, no head, from=4-5, to=4-1]
	\arrow[shift left=1, curve={height=6pt}, from=2-4, to=2-5]
	\arrow[curve={height=6pt}, from=1-4, to=2-4]
	\arrow[curve={height=18pt}, dashed, no head, from=2-5, to=1-4]
	\arrow[shift right=1, curve={height=-6pt}, from=2-2, to=2-1]
	\arrow[curve={height=-6pt}, from=1-2, to=2-2]
	\arrow[curve={height=18pt}, dashed, no head, from=1-2, to=2-1]
	\arrow["{\gamma_n}"{description}, curve={height=-6pt}, from=3-7, to=2-8]
	\arrow["{\beta_1}"', color={rgb,255:red,42;green,141;blue,27}, from=2-8, to=3-9]
	\arrow["{\alpha_1}"{pos=0.8}, color={rgb,255:red,25;green,0;blue,250}, from=2-8, to=1-9]
	\arrow["{\alpha_2}"{pos=0.3}, color={rgb,255:red,25;green,0;blue,250}, from=1-9, to=2-10]
	\arrow["{\beta_2}"', color={rgb,255:red,42;green,141;blue,27}, from=3-9, to=2-10]
	\arrow["{\gamma_2}"{description}, curve={height=-6pt}, from=2-10, to=3-11]
	\arrow[curve={height=6pt}, from=1-10, to=2-10]
	\arrow[curve={height=-6pt}, from=4-7, to=3-7]
	\arrow[curve={height=-6pt}, from=1-8, to=2-8]
	\arrow[shift right=1, curve={height=-6pt}, from=2-8, to=2-7]
	\arrow[curve={height=-18pt}, dashed, no head, from=2-7, to=1-8]
	\arrow[shift left=1, curve={height=6pt}, from=2-10, to=2-11]
	\arrow[curve={height=-18pt}, dashed, no head, from=1-10, to=2-11]
	\arrow[curve={height=-6pt}, from=3-11, to=4-11]
	\arrow[curve={height=-12pt}, dashed, no head, from=4-11, to=4-7]
\end{tikzcd}\]
\caption{Obtaining $Q_{G^\gamma}$ from $Q_G$.}
\centering
\end{figure}

\begin{definition} We say that a finite-dimensional $k$-algebra $A$ is of the {\bf Double Edge type} if it is Morita equivalent to $A_{G^\gamma} = kQ_{G^\gamma}/I_{G^\gamma}$ described above, for some Brauer tree $G$ and some arrow $\gamma$ of the corresponding quiver $Q_G$. \end{definition} 

Now let us turn to the ``Brauer tree side''. We are going to describe the {\bf modified Brauer tree} $G^\gamma$ of the algebra $A_{G^\gamma}$. It is, strictly speaking, no longer a tree, but rather a tree with a bit of an extra structure. $G^\gamma$ is obtained by a slight modification of the Brauer tree $G$, hence the name. More precisely, let $G^\gamma$ be a tree with a cyclic ordering of half-edges obtained from $G$ in the following way. Let $\mathscr{X}$ be the vertex of $G$ corresponding to the only cycle of $Q_G$ containing $\gamma$. Let $(G^\gamma)_0 = G_0 \sqcup \{\mathscr{X}'\}$, $(G^\gamma)_1 = G^1 \sqcup \{(\mathscr{X},\mathscr{X}')\}$.
We label the new edge connecting $\mathscr{X}$ and $\mathscr{X}'$ by two symbols $a, b$ and refer to it as {\bf double}. The cyclic ordering of the half-edges incident to each vertex of $G^\gamma$ is the same as in $G$ with the exception of the new double edge labeled $(a,b)$, which is placed between edges $x_1$ and $x_2$ in the cyclic ordering around $\mathscr{X}$. As it is customary in the case of Brauer graphs, we fix a planar representation of $G^\gamma$ in such a way that the cyclic ordering of the half-edges is in agreement with the clockwise ordering. We remark that the double edge can be interpreted as two regular edges $a$ and $b$ positioned one on top of the other, on two parallel 'sheets' of the picture. It is then natural to think of the vertex $\mathscr{X}$ as having not one, but two cyclic orderings on regular edges incident to it, one on edges $x_1, a, x_2, \dots, x_n$ and the other on edges $x_1, b, x_2, \dots, x_n$. Although these two orderings are the same modulo replacing $a$ with $b$ and vice versa, it might be more convenient to consider them as two independent pieces of data, especially for the purpose of carrying as much of the Brauer graphs language to our constructions as possible.

\begin{figure}[H]
\includegraphics[scale=1.1]{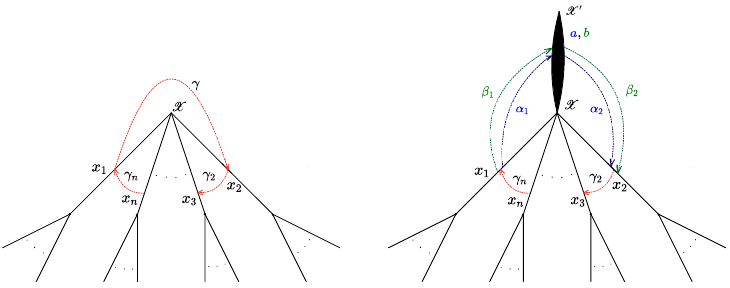}
\caption{Constructing the modified Brauer tree $G^\gamma$ of $A_{G^\gamma}$ (on the right) from $G$ (on the left). }
\centering
\end{figure}

As it is clear from the construction, any tree with a cyclic ordering of half-edges around its vertices and a distinguished (double) pendant edge gives rise to an algebra of the Double Edge tree type. In particular, one can consider algebras whose modified Brauer tree is a star with $m-2$ regular edges and a double edge or a line consisting of $m-2$ regular edges and a double edge at the end. These two examples play a central role in this paper, see Definitions \ref{def:Lambda} and \ref{def:R} for details.

Now we introduce algebras of the {\bf Triple Tree} type. Let $G^i$, $i = 1,2,3$ be three Brauer trees and let $A^i = kQ^i/I^i$ be the corresponding Brauer tree algebras. Denote by $m_i$ the number of edges $|G^i_1|$ in the corresponding tree $G^i$ and assume additionally that $m_1 + m_2 + m_3 = m-3$.
 Since $m \geq 4$, here we assume that at least one of the trees $G^1, G^2, G^3$ has at least one edge.  Below we first consider the case when each of the trees has at least one edge and then remark on how to adapt the construction to the case when one or two of the trees consist of just one vertex.

Pick an arrow in each of the quivers $Q^i$, $(x_1 \xrightarrow{\alpha} x_2) \in Q^1_1, (y_1 \xrightarrow{\beta} y_2) \in Q^2_1, (z_1 \xrightarrow{\gamma} z_2) \in Q^3_1$. Let ${(x_1 \xrightarrow{\alpha} x_2 \xrightarrow{\alpha_2}  \dots \xrightarrow{} x_n \xrightarrow{\alpha_n} x_1)}$, ${(y_1 \xrightarrow{\beta} y_2 \xrightarrow{\beta_2} \dots \xrightarrow{} y_n \xrightarrow{\beta_m} y_1)}$ and  ${(z_1 \xrightarrow{\gamma} z_2 \xrightarrow{\gamma_2} \dots \xrightarrow{} x_n \xrightarrow{\gamma_k} z_1)}$ be the cycles $\alpha, \beta$ and $\gamma$ belong to respectively. As in the previous case, we allow $\alpha, \beta$ and $\gamma$ to be formal loops. We construct a new quiver $Q' = Q'(G_1, G_2, G_3, \alpha, \beta, \gamma)$ in the following way. Let 
$$ Q'_0 = Q^1_0 \sqcup Q^2_0 \sqcup Q^3_0 \sqcup \{a, b, c\}$$
\begin{multline*} Q'_1 =( Q^1_1 \sqcup Q^2_1 \sqcup Q^3_1  \setminus \{\alpha, \beta, \gamma\}) \sqcup \{ (b \xrightarrow{\alpha'} c), (c \xrightarrow{\beta'} a), (a \xrightarrow{\gamma'} b), \\ (x_1 \xrightarrow{\alpha_1'} b), (c \xrightarrow{\alpha_1''} x_2), (y_1\xrightarrow{\beta_1'} c), (a \xrightarrow{\beta_1''} y_2), (z_1 \xrightarrow{\gamma_1'} a), (b \xrightarrow{\gamma_1''} z_2)  \} 
\end{multline*}

Now let $I'$ be the ideal of $kQ'$ generated by: 
\begin{enumerate}
    \item Generators of $I^1$ in which $\alpha$ is replaced with $\alpha_1' \alpha' \alpha_1''$.
    \item Generators of $I^2$ in which $\beta$ is replaced with $\beta_1' \beta' \beta_1''$.
    \item Generators of $I^3$ in which $\gamma$ is replaced with $\gamma_1' \gamma' \gamma_1''$. 
    \item The elements $\alpha'\beta' -  \gamma_1' \gamma_k \dots \gamma_2 \gamma_1''$, $\beta'\gamma' - \alpha_1''\alpha_2 \dots \alpha_n \alpha_1'$, $\gamma'\alpha' - \beta_1''\beta_2\dots \beta_n \beta_1'$ (two paths between any two of the new vertices $a,b,c$ are equal). 
    \item The elements $\alpha_1'\gamma_1''$, $\beta_1'\alpha_1''$, $\gamma_1'\beta_1''$ (there are no non-zero paths going from one of the three Brauer trees to the other).
\end{enumerate}

\begin{figure}[H]\begin{small}
\[\begin{tikzcd}
	&& {{\large z_3}} &&&& {{\large x_n}} \\
	&&& {{\large z_2}} && {{\large x_1}} \\
	{{\large z_k}} &&&& \textcolor{rgb,255:red,240;green,69;blue,66}{{\huge b}} &&&& {{\large x_3}} \\
	& {{\large z_1}} &&&&&& {{\large x_2}} \\
	&& \textcolor{rgb,255:red,240;green,69;blue,66}{{\huge a}} &&&& \textcolor{rgb,255:red,240;green,69;blue,66}{{\huge c}} \\
	&& {{\large y_2}} &&&& {{\large y_1}} \\
	&& {{\large y_3}} &&&& {{\large y_m}}
	\arrow["{{\huge \gamma'}}"{description}, color={rgb,255:red,240;green,69;blue,66}, from=5-3, to=3-5]
	\arrow["{{\huge \beta'}}"{description}, color={rgb,255:red,240;green,69;blue,66}, from=5-7, to=5-3]
	\arrow["{{\huge \alpha'}}"{description}, color={rgb,255:red,240;green,69;blue,66}, from=3-5, to=5-7]
	\arrow["{{\huge \alpha_1'}}", color={rgb,255:red,240;green,69;blue,66}, from=2-6, to=3-5]
	\arrow["{{\huge \alpha_1''}}", color={rgb,255:red,240;green,69;blue,66}, from=5-7, to=4-8]
	\arrow[dashed,"{{\large \alpha}}", "\shortmid"{marking, text={rgb,255:red,214;green,153;blue,92}}, color={rgb,255:red,214;green,153;blue,92}, from=2-6, to=4-8]
	\arrow["{{\large \alpha_n}}"', from=1-7, to=2-6]
	\arrow["{{\large \alpha_2}}"', from=4-8, to=3-9]
	\arrow[curve={height=30pt}, dashed, from=3-9, to=1-7]
	\arrow["{{\huge \gamma_1''}}"', color={rgb,255:red,240;green,69;blue,66}, from=3-5, to=2-4]
	\arrow["{{\huge \gamma_1'}}"', color={rgb,255:red,240;green,69;blue,66}, from=4-2, to=5-3]
	\arrow[dashed,"{{\large \gamma}}", "\shortmid"{marking, text={rgb,255:red,214;green,153;blue,92}}, color={rgb,255:red,214;green,153;blue,92}, from=4-2, to=2-4]
	\arrow["{{\large \gamma_2}}", from=2-4, to=1-3]
	\arrow["{{\large \gamma_k}}", from=3-1, to=4-2]
	\arrow[curve={height=30pt}, dashed, from=1-3, to=3-1]
	\arrow["{{\huge \beta_1''}}", color={rgb,255:red,240;green,69;blue,66}, from=5-3, to=6-3]
	\arrow["{{\large \beta_2}}", from=6-3, to=7-3]
	\arrow["{{\huge \beta_1'}}", color={rgb,255:red,240;green,69;blue,66}, from=6-7, to=5-7]
	\arrow["{\beta_m}", from=7-7, to=6-7]
	\arrow[dashed,"{{\large \beta}}", "\shortmid"{marking, text={rgb,255:red,214;green,153;blue,92}}, color={rgb,255:red,214;green,153;blue,92}, from=6-7, to=6-3]
	\arrow[curve={height=30pt}, dashed, from=7-3, to=7-7]
\end{tikzcd}\]
\caption{Construction of the quiver $Q'$ from the three quivers $Q^1, Q^2$ and $Q^3$. }
\centering
\end{small}
\end{figure}

\begin{remark} Now assume that, for instance, $G^3_1 = \varnothing$, $G_1^1, G_2^1 \neq \varnothing$. Then we again put $$Q'_0 = Q^1_0 \sqcup Q^2_0 \sqcup Q^3_0 \sqcup \{a, b, c\}$$ $$Q'_1 =( Q^1_1 \sqcup Q^2_1 \sqcup Q^3_1  \setminus \{\alpha, \beta\}) \sqcup \{ (b \xrightarrow{\alpha'} c), (c \xrightarrow{\beta'} a), (a \xrightarrow{\gamma'} b), (x_1 \xrightarrow{\alpha_1'} b), (c \xrightarrow{\alpha_1''} x_2), (y_1\xrightarrow{\beta_1'} c), (a \xrightarrow{\beta_1''} y_2))  \}.$$ 

And the ideal $I'$ is now generated by (1) generators of $I^1$ in which $\alpha$ is replaced with $\alpha_1' \alpha' \alpha_1''$, (2) generators of $I^2$ in which $\beta$ is replaced with $\beta_1' \beta' \beta_1''$, (3) the elements $\beta'\gamma' - \alpha_1''\alpha_2 \dots \alpha_n \alpha_1'$, $\gamma'\alpha' - \beta_1''\beta_2\dots \beta_n \beta_1'$, (4) the element  $\beta_1'\alpha_1''$.  \end{remark}

\begin{definition} We say that a finite-dimensional $k$-algebra $A$ is of the {\bf Triple Tree} type if it is Morita equivalent to $A_{G'} = kQ'/I'$ described above, for some Brauer trees $G^1, G^2$ and $G^3$ and some arrows $\alpha, \beta$ and $\gamma$ in $Q^1, Q^2$ and $Q^3$ respectively.  \end{definition}

\begin{remark} Note that the order of the trees $G^1,G^2$ and $G^3$ for a Triple Tree algebra is determined only up to a cyclic permutation.  \end{remark}

Now we describe the modified Brauer tree of $A_{G'}$. By $\mathscr{X}, \mathscr{Y}$ and $\mathscr{Z}$ respectively we denote the vertices of $G^1, G^2$ and $G^3$ respectively corresponding to the cycles containting $\alpha, \beta$, and $\gamma$. As the reader might have already guessed, this time we are going to, vaguely speaking, glue the three Brauer trees $G^1, G^2$ and $G^3$ together, same as the corresponding quivers. More precisely, let $G'$ be the following unoriented graph: 
$$G'_0 = G^1_0 \sqcup G^2_0 \sqcup G^3_0$$
$$G_1' = G^1_1 \sqcup G^2_1 \sqcup G^3_1 \sqcup \{(\mathscr{Z}, \mathscr{Y}), (\mathscr{X}, \mathscr{Z}), (\mathscr{Y}, \mathscr{X})\}$$

We label the new edges $(\mathscr{Z}, \mathscr{Y}), (\mathscr{X}, \mathscr{Z}), (\mathscr{Y}, \mathscr{X})$ by $a$, $b$ and $c$ respectively. We fix a cyclic ordering of half-edges indecent to each vertex of $G$ in such a way that it corresponds to the ordering of the arrows in the cycles of $Q'$. Namely, for all vertices except for $\mathscr{X}, \mathscr{Y}$ and $\mathscr{Z}$, it is the same as in the corresponding Brauer trees $G^1$, $G^2$ and $G^3$. For $\mathscr{X}$ it is $x_1, b, c, x_2, \dots, x_n, x_1$, for $\mathscr{Y}$ it is $y_1, c, a, y_2, \dots, y_m$  and, finally, for $\mathscr{Z}$ it is $z_1, a, b, z_2, \dots, z_k, z_1$. Again we can fix an embedding of $G'$ into the plane in such a way that this ordering is given by the clockwise orientation. We will say that edges $a, b$ and $c$ form {\bf the central triangle} and refer to edges $a,b$ and $c$ themselves as {\bf central edges} (see the left part of Figure \ref{fig:3trees} below). 

The graph $G'$ together with the cyclic ordering of half-edges already carries enough information to reconstruct the quiver $Q'$ and the algebra $A_{Q'}$. On the other hand, one can see that there is a cycle in the quiver $Q'$ that does not correspond to any vertex of the graph $G'$, as it would be the case for Brauer graphs. One way to approach this inconvenience is to consider $G'$ as a 2-dimensional simplicial complex, rather than a graph, i.e. a 1-dimension simplicial complex. More precisely, replace the edges (one-dimensional simplicies) $a$, $b$, and $c$ by {\it two-dimensional} simplicies connecting their ends $\mathscr{X}, \mathscr{Y}, \mathscr{Z}$ to an imaginary vertex in the middle of the cycle formed by them, as shown in Figure \ref{fig:3trees} below on the right. Note that the cyclic ordering of 2-edges incident to the new imaginary vertex should be the opposite to the ordering induced by the clockwise orientation of the plane. 

\begin{figure}[H]
\includegraphics[scale=1.3]{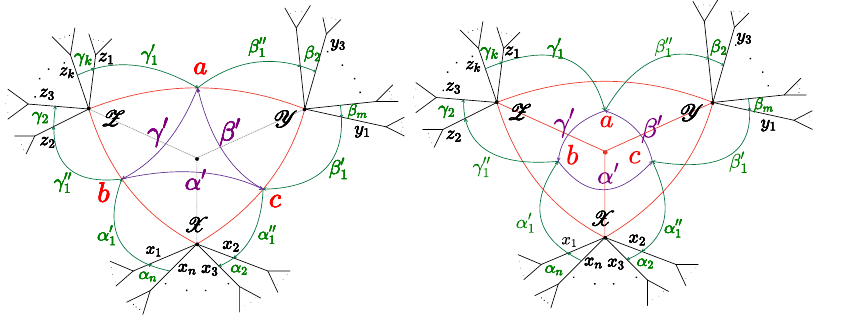}
\caption{The modified Brauer tree $G'$ of $A_{Q'}$ (on the left) and an alternative way to depict it using 2-cells (on the right). }
\label{fig:3trees}
\centering
\end{figure}

%\begin{figure}[H]
%\caption{2-edges $a, b, c$ form the central triangle.}
%\includegraphics[scale=0.14]{5}
%\centering
%\end{figure}

\subsection{Mutations}

The goal of this part of the paper is to investigate the structure of the graph of all silting mutations of symmetric representation-finite algebras of type $D$. With this end in view, we will show how these mutations act on modified Brauer trees which we introduced in the previous subsection. 

Let $A = A_G = kQ_G/I_G$ be one of the algebras described in the previous subsection and let $G$ be the corresponding modified Brauer tree (either a Brauer tree with a double edge or a Triple Tree configuration). From now on we will always labelthe vertices of the quiver $Q$ with natural numbers from $1$ to $m$. Additionally, when describing mutations of Double Edge trees in this section, we will always start with a quiver whose two vertices corresponding to the double edge are labeled with $1$ and $2$. For $j \in (Q_G)_0$ let $P_j = e_jA$ be the right indecomposable projective $A$-module corresponding to the vertex $j$. Note that, as discussed earlier, here it is convenient to consider the double edge as two parallel regular edges. As always, the edge of $G$ (or a ``half'' of the double edge) corresponding to the vertex $j$ will be denoted by $j$ as well. Consider $A = \bigoplus_{i=1}^m P_i$ as a tilting complex over itself. Here and throughout the paper by $soc \colon P_i \xrightarrow{soc} P_i$ we denote the multiplication by the unique socle path in $e_iAe_i$. Note that this will always be the unique non-trivial cycle containing the vertex $i$ in the quiver. Recall that by $\mu_j^+(A)$ we denote the left mutation of $A$ with respect to $P_j$. The algebra $\End_{D^b(A)}(\mu_j^+(A))$ is derived equivalent to $A$, and we will see that it also belongs to one of the classes we introduced in the previous subsection, in particular, up to Morita equivalence it can be encoded by a modified Brauer tree of one of the two types. In this section we will derive simple rules for describing the algebra $\End_{D^b(A)}(\mu_j^+(A))$ using the language of modified Brauer trees. There are a few cases to consider, but in all of them the proof follows the same scheme. In each of the cases we first compute the tilting complex $\mu_j^+(A) = P_j' \oplus \bigoplus_{i\neq j} P_i$ and introduce an algebra $A' = kQ'/{I'} \in \mathfrak{C}(m)$. In each case we will define $A'$ via a quiver with relations or, equivalently, a modified Brauer tree. The algebra $A'$ will be a candidate for $\End_{D^b(A)}(\mu_j^+(A))$. After that we define a homomorphism $\varphi$ from $A'$ to $\End_{D^b(A)}(\mu_j^+(A))$, specifying the images of the primitive idempotents and the arrows of $Q'$. The goal is to show that $\varphi$ is an isomorphism. In order to show that $\varphi$ is injective, we easily observe that no socle element of $A'$ is mapped to $0$. Finally, we prove that $\varphi$ is surjective by showing that the dimensions of $A'$ and $\End_{D^b(A)}(\mu_j^+(A))$ coincide. To do so, we in turn show that the dimension of the $\Hom_{D^b(A)}$-space between any two indecomposable direct summands of $\mu_j^+(A)$ is the same as the number of linearly independent paths between the corresponding vertices of the quiver $Q'$. For the last step we make use the following formula: 

\begin{theorem}[Happel \cite{Hap}, sections III.1.3-4]\label{thm:happel} Let $T,T'$ be two bounded complexes of projective $A-$modules such that $\Hom_{\K^b(\proj-A)}(T,T'[i]) = 0$ for all $i \neq 0$ (for instance, $T$ and $T'$ are direct summands of a tilting complex). Then 

$$\dim_{\kk} \Hom_{\D^b(A)}(T,T') = \dim_{\kk} \Hom_{\K^b(\proj-A)}(T,T') = \sum_{i,j \in \Z} (-1)^{i-j} \dim_{\kk} \Hom_A(T_i, T'_j). $$

\end{theorem}.

We will show how to implement this scheme for one nontrivial case below, whereas for the remaining 4 cases we are only going to carefully define the algebra $A'$ and the map $\varphi$. To check that $\varphi$ is indeed an isomorphism, following the plan we have just described, is routine.  

\begin{enumerate}
\item {\bf Mutations inside a Brauer tree.} If $j$ is a regular edge such that both edges immediately following $j$ in the cyclic ordering are also regular, it is easy to see that the modified Brauer tree of $\End_{D^b(A)}(\mu_j^+(A))$ is obtained from the modified Brauer tree of $A$ by applying the ordinary Kauer move at $j$ (i.e. the two ends of the edge $j$ slide along respectively the two edges immediately following $j$ in the cyclic ordering). 

For convenience, here we remind the reader what happens in the case when $j$ is not a pendant edge. Let $k$ and $l$ be the two edges following $j$ in the cyclic ordering, as Figure below shows on the left. Then $\mu_j^+ = \bigoplus_{i\neq j}P_i \oplus P_j'$, where $P_j' = (P_j \xrightarrow{\begin{pmatrix} \gamma_l \\ \delta_k \end{pmatrix} }P_l \oplus P_k)[1]$ and $\delta_k, \gamma_l$ are as Figure \ref{fig:brauer_quiv} shows (on the left). Let $A'$ be the algebra associated to the Brauer tree that Figure shows on the right. The relevant part of the quiver $Q'$ of $A'$ is shown in Figure \ref{fig:brauer_quiv}  on the right. 

Now we define a homomorphism $\varphi \colon A' \xrightarrow{} \End_{D^b(A)}(\mu_j^+(A))$ by setting the images of the idempotents and the arrows of $Q'$ in the following way. We send every vertex $i$ of the quiver $Q'$ to the endomorphism of $\mu_j^+(A)$ induced by the corresponding identity morphism of the corresponding summand, i.e. either the morphism $P_i \xrightarrow{id} P_i$ if $i \neq j$ or the following morphism if $i = 1$:  
\begin{center}
\begin{tikzcd}
P_j \arrow[r, "\begin{pmatrix} \gamma_l \\ \delta_k \end{pmatrix}"] \arrow[d, "id"]  &  P_l \oplus P_k \arrow[d, "\begin{pmatrix} id \,\,\,\, 0 \\ 0  \,\,\,\, id \end{pmatrix}"] \\ P_j \arrow[r, "\begin{pmatrix} \gamma_l \\ \delta_k \end{pmatrix}"] &  P_l \oplus P_k   \\
\end{tikzcd} 
\end{center} 
The right part of Figure \ref{fig:brauer_quiv} shows where the arrows of $Q'$ are sent under $\varphi$. Namely, an arrow $(a \xrightarrow{} b)$ of $Q'$ for which there exists an arrow $(a \xrightarrow{\theta} b)$ in $Q$ is mapped to the endomorphism of $\mu_j^+(A)$ induced by $P_a \xrightarrow{\theta} P_b$. Otherwise an arrow $(a \xrightarrow{} b)$ is sent to the endomorphism induced by the morphism written on it in the picture. For example, $(j \xrightarrow{} d)$ is sent to $(0, \begin{pmatrix} \delta_d & 0 \end{pmatrix})$, the morphism from $P_j'$ to $P_d$ which is $0$ in degree $-1$ and $\delta_d$ in degree $0$. It is not difficult to see that $\varphi$ is indeed an isomorphism. We leave this to the reader and refer to \cite{Kauer}, \cite{Aih2}, \cite{Aih3} for further details on this classical case.

 \begin{figure}[H]
  \includegraphics[scale=1.2]{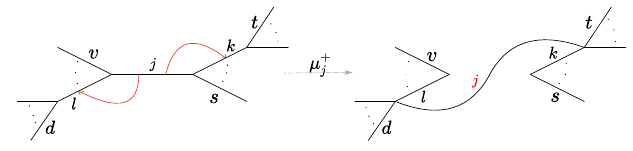}
  \caption{The Brauer tree $G$ of $A$ (on the left) and the Brauer tree $G'$ of ${A' \cong \End_{D^b(A)}(\mu_1^+(A))}$ (on the right).}
  \label{fig:brauer_tree}
\end{figure}

\begin{figure}[H]
    \centering
\includegraphics[scale=1.0]{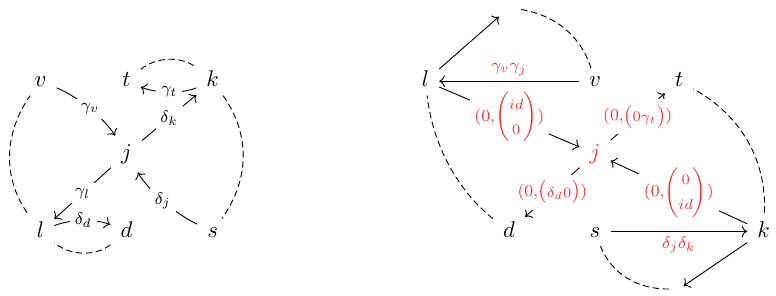}
 \caption{\small{Relevant parts of the quivers of $A$ (on the left) and $A' \cong \End_{D^b(A)}(\mu_j^+(A))$ (on the right)}}
    \label{fig:brauer_quiv}
\end{figure}

\begin{remark} Note that there is a particular labeling of the edges of $G'$ naturally arising from the mutation $\mu_j^+$ and the labeling of the edges of $G$. Namely, the only edge that changes its position is labeled with $j$ and the labels of all ``old'' edges are exactly the same as in $G$. This will be the case for all types of mutations of modified Brauer trees. \end{remark}

\item {\bf Mutating Double Edge trees at the double edge.}
Let $A = A_G$ be a Double Edge algebra. The cases $j = 1$ and $j = 2$ are identical, so we can assume that $j=1$ without loss of generality. For convenience we will assume that the regular edges of $G$ are labeled using deep-first search starting from the edge which is immediately followed by the double edge in the cyclic ordering, i.e. this edge is labeled with $3$. Let $k$ be the edge immediately following the double edge. The relevant parts of the quiver $Q_G = Q$ and the Double Edge tree $G$ are shown on the left in Figures \ref{fig:mut1} and \ref{fig:mut2} respectively. Clearly, we have $P_1' = (P_1 \xrightarrow{\beta_1} P_k)[1]$, where $\mu^+_j(A) = P_1' \oplus \bigoplus_{i=2}^m P_i$. 

Now let $A'$ be the Triple Tree algebra corresponding to the Triple Tree $G'$ shown below on the right in Figure \ref{fig:mut1}. Same as for Kauer moves, only one edge changes its position when going from $G$ to $G'$. Namely, the edge $1$, one ``half'' of the double edge, slides along the edge $k$ with one of its ends and forms the central triangle together with the edges $2$ and $k$. The graph $G'$ automatically defines a quiver $Q'$ and an ideal of relations $I'$, as described in the previous subsection. The relevant part of the quiver $Q'$ of $A'$ is shown in Figure \ref{fig:mut2} (on the right). Same as in the previous case, we now define a homomorphism $\varphi \colon A' \xrightarrow{} \End_{D^b(A)}(\mu_j^+(A))$ by setting the images of idempotents and arrows of $Q'$. Again, we send every vertex $i$ of the quiver $Q'$ to the endomorphism of $\mu_j^+(A)$ induced by the corresponding id morphism at this vertex, i.e. either the morphism $P_i \xrightarrow{id} P_i$ if $i \neq 1$ or $(id, id): P_1' \to P_1'$ if $i=1$. The right part of Figure \ref{fig:mut2} shows where the arrows of $Q'$ are sent under $\varphi$. Namely, an arrow $(a \xrightarrow{} b)$ of $Q'$ for which there exists an arrow $(a \xrightarrow{\theta} b)$ in $Q$ is mapped to the endomorphism of $\mu_j^+(A)$ induced by $P_a \xrightarrow{\theta} P_b$. Otherwise an arrow $(a \xrightarrow{} b)$ is sent to the endomorphism induced by the morphism written on it in the picture, e.g. $(1 \xrightarrow{} 2)$ is sent to $(0, \gamma_d \dots, \gamma_3 \alpha_2)$, $(1 \xrightarrow{} v)$ is sent to $(0, \gamma_v)$, etc. 
  
 \begin{figure}[h!]
 \includegraphics[scale=1.1]{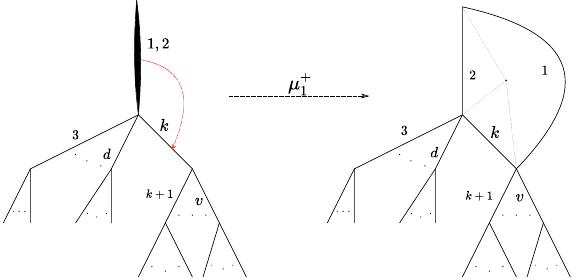}
  \caption{The modified Brauer tree $G$ of $A$ (on the left) and the modified Brauer tree $G'$ of ${A' \cong \End_{D^b(A)}(\mu_j^+(A))}$ (on the right).}
  \label{fig:mut1}
\end{figure}

 \begin{figure}[h!]
     \begin{tiny}
  \hspace{2mm}\[\begin{tikzcd}
	&&&&&&&&&&&&& \textcolor{rgb,255:red,255;green,51;blue,68}{{\large 1}} &&& {} \\
	&&& {{\large 1}} &&&&&&&&&&&& {{\large v}} \\
	\bullet &&& {{\large 2}} &&&& {{\large v}} && \bullet &&& 2 &&& {{\large k+1}} \\
	\bullet & {{\large 3}} &&&& {{\large k}} && {{\large k+1}} & {} & \bullet & {{\large 3}} &&& {} & {{\large k}} &&&&&& {} \\
	&& \bullet & {{\large d}} &&&&&&&&& \bullet & {{\large d}} &&&& {} &&&& {} \\
	&&& \bullet & \bullet &&&&&&&&& \bullet & \bullet
	\arrow["{{\huge \alpha_2}}"{description}, from=4-2, to=3-4]
	\arrow["{{\huge \beta_2}}"{description}, from=3-4, to=4-6]
	\arrow["{{\huge \beta_1}}"{description}, from=2-4, to=4-6]
	\arrow["{{\huge \alpha_1}}"{description}, from=4-2, to=2-4]
	\arrow[curve={height=-6pt}, from=3-1, to=4-2]
	\arrow[curve={height=-6pt}, from=4-2, to=4-1]
	\arrow[curve={height=-6pt}, dashed, no head, from=4-1, to=3-1]
	\arrow["{{\huge \gamma_3}}"{description}, curve={height=-6pt}, from=5-3, to=4-2]
	\arrow["{{\huge \gamma_k}}"{description}, curve={height=-12pt}, from=4-8, to=4-6]
	\arrow["{{\huge \gamma_v}}"{description}, curve={height=-6pt}, from=4-6, to=3-8]
	\arrow[curve={height=-12pt}, dashed, no head, from=3-8, to=4-8]
	\arrow[dashed, no head, from=5-4, to=5-3]
	\arrow["{{\huge \gamma_d}}"{description}, curve={height=-6pt}, from=4-6, to=5-4]
	\arrow[curve={height=6pt}, shorten >=5pt, from=6-5, to=5-4]
	\arrow[curve={height=6pt}, from=5-4, to=6-4]
	\arrow[curve={height=6pt}, dashed, no head, from=6-4, to=6-5]
	\arrow[curve={height=-6pt}, dashed, no head, from=4-10, to=3-10]
	\arrow[curve={height=-6pt}, from=4-11, to=4-10]
	\arrow[curve={height=-6pt}, from=3-10, to=4-11]
	\arrow["{{\huge \gamma_3}}"{description}, curve={height=-6pt}, from=5-13, to=4-11]
	\arrow[dashed, no head, from=5-14, to=5-13]
	\arrow[curve={height=6pt}, from=5-14, to=6-14]
	\arrow[curve={height=6pt}, dashed, no head, from=6-14, to=6-15]
	\arrow[curve={height=6pt}, from=6-15, to=5-14]
	\arrow["{{\huge \alpha_2}}"{description}, from=4-11, to=3-13]
	\arrow["{{\huge (0,\gamma_d\dots\gamma_3\alpha_2)}}"{description, text={rgb,255:red,255;green,51;blue,68}}, curve={height=6pt}, from=1-14, to=3-13]
	\arrow["{{\huge \beta_2}}"{description}, from=3-13, to=4-15]
	\arrow["{{\huge \gamma_d}}"{description}, curve={height=-6pt}, from=4-15, to=5-14]
	\arrow["{{\huge (0, id)}}"{description, text={rgb,255:red,255;green,51;blue,68}}, curve={height=12pt}, from=4-15, to=1-14]
	\arrow["{{\huge \gamma_k}}"{description}, curve={height=-6pt}, from=3-16, to=4-15]
	\arrow["{{\huge (0,\gamma_v)}}"{description, text={rgb,255:red,255;green,51;blue,68}}, curve={height=-6pt}, from=1-14, to=2-16]
	\arrow[curve={height=-6pt}, dashed, no head, from=2-16, to=3-16]
\end{tikzcd}\]\end{tiny}
  \caption{\small{Relevant parts of the quivers of $A$ (on the left) and $A' \cong \End_{D^b(A)}(\mu_j^+(A))$ (on the right)}}
  \label{fig:mut2}
\end{figure}
 
 It is trivial to check that $\varphi$ is indeed a homomorphism and that it does not map socle elements to $0$, hence it is injective. What remains to show is that $\varphi$ is surjective. Again, it is trivial to check that $\dim_\kk\Hom_{D^b(A)}(P_i,P_t)$ is equal to the number of linearly independent paths from $i$ to $t$ in $Q'$, where $i,t\neq 1$. When the remaining summand $P_1'$ is involved, Theorem \ref{thm:happel} gives
 $$\dim_\kk \Hom_{D^b(A)} (P_1', P_t) = \dim_\kk \Hom_{A} (P_k, P_t) - \dim_\kk \Hom_{A}(P_1,P_t), \quad t \neq 1 $$
 When both summands on the right-hand side of the formula vanish, the required assertion is clear. Now if $t$ belongs to the cycle $(k,d,\dots,3,2)$, $t \neq k$, then both summands on the right-hand side are $1$ and indeed there are no non-zero paths from $1$ to $t$ in $Q'$. Next let $t$ be a vertex in the cycle $(v, \dots, k+1,k)$ such that $t \neq k$. Then $\dim_\kk \Hom_{A}(P_k, P_t) = 1$, $\dim_\kk \Hom_{A}(P_1,P_t) = 0$ and there is indeed exactly one non-zero path from $1$ to $t$ in $Q'$. If $t = k$, we have $\dim_\kk \Hom_{A} (P_k, P_t) = 2$, $\dim_\kk \Hom_{A}(P_1,P_t) = 1$, and the number of non-zero paths from $1$ to $k$ in $Q'$ is indeed one. Analogously, for every $t \neq 1$ we have:
 $$\dim_\kk \Hom_{D^b(A)} (P_t, P_1') = \dim_\kk \Hom_{A} (P_t, P_k) - \dim_\kk \Hom_{A}(P_t,P_1)$$
 As before, we consider the cases when at least one of the summands on the left-hand side is nonzero. If $t$ belongs to the cycle $( v \dots, k+1, k)$, $t \neq k$, we have $\dim_k \Hom_{A}(P_t,P_k) = 1$, $\Hom_{A}(P_t,P_1) = 0$ and, indeed, $\dim_k \Hom_{D^b(A)} (P_t, P_1') = 1$ is the number of independent paths from $t$ to $1$ in $Q'$. Now let $t$ belong to the cycle $(k,d,\dots,3,2)$, $t \neq k$. Then ${\dim_\kk \Hom_{A}(P_t,P_k) = 1}$, ${\dim_\kk \Hom_{A}(P_t,P_1) = 1}$ and ${\dim_k \Hom_{D^b(A)} (P_t, P_1') = 0}$, as desired. If $t = k$, we have $\dim_\kk \Hom_{A}(P_k,P_k) = 2$, $\dim_\kk\Hom_{A}(P_k,P_1) = 1$, and $\dim_\kk \Hom_{D^b(A)} (P_k, P_1') = 1$, same as the number of paths from $k$ to $1$ in $Q'$. It remains to check that $\dim_k \Hom_{D^b(A)} (P_1', P_1') = 2$, but this is again straightforward using Theorem \ref{thm:happel}.  

\item {\bf Mutating Double Edge Brauer trees at an edge followed by the double edge.} Let $A = A_G$ be again a Double Edge algebra and let $j = 3$ be the edge of $G$ followed by the double edge $1,2$. Figure \ref{fig:mut3} illustrates how the Double Edge tree $G'$ of $A'$ is obtained from $G$. Again, same as for mutations of Brauer trees, only one edge changes its position. Namely, as it is to be expected, one end of the edge $3$ slides along the edge $d$, which follows $3$ in the cyclic ordering. The other end of $3$, however, slides ``around'' the double edge, leaving it pendant. The algebra $A'$ is the Double Edge algebra constructed from $G'$. Figure \ref{fig:mut4} shows the relevant part of the quiver of $A$ on the left and the relevant part of the quiver of $A'$ on the right. 
 \begin{figure}[H]
  \includegraphics[scale=1.0]{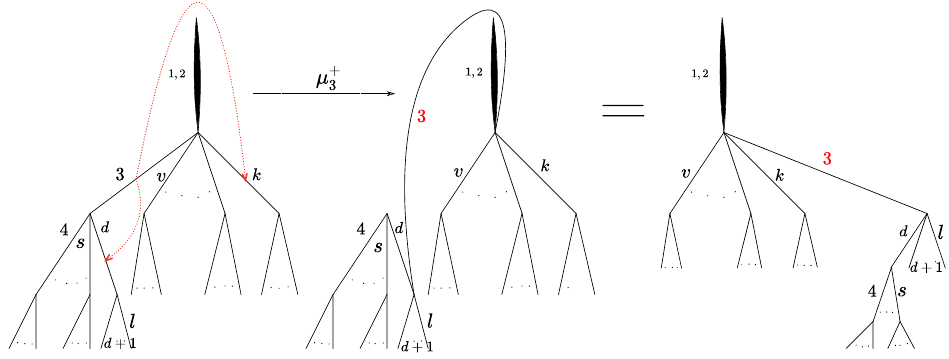}
  \caption{The modified Brauer tree $G$ of $A$ (on the left) and the modified Brauer tree $G'$ of ${A' \cong \End_{D^b(A)}(\mu_3^+(A))}$ (on the right).}
  \label{fig:mut3}
\end{figure}

 \begin{figure}[H]
 \begin{tiny}
\[\begin{tikzcd}
	\cdot & \cdot \\
	s & l &&&&&&&&&& l & \cdot \\
	& d & {d+1} & 1 & \cdot &&& \cdot & \cdot & 1 &&&& {d+1} \\
	4 && 3 & 2 & k & \cdot &&& v & 2 & {} & \textcolor{rgb,255:red,255;green,51;blue,54}{3} \\
	\cdot & v &&&& \cdot && q &&&&&& d \\
	\cdot && q && \cdot &&& \cdot &&& k && s & 4 \\
	&&&&&&&&&&& \cdot & \cdot & \cdot & 5
	\arrow["{\alpha_1}"{description}, from=4-3, to=3-4]
	\arrow["{\alpha_2}"', from=4-3, to=4-4]
	\arrow["{\beta_1}"{description}, from=3-4, to=4-5]
	\arrow["{\beta_2}"', from=4-4, to=4-5]
	\arrow["{\gamma_3}"{description}, curve={height=-6pt}, from=5-2, to=4-3]
	\arrow[curve={height=-6pt}, from=6-3, to=5-2]
	\arrow[curve={height=-6pt}, from=4-5, to=5-6]
	\arrow[curve={height=-6pt}, from=5-6, to=6-5]
	\arrow[curve={height=-12pt}, dashed, no head, from=6-5, to=6-3]
	\arrow[from=4-5, to=3-5]
	\arrow[curve={height=-12pt}, dashed, no head, from=3-5, to=4-6]
	\arrow[from=4-6, to=4-5]
	\arrow["{\delta_d}"{description}, curve={height=6pt}, from=4-3, to=3-2]
	\arrow["{\delta_3}"{description}, curve={height=6pt}, from=4-1, to=4-3]
	\arrow["{\delta_s}"{description}, from=3-2, to=2-1]
	\arrow[curve={height=-12pt}, dashed, no head, from=4-1, to=2-1]
	\arrow[curve={height=6pt}, from=5-2, to=5-1]
	\arrow[curve={height=6pt}, dashed, no head, from=5-1, to=6-1]
	\arrow[curve={height=6pt}, from=6-1, to=5-2]
	\arrow[from=3-3, to=3-2]
	\arrow[curve={height=6pt}, from=2-1, to=1-2]
	\arrow[curve={height=6pt}, from=1-1, to=2-1]
	\arrow[curve={height=6pt}, dashed, no head, from=1-2, to=1-1]
	\arrow["{\gamma_l}"', from=3-2, to=2-2]
	\arrow[curve={height=-6pt}, dashed, no head, from=2-2, to=3-3]
	\arrow["{\gamma_3\alpha_2}"'{text={rgb,255:red,255;green,51;blue,54}}, from=4-9, to=4-10]
	\arrow["{\gamma_3\alpha_1}"{description, text={rgb,255:red,255;green,51;blue,54}}, from=4-9, to=3-10]
	\arrow[curve={height=-6pt}, from=5-8, to=4-9]
	\arrow[curve={height=-6pt}, from=6-8, to=5-8]
	\arrow[curve={height=6pt}, from=4-9, to=3-9]
	\arrow[curve={height=6pt}, dashed, no head, from=3-9, to=3-8]
	\arrow[curve={height=6pt}, from=3-8, to=4-9]
	\arrow[curve={height=-18pt}, dashed, no head, from=6-11, to=6-8]
	\arrow["{{\footnotesize (0, \begin{pmatrix} 0 \\ id \\ 0 \end{pmatrix})}}"{description, text={rgb,255:red,255;green,51;blue,54}}, from=4-10, to=4-12]
	\arrow["{{\footnotesize (0, \begin{pmatrix} id \\ 0 \\ 0 \end{pmatrix})}}"{pos=0.3, text={rgb,255:red,255;green,51;blue,54}}, from=3-10, to=4-12]
	\arrow["{(0, ( -\beta_1 \  \beta_2 \  0 ))}"{description, text={rgb,255:red,255;green,51;blue,54}}, curve={height=-12pt}, from=4-12, to=6-11]
	\arrow["{{\footnotesize (0, \begin{pmatrix}0 \\ 0 \\ id \end{pmatrix})}}"{description, pos=0.4, text={rgb,255:red,255;green,51;blue,54}}, curve={height=-6pt}, from=5-14, to=4-12]
	\arrow["{(0, (0 \ 0 \ \gamma_l))}"'{text={rgb,255:red,255;green,51;blue,54}}, curve={height=-6pt}, from=4-12, to=2-12]
	\arrow[curve={height=-6pt}, from=2-12, to=2-13]
	\arrow[curve={height=-6pt}, from=3-14, to=5-14]
	\arrow[curve={height=-6pt}, dashed, no head, from=2-13, to=3-14]
	\arrow[curve={height=6pt}, from=5-14, to=6-13]
	\arrow[curve={height=-6pt}, dashed, no head, from=6-14, to=6-13]
	\arrow["\delta_3\delta_d"{description, pos=0.4, text={rgb,255:red,255;green,51;blue,54}},curve={height=6pt}, from=6-14, to=5-14]
	\arrow[from=7-12, to=6-13]
	\arrow[from=6-13, to=7-13]
	\arrow[curve={height=-6pt}, dashed, no head, from=7-13, to=7-12]
	\arrow[from=6-14, to=7-14]
	\arrow[from=7-15, to=6-14]
	\arrow[curve={height=6pt}, dashed, no head, from=7-14, to=7-15]
\end{tikzcd}\] \end{tiny}
  \caption{\small{Relevant parts of the quivers of $A$ (on the left) and $A' \cong \End_{D^b(A)}(\mu_3^+(A))$ (on the right)}}
  \label{fig:mut4}
\end{figure}

We have ${\mu_3^+(A) = P_3' \oplus \bigoplus_{i\neq 3}^m P_i}$, where ${P_3' = (P_3 \xrightarrow{\begin{pmatrix} \alpha_1 \\ \alpha_2 \\ \delta_d \end{pmatrix}} P_1 \oplus P_2 \oplus P_d)[1]}$. Now we need to define the map $\varphi \colon A' \xrightarrow{} \End_{D^b(A)}(\mu_3^+(A))$ by setting it on arrows. As before, the images of the arrows are as indicated in Figure \ref{fig:mut4}. Namely, $$\varphi(v \to 1) = \gamma_3\alpha_1, \,\, \varphi(v \to 2) = \gamma_3\alpha_2,  \,\, \varphi(1 \to 3) = (0, {\footnotesize \begin{pmatrix} id \\ 0 \\ 0 \end{pmatrix}}),  \,\, \varphi(2 \to 3) = ( 0,{\footnotesize  \begin{pmatrix} 0 \\ id \\ 0 \end{pmatrix}}), $$ $$ \varphi(3 \to k) = (0, \begin{pmatrix} -\beta_1 & \beta_2 & 0 \end{pmatrix})),  \,\, \varphi(3 \to l) = (0, \begin{pmatrix} 0 & 0 & \gamma_l \end{pmatrix})),  \,\, \varphi(d \to 3) = (0, {\footnotesize \begin{pmatrix} 0 \\ 0 \\ id \end{pmatrix}}).$$
Here by morphisms on the right-hand side we mean the endomorphisms of $\mu_3^+(A)$ induced by them. Unmarked arrows are sent to endomorphisms induced by the corresponding arrows in $Q$. 

\item {\bf Mutating a Triple Tree algebra at a central edge followed by a non-empty tree.} Now let $A=A_G$ be a Triple Tree algebra. The edges forming the central triangle in $G$ are labeled with $1, 2$ and $m$, where $m$ is the number of edges in $G$, as shown in Figure \ref{fig:mut5} on the left. Although right now such a labeling may seem strange, this will be convenient for us later. Inside each of the three trees of $G$ we will adopt the same labeling as for Double Edge Brauer trees in the previous cases, i.e. by the depth-first search (starting from the vertex belonging to the central triangle). Let $j = 1$ and assume additionally that the tree immediately following $1$ in the cyclic ordering around one of its ends is non-empty. Figure \ref{fig:mut5} illustrates how the modified Brauer tree $G'$ of $A'$ is obtained from $G$ in this case. To see why this is in fact natural one should consider all three edges forming a central triangle as regular, one-dimensional edges and mutate at $1$ as if it were a regular edge in a Brauer tree. Then one end of $1$ would slide along $m$, which immediately follows $1$ in the cyclic ordering, and the other end would slide along $d$, which also immediately follows $1$ in the cyclic ordering. This way we again obtain a configuration consisting of three trees attached to the vertices of a 3-cycle (now formed by the edges $1$, $2$ and $d$), hence an algebra $A'$ of the Triple Tree type.

Figure \ref{fig:mut6} shows the relevant part of the quiver of $A$ on the left and the relevant part of the quiver of $A'$ on the right. We have ${\mu_1^+(A) = P_1' \oplus \bigoplus_{i=2}^m P_i}$, where ${P_1' = (P_1 \xrightarrow{\begin{pmatrix} \alpha_m \\ \gamma_d \end{pmatrix}} P_m \oplus P_d )[1]}$. Now we need to define the map ${\varphi \colon A' \xrightarrow{} \End_{D^b(A)}(\mu_1^+(A))}$ by setting it on the arrows of $Q'$. Once again, the images of the arrows are as indicated in Figure \ref{fig:mut6}. Namely, 

\begin{center}
    
$\varphi(q \to m) = \gamma_1\alpha_m, \,\, \varphi(m \to 1) = (0, {\footnotesize \begin{pmatrix} id \\ 0 \end{pmatrix}}),  \,\,  \varphi(d \to 1) = (0, {\footnotesize\begin{pmatrix} 0 \\ id \end{pmatrix}}), $

$\varphi(1 \to 2) = ( 0, \begin{pmatrix} -\alpha_2 & \gamma_a \dots \gamma_2 \end{pmatrix}),  \,\, \varphi(2 \to d) = \alpha_1\gamma_d,  \,\, \varphi(2 \to d) = \alpha_1\gamma_d,  \,\, \varphi(1 \to v) = (0, \begin{pmatrix} 0 & \delta_v \end{pmatrix}).$
\end{center}

Here and in the next cases by morphisms on the right-hand side we mean the endomorphisms of $\mu_j^+(A)$ induced by them. Other arrows are sent to the endomorphisms induced by the corresponding arrows in $Q$. 

 \begin{figure}[H]
  \includegraphics[scale=0.8]{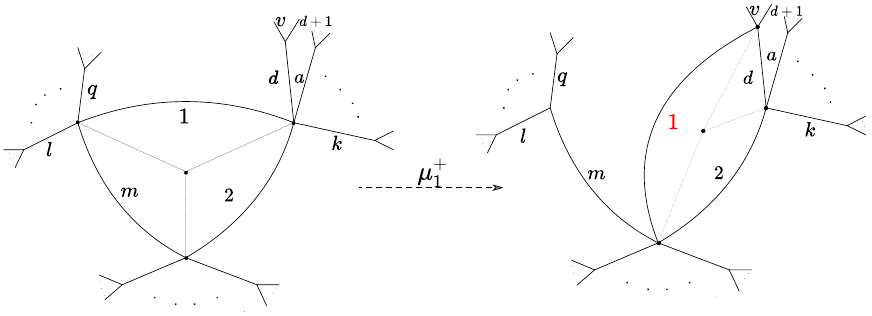}
  \caption{The modified Brauer tree $G$ of $A$ (on the left) and the modified Brauer tree $G'$ of ${A' \cong \End_{D^b(A)}(\mu_1^+(A))}$ (on the right).}
  \label{fig:mut5}
\end{figure}

 \begin{figure}[H]
 \begin{tiny}
\[\begin{tikzcd}
	&&& q && v & {d+1} &&&& q &&& v && {d+1} \\
	{} & l && \textcolor{rgb,255:red,92;green,92;blue,214}{1} && d &&&& l \\
	& \textcolor{rgb,255:red,92;green,92;blue,214}{m} &&&&& \cdot &&&& m &&& \textcolor{rgb,255:red,92;green,92;blue,214}{1} &&& \textcolor{rgb,255:red,92;green,92;blue,214}{d} \\
	\cdot &&&&& a & \cdot &&&& \cdot \\
	&&& \textcolor{rgb,255:red,92;green,92;blue,214}{2} &&&&&&&& \cdot & \textcolor{rgb,255:red,92;green,92;blue,214}{2} &&&& a \\
	&& \cdot && k &&&&&&&&& k &&& \cdot & \cdot \\
	&&& \cdot & \cdot &&&&&&&& \cdot & \cdot
	\arrow["{\alpha_m}"{description}, draw={rgb,255:red,92;green,92;blue,214}, curve={height=12pt}, from=2-4, to=3-2]
	\arrow["{\alpha_2}"{description}, draw={rgb,255:red,92;green,92;blue,214}, curve={height=12pt}, from=3-2, to=5-4]
	\arrow["{\alpha_1}"{description}, draw={rgb,255:red,92;green,92;blue,214}, curve={height=18pt}, from=5-4, to=2-4]
	\arrow["{\gamma_1}"{description}, shift right=1, curve={height=-6pt}, from=1-4, to=2-4]
	\arrow[curve={height=-6pt}, from=4-1, to=3-2]
	\arrow[curve={height=-6pt}, from=5-4, to=6-3]
	\arrow[curve={height=-18pt}, dashed, no head, from=6-3, to=4-1]
	\arrow["{\gamma_2}"{description}, curve={height=-12pt}, from=6-5, to=5-4]
	\arrow[curve={height=6pt}, from=6-5, to=7-4]
	\arrow[curve={height=6pt}, from=7-5, to=6-5]
	\arrow[curve={height=12pt}, dashed, no head, from=7-4, to=7-5]
	\arrow["{\gamma_d}",,curve={height=-12pt}, from=2-4, to=2-6]
	\arrow["{\delta_v}"', curve={height=-6pt}, from=2-6, to=1-6]
	\arrow["{\delta_d}"{description}, curve={height=-6pt}, from=1-7, to=2-6]
	\arrow[curve={height=12pt}, dashed, no head, from=1-7, to=1-6]
	\arrow["{(0, \begin{pmatrix}id \\ 0 \end{pmatrix})}"{description, text={rgb,255:red,214;green,92;blue,92}}, curve={height=-12pt}, from=3-11, to=3-14]
	\arrow["{(0,(-\alpha_2 \ \gamma_a\dots\gamma_2))}"{description, text={rgb,255:red,214;green,92;blue,92}}, draw={rgb,255:red,92;green,92;blue,214}, from=3-14, to=5-13]
	\arrow["{(0, (0 \ \delta_v))}"{description, text={rgb,255:red,214;green,92;blue,92}}, curve={height=-6pt}, from=3-14, to=1-14]
	\arrow["{(0, \begin{pmatrix} 0 \\ id \end{pmatrix})}"{description, text={rgb,255:red,214;green,92;blue,92}}, draw={rgb,255:red,92;green,92;blue,214}, curve={height=12pt}, from=3-17, to=3-14]
	\arrow["{\delta_d}"{description}, curve={height=-6pt}, from=1-16, to=3-17]
	\arrow[curve={height=-12pt}, dashed, from=1-14, to=1-16]
	\arrow["{\alpha_1\gamma_d}"{description, text={rgb,255:red,214;green,92;blue,92}}, draw={rgb,255:red,92;green,92;blue,214}, from=5-13, to=3-17]
	\arrow["{\gamma_2}"{description}, curve={height=-6pt}, from=6-14, to=5-13]
	\arrow["{\gamma_a}"{description}, curve={height=-12pt}, from=3-17, to=5-17]
	\arrow[curve={height=-18pt}, dashed, from=5-17, to=6-14]
	\arrow[curve={height=6pt}, from=6-14, to=7-13]
	\arrow[curve={height=6pt}, dashed, no head, from=7-13, to=7-14]
	\arrow[curve={height=6pt}, from=7-14, to=6-14]
	\arrow["{\gamma_l}"{description}, curve={height=-6pt}, from=3-11, to=2-10]
	\arrow[curve={height=6pt}, from=5-17, to=6-17]
	\arrow[curve={height=6pt}, dashed, no head, from=6-17, to=6-18]
	\arrow[curve={height=6pt}, from=6-18, to=5-17]
	\arrow[curve={height=-12pt}, dashed, no head, from=2-10, to=1-11]
	\arrow["{\gamma_1\alpha_m}"{description, text={rgb,255:red,214;green,92;blue,92}}, curve={height=-12pt}, from=1-11, to=3-11]
	\arrow["{\gamma_l}"{description}, curve={height=-6pt}, from=3-2, to=2-2]
	\arrow[curve={height=-18pt}, dashed, no head, from=2-2, to=1-4]
	\arrow[curve={height=-6pt}, from=4-11, to=3-11]
	\arrow[from=5-13, to=5-12]
	\arrow[curve={height=-6pt}, dashed, no head, from=5-12, to=4-11]
	\arrow["{\gamma_a}"', curve={height=-6pt}, from=2-6, to=4-6]
	\arrow[curve={height=-6pt}, dashed, no head, from=4-6, to=6-5]
	\arrow[curve={height=-6pt}, from=4-7, to=4-6]
	\arrow[curve={height=-6pt}, from=4-6, to=3-7]
	\arrow[shift left=1, curve={height=-6pt}, dashed, no head, from=3-7, to=4-7]
\end{tikzcd}\]\end{tiny}
  \caption{{\small Relevant parts of the quivers of $A$ (on the left) and $A' \cong \End_{D^b(A)}(\mu_1^+(A))$ (on the right)}}
  \label{fig:mut6}
\end{figure}

\item {\bf Mutating a Triple Tree algebra at a central edge followed by an empty tree.} Let $A=A_G$ be a Triple Tree algebra as in the previous case and $j = 1$, but now assume that the tree immediately following the edge $1$ in $G$ in the cyclic ordering {\it is} empty. Figure \ref{fig:mut7} illustrates how the modified Brauer tree $G'$ of $A'$ is obtained from $G$ in this case. Let us first informally explain one way to see what is going on. Again, consider the edges $1$, $2$, $m$ as regular edges, forming a 3-cycle in $G$. Observe that in the quiver $Q$ there is no 2-cycle $(1,2)$, so, in some sense, at least for the purposes of mutations, we can think of the edge $1$ in $G$ as pendant, because the cyclic ordering of the half-edges around one of its ends consists of just one element, namely, $1$. Hence, when mutating at $1$, it is to be expected that the end of $1$ incident to $2$ remains fixed. The other end of $1$ behaves in a usual way, namely, it slides along the edge $m$, which is followed by $1$ in the cyclic ordering. To sum up, the edge $1$ moves to take exactly the same position as the edge $2$, so we end up with a double, ``two-sheeted'' edge $1,2$, exactly as Figure \ref{fig:mut7} depicts. In other words, one could temporarily remove the edge $2$, mutate at $1$ as in an ordinary Brauer tree, and then glue in the edge $2$ back. The graph on the left of Figure \ref{fig:mut7} determines a Double Edge type algebra $A'$.

Figure \ref{fig:mut8} shows the relevant part of the quiver of $A$ on the left and the relevant part of the quiver of $A'$ on the right. We have ${\mu_1^+(A) = P_1' \oplus \bigoplus_{i=2}^m P_i}$, where ${P_1' = (P_1 \xrightarrow{ \alpha_m} P_m)[1]}$. Now we need to define the map ${\varphi \colon A' \xrightarrow{} \End_{D^b(A)}(\mu_1^+(A))}$ by setting it on the arrows of $Q'$. Once again, the images of the arrows are as indicated in Figure \ref{fig:mut8}. Namely, 
\begin{center}
$ \varphi(m \to 1) = (0,id), \varphi(q \to m) = \gamma_1\alpha_m, \varphi(1 \to k) = (0, \alpha_2\gamma_k) $ \end{center}
 
 Other arrows are sent to the endomorphisms induced by the corresponding arrows of $Q$. 

 \begin{figure}[H]
  \includegraphics[scale=0.9]{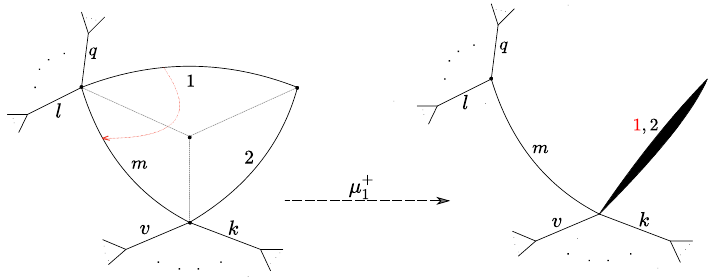}
  \caption{The modified Brauer tree $G$ of $A$ (on the left) and the modified Brauer tree $G'$ of ${A' \cong \End_{D^b(A)}(\mu_1^+(A))}$ (on the right).}
  \label{fig:mut7}
\end{figure}

 \begin{figure}[H]
 \begin{small}
 \[\begin{tikzcd}
	& q \\
	l && \textcolor{rgb,255:red,92;green,92;blue,214}{1} &&& q & 1 \\
	{} & {} &&& l \\
	& \textcolor{rgb,255:red,92;green,92;blue,214}{m} & {} & \textcolor{rgb,255:red,92;green,92;blue,214}{2} && m & 2 & k \\
	& v && k && v & {} & {} \\
	& {} && {}
	\arrow["{\alpha_m}"{description}, color={rgb,255:red,92;green,92;blue,214}, from=2-3, to=4-2]
	\arrow["{\alpha_2}", color={rgb,255:red,92;green,92;blue,214}, curve={height=6pt}, from=4-2, to=4-4]
	\arrow["{\alpha_1}"{description}, color={rgb,255:red,92;green,92;blue,214}, from=4-4, to=2-3]
	\arrow["{\alpha_2}", from=4-6, to=4-7]
	\arrow["{\gamma_k}", from=4-7, to=4-8]
	\arrow["{(0,id)}"{description, text={rgb,255:red,255;green,51;blue,54}}, from=4-6, to=2-7]
	\arrow["{(0,\alpha_2\gamma_k)}"{description, text={rgb,255:red,255;green,51;blue,54}}, from=2-7, to=4-8]
	\arrow["{\gamma_l}"{description}, curve={height=-6pt}, from=4-6, to=3-5]
	\arrow["{\gamma_1\alpha_m}"'{text={rgb,255:red,255;green,51;blue,54}}, curve={height=-6pt}, from=2-6, to=4-6]
	\arrow[curve={height=-12pt}, dashed, no head, from=3-5, to=2-6]
	\arrow["{\gamma_l}"curve={height=-6pt}, from=4-2, to=2-1]
	\arrow[curve={height=-6pt}, dashed, no head, from=2-1, to=1-2]
	\arrow["{\gamma_1}"curve={height=-6pt}, from=1-2, to=2-3]
	\arrow["{\gamma_m}"{description}, curve={height=-6pt}, from=5-2, to=4-2]
	\arrow["{\gamma_k}"{description}, curve={height=-6pt}, from=4-4, to=5-4]
	\arrow[curve={height=-12pt}, dashed, no head, from=5-4, to=5-2]
	\arrow["{\gamma_m}"{description}, curve={height=-6pt}, from=5-6, to=4-6]
	\arrow[curve={height=-6pt}, dashed, no head, from=5-8, to=5-6]
	\arrow[curve={height=-6pt}, from=4-8, to=5-8]
\end{tikzcd}\] \end{small}
  \caption{Relevant parts of the quivers of $A$ (on the left) and $A' \cong \End_{D^b(A)}(\mu_1^+(A))$ (on the right)}
  \label{fig:mut8}
\end{figure}

\item {\bf Mutating a Triple Tree algebra at a regular edge followed by a central edge.} Let $A_G = A$ be a Triple Tree algebra and let $j$ be a regular edge of $G$ immediately followed by the central edge $1$. Let $d$ be the other edge which immediately follows $j$. Figure \ref{fig:mut9} illustrates how the modified Brauer tree  $G'$ of $A'$ is obtained from $G$ in this case. To see why it is natural, just consider the edge $1$ as a regular, one-dimensional edge and perform the ordinary Kauer move at the edge $j$, sliding one of its ends along $d$ and the other along $1$. 

Figure \ref{fig:mut8} shows the relevant part of the quiver of $A$ on the left and the relevant part of the quiver of $A'$ on the right. We have ${\mu_j^+(A) = P_j' \oplus \bigoplus_{i\neq j}^m P_i}$, where ${P_j' = (P_j \xrightarrow{\begin{pmatrix} \gamma_1 \\ \delta_d \end{pmatrix}} P_1 \oplus P_d)[1]}$. Again, we need to define the map ${\varphi \colon A' \xrightarrow{} \End_{D^b(A)}(\mu_j^+(A))}$ by setting it on the arrows of $Q'$. The images of the arrows are as indicated in Figure \ref{fig:mut10}. Namely, 
\begin{center}
$\varphi(k \to 1) = \gamma_j \gamma_1, \varphi(1 \to j) = (0, \begin{pmatrix} id \\ 0 \end{pmatrix}), \varphi(j+1 \to d) = \delta_j\delta_d, $  
$\varphi(j \to p) = (0, (\gamma_p,0)), \varphi(j \to s) = (0, (0,\gamma_s)),  \varphi(d \to j) = (0, \begin{pmatrix} 0 \\ id \end{pmatrix})  $
\end{center}
Other arrows are sent to the endomorphisms induced by the corresponding arrows of $Q$. 

 \begin{figure}[H]
  \includegraphics[scale=0.95]{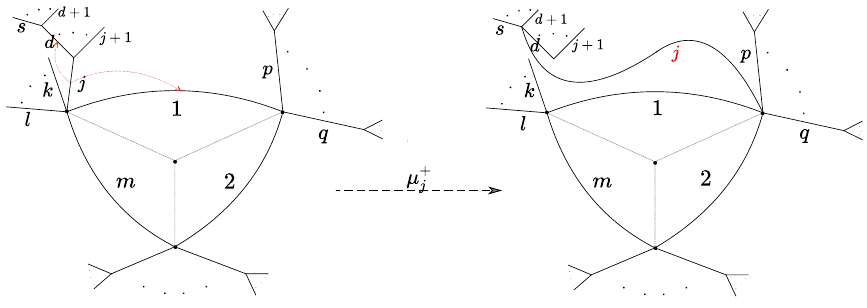}
  \caption{The modified Brauer tree $G$ of $A$ (on the left) and the modified Brauer tree $G'$ of ${A' \cong \End_{D^b(A)}(\mu_j^+(A))}$ (on the right).}
  \label{fig:mut9}
\end{figure}

 \begin{figure}[H]
 \begin{footnotesize}
   
\[\begin{tikzcd}
	&&&&&&&&&& {} \\
	{} & {d+1} &&&&&&&& {d+1} & d & {j+1} \\
	s & d && {j+1} &&&&&& s & j \\
	&& j && p &&&&&&&& p \\
	& k && \textcolor{rgb,255:red,92;green,92;blue,214}{1} &&&&& k && \textcolor{rgb,255:red,92;green,92;blue,214}{1} \\
	& l &&&& q &&& l &&&& q \\
	& \textcolor{rgb,255:red,92;green,92;blue,214}{m} &&&& \textcolor{rgb,255:red,92;green,92;blue,214}{2} &&& \textcolor{rgb,255:red,92;green,92;blue,214}{m} &&&& \textcolor{rgb,255:red,92;green,92;blue,214}{2} \\
	& {} &&&& {} &&& {} &&&& {}
	\arrow["{\gamma_1}"{description}, curve={height=-6pt}, from=4-3, to=5-4]
	\arrow["{\gamma_p}"{description}, curve={height=-6pt}, from=5-4, to=4-5]
	\arrow[curve={height=-6pt}, from=7-6, to=8-6]
	\arrow[curve={height=-6pt}, from=8-2, to=7-2]
	\arrow[curve={height=-12pt}, dashed, from=8-6, to=8-2]
	\arrow[curve={height=12pt}, dashed, no head, from=3-4, to=3-2]
	\arrow["{\delta_d}"{description}, curve={height=-6pt}, from=4-3, to=3-2]
	\arrow["{\delta_j}"{description}, curve={height=-6pt}, from=3-4, to=4-3]
	\arrow["{\gamma_d}"{description}, curve={height=-6pt}, from=2-2, to=3-2]
	\arrow["{\gamma_j}"{description}, curve={height=-6pt}, from=5-2, to=4-3]
	\arrow["{\gamma_l}"{description}, from=7-2, to=6-2]
	\arrow[dashed, no head, from=6-2, to=5-2]
	\arrow[from=6-6, to=7-6]
	\arrow[curve={height=-12pt}, dashed, no head, from=4-5, to=6-6]
	\arrow["{\alpha_m}"{description}, draw={rgb,255:red,92;green,92;blue,214}, from=5-4, to=7-2]
	\arrow["{\alpha_2}"{description}, draw={rgb,255:red,92;green,92;blue,214}, from=7-2, to=7-6]
	\arrow["{\alpha_1}"{description}, draw={rgb,255:red,92;green,92;blue,214}, from=7-6, to=5-4]
	\arrow["{\alpha_2}"{description}, draw={rgb,255:red,92;green,92;blue,214},, from=7-9, to=7-13]
	\arrow["{\alpha_1}"{description}, draw={rgb,255:red,92;green,92;blue,214}, from=7-13, to=5-11]
	\arrow["{\alpha_m}"{description}, draw={rgb,255:red,92;green,92;blue,214}, from=5-11, to=7-9]
	\arrow[from=6-13, to=7-13]
	\arrow["{{\footnotesize (0, \begin{pmatrix} id \\ 0 \end{pmatrix})}}"{description, text={rgb,255:red,255;green,51;blue,58}}, from=5-11, to=3-11]
	\arrow["{(0, (0 \ \gamma_s))}"{text={rgb,255:red,255;green,51;blue,58}}, curve={height=-6pt}, from=3-11, to=3-10]
	\arrow["{{\footnotesize (0, \begin{pmatrix} 0 \\ id \end{pmatrix})}}"'{text={rgb,255:red,255;green,51;blue,58}}, curve={height=-6pt}, from=2-11, to=3-11]
	\arrow[curve={height=-6pt}, from=2-11, to=1-11]
	\arrow["{\delta_j\delta_d}"{description, text={rgb,255:red,255;green,51;blue,58}}, curve={height=-6pt}, from=2-12, to=2-11]
	\arrow[curve={height=-12pt}, dashed, no head, from=1-11, to=2-12]
	\arrow[curve={height=-6pt}, from=2-10, to=2-11]
	\arrow[curve={height=-6pt}, dashed, no head, from=3-10, to=2-10]
	\arrow[curve={height=-6pt}, from=8-9, to=7-9]
	\arrow[curve={height=-6pt}, from=7-13, to=8-13]
	\arrow[curve={height=-12pt}, dashed, no head, from=8-13, to=8-9]
	\arrow["{(0,(\gamma_p \ 0))}"{description, text={rgb,255:red,255;green,51;blue,58}}, from=3-11, to=4-13]
	\arrow[dashed, no head, from=4-13, to=6-13]
	\arrow[dashed, no head, from=6-9, to=5-9]
	\arrow["{\gamma_j\gamma_1}"{description, pos=0.4, text={rgb,255:red,255;green,51;blue,58}}, curve={height=-6pt}, from=5-9, to=5-11]
	\arrow["{\gamma_l}"{description}, from=7-9, to=6-9]
	\arrow["{\gamma_s}"{description}, curve={height=-6pt}, from=3-2, to=3-1]
	\arrow[curve={height=-12pt}, dashed, no head, from=3-1, to=2-2]
\end{tikzcd}\] \end{footnotesize}
  \caption{{\small Relevant parts of the quivers of $A$ (on the left) and $A' \cong \End_{D^b(A)}(\mu_j^+(A))$ (on the right)}}
  \label{fig:mut10}
\end{figure}
\end{enumerate} 

\section{The two algebras \texorpdfstring{$\Lambda$}{Lg} and \texorpdfstring{$\mathcal{R}$}{Lg}}\label{sec:2alg}
In this section we are going to discuss some facts about the two symmetric representation-finite algebras of type $D$ that play a central role in the proof of the main result. Being derived equivalent, these two algebras have isomorphic derived Picard groups, so we will be able to use one or the other depending on what is more convenient for a particular part of the proof. We begin with definitions. Let $m \geq 4$. 

\begin{definition}\label{def:Lambda} Let $Q_m$ be the quiver shown in Figure \ref{fig:algLambda} on the left and let $I_m$ be the two-sided ideal of $\kk Q_m$ generated by $\alpha_1\delta_1 - \alpha_2\delta_2$ and all paths of length $m$. Let $\Lambda_m = \kk Q_m/I_m$. In the terms of the previous section, $\Lambda_m$ is the Double Edge algebra corresponding to the star formed by $m-2$ regular edges and a double edge, as shown in Figure \ref{fig:algLambda} on the right. We will denote such a modified Brauer tree with this particular labeling of edges by $\Gamma_m$. By $P_i = e_i\Lambda_m$ for $i=1, \dots, m$ we denote the indecomposable projective $\Lambda_m-$modules. 
\end{definition}

 \begin{figure}[h!]
%\begin{subfigure}{.5\textwidth}
 % \centering
 % \includegraphics[width=.6\linewidth]{Lambda1.pdf}
%\end{subfigure}%
%\begin{subfigure}{.5\textwidth}
 % \centering
 % \includegraphics[width=.6\linewidth]{Lambda2.pdf}
%\end{subfigure}

\includegraphics[scale=1.05]{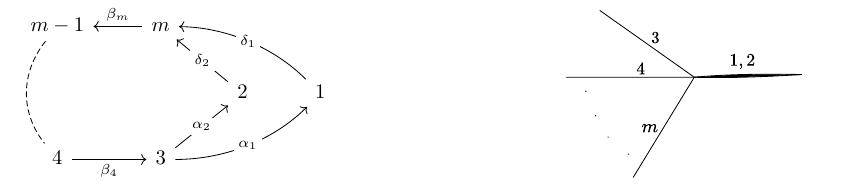}
  \caption{\footnotesize{The quiver $Q_m$ (on the left) and the modified Brauer tree $\Gamma_m$ (on the right) of $\Lambda_m$.}}
  \label{fig:algLambda}
\end{figure}

\begin{definition}\label{def:R} Let $Q_m'$ be the quiver shown in Figure \ref{fig:algR} on the left. Let $I_m'$ be the two-sided ideal of $\kk Q_m'$  generated by (1) all paths of the form $\gamma_i\gamma_j$ and $\gamma_i'\gamma_j'$, (2) elements $\gamma_k'\gamma_k - \gamma_{k-1}\gamma_{k-1}'$ for $4 \leq k \leq m-1$, and, finally, (3) $\gamma_1\gamma_1' - \gamma_3'\gamma_3$ and $\gamma_2\gamma_2' - \gamma_3'\gamma_3$. Let $\mathcal{R}_m = \kk Q_m'/I_m'$. In other words, $\mathcal{R}_m$ is the trivial extension algebra of the $D_m$ Dynkin diagram with alternating orientation. In terms of the previous section, $\mathcal{R}_m$ is the Double Edge algebra corresponding to the line formed by $m-2$ regular edges and a double edge at the end, as Figure \ref{fig:algR} shows on the right. By $R_i = e_i\mathcal{R}$ for $i=1, \dots, m$ we denote the indecomposable projective $\mathcal{R}_m-$modules.
\end{definition}

 \begin{figure}[H]
 \begin{subfigure}{.5\textwidth}
  \centering
  \includegraphics[width=.9\linewidth]{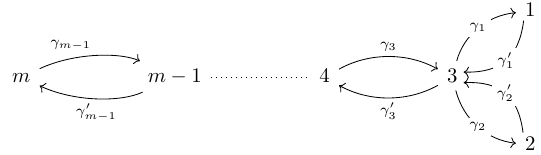}
\end{subfigure}%
\begin{subfigure}{.5\textwidth}
  \centering
  \includegraphics[width=.9\linewidth]{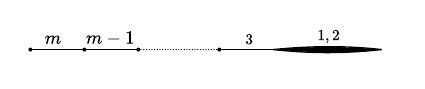}
\end{subfigure}
  \caption{\footnotesize{The quiver $Q_m'$ (on the left) and the modified Brauer tree (on the right) of $\mathcal{R}_m$.}}
  \label{fig:algR}
\end{figure}

Usually we will fix $m \geq 4$ and simply write $\Lambda$ and $\mathcal{R}$ for $\Lambda_m$ and $\mathcal{R}_m$ respectively. However, sometimes the parity of $m$ will be important, in which case we will keep the subscripts.

\begin{lemma}\label{thm:equiv} $\Lambda$ and $\mathcal{R}$ are derived equivalent. 
\end{lemma} 

\begin{proof} We are going to construct a particular derived equivalence between $\Lambda$ and $\mathcal{R}$ that we will use later. Consider $\mathcal{R} = \bigoplus_{i=1}^m R_i$ as a tilting complex over itself and apply a sequence of mutations ${\mu = (\mu_4^-) \circ (\mu_5^-)^2 \circ\dots\circ(\mu_m^-)^{m-3}}$. Let ${U = \mu(\bigoplus_{i=1}^m R_i)}$. In the light of the previous section, it is easy to see that ${\End_{D^b(\mathcal{R})}(U) \cong \Lambda}$ and ${U = \bigoplus_{i=1}^m U_i}$, where 
$$ U_i = \begin{cases} \underline{R_i}, & \quad i = 1,2,3 \\ \underline{R_3} \xrightarrow{\gamma_3'} R_4 \xrightarrow{} \dots \xrightarrow{} R_{i-1} \xrightarrow{\gamma_{i-1}'} R_i, &  \quad 4 \leq i \leq m \end{cases} $$

Here the underlined terms are placed in degree $0$. Let $ \mathcal{P}_{\Lambda}$ be the category of finitely generated projective $\Lambda$-modules. Recall that by $\add U$ we denote the smallest subcategory of $\K^b(\proj-\mathcal{R})$ containing $U$, closed under direct summands, finite direct sums and isomorphisms. Now define an equivalence ${S \colon \mathcal{P}_{\Lambda} \to \add U}$ in the following way 
$$S(P_i) = U_i, \quad 1 \leq i \leq m $$
    $$S(\beta_i) = (Id, \dots, Id, 0)$$ 
    $$S(\alpha_i) = \gamma_i, \quad i = 1,2$$ 
    $$S(\delta_i) = \gamma_i', \quad i=1,2$$
Then using the results of \cite{R}, we can extend $S$ to an equivalence ${F \colon \K^b(\proj-\Lambda) \to \K^b(\proj-\mathcal{R})}$. Note that, given a tilting complex $U$, defining an equivalence between $\K^b(\proj-\Lambda)$ and $\K^b(\proj-\mathcal{R})$ is the same as fixing an isomorphism ${\End_{D^b(\mathcal{R})}(U) \cong \Lambda}$. 
\end{proof} 

\subsection{The Picard group of \texorpdfstring{$\Lambda$}{Lg}} 
Now we are going to describe the classical {\bf Picard group} $\Pic(\Lambda)$ of $\Lambda$, i.e. the group of isomorphism classes of invertible $\Lambda- \Lambda$-bimodules. Equivalently, it can be defined as the group of Morita autoequivalences of $\Lambda$ modulo natural isomorphisms. Since $\Lambda$ is basic, $\Pic(\Lambda)$ coincides with the group of outer automorphisms $\Out(\Lambda) = \Aut(\Lambda)/\Inn(\Lambda)$ (see \cite{B}, Proposition 3.8). The group $\Out(\Lambda)$ can be naturally considered as a subgroup of $\TrPic(\Lambda)$. We will use the same letters to denote an automorphism of an algebra and the induced standard autoequivalence.

\begin{theorem}\label{thm:pic} $$\Pic(\Lambda_m) \cong \Out(\Lambda_m) \cong k^* \times  \Z/2\Z$$
 \end{theorem}
  \begin{proof} Let $S = \langle e_1, e_2, \dots, e_m \rangle_k $ be the linear span of the primitive idempotents of $\Lambda$. Then $\Lambda$ admits a Wedderbern-Malcev decomposition ${\Lambda = S \oplus \Rad(\Lambda)}$, where $\Rad(\Lambda)$ is the Jacobson radical of $\Lambda$. The group $\Out(\Lambda)$ coincides with the group $\Aut^S(\Lambda)/\Aut^S(\Lambda)\cap \Inn(\Lambda)$, where ${\Aut^S(\Lambda) = \{ f \in \Aut(\Lambda) | f(S) \subseteq S \} }$. Let us first describe all elements of $\Aut^S(\Lambda)$. It is easy to see that any $f \in \Aut^S(\Lambda)$ sends primitive idempotents to primitive idempotents. Let $e_{f(i)} = f(e_i)$ be the image of $e_i$ under $f$. Note that $e_{f(i)}\Lambda e_{f(j)} \subset \Rad(\Lambda)^2$ if and only if there are no arrows from $f(i)$ to $f(j)$ in the quiver of $\Lambda$. This clearly implies that $f$ either acts trivially on $S$, fixing all idempotents, or interchanges $e_1$ and $e_2$. Hence for every $f \in \Aut^S(\Lambda)$ one of the following two options holds
  
  \begin{enumerate}
      \item[(1)] $f(\beta_i) = b_i \beta_i$, where $i = 4, \dots, m$, $f(\alpha_i) = a_i \alpha_i$, $f(\delta_i) = c_i \delta_i$, where $i = 1,2$. 
      
      \item[(2)] $f(\beta_i) = b_i \beta_i$, where $i = 4, \dots, m$, $f(\alpha_i) = a_i \alpha_{3-i}, f(\delta_i) = c_i \delta_{3-i}$, where $i = 1,2$
  \end{enumerate}
    
  \noindent for some elements $b_i, a_i, c_i \in k^*$, where $a_1c_1 = a_2c_2$. Clearly, an automorphism is uniquely determined by such data. In other words, elements of $\Aut^S(\Lambda)$ are parameterized by  $\Z/2\Z \times (k^*)^{m}$. Next we shall see which of the values of these parameters correspond to inner automorphisms. Let $g \in \Inn(\Lambda)$ be an inner automorphism and $\lambda \in \Lambda^*$ such that $g(x) = \lambda^{-1}x\lambda$ for every $x \in \Lambda$. For every $\lambda \in \Lambda^*$ we can write $\lambda = \sum_{i=1}^m d_i e_i + r$ for some $r \in \Rad(\Lambda)$ and $d_i \in k^*$. Then, substituting all arrows for $x$, one has 
  $$g(\beta_i) = d_id_{i-1}^{-1}\beta_i \text{ for every } i = 4, \dots, m $$ 
  $$g(\alpha_i) = d_3d_i^{-1} \alpha_i, g(\delta_i) = d_id_m^{-1} \delta_i, \text{ where } i = 1,2$$
  In other words, among automorphisms that preserve $S$, inner automorphisms are precisely those fixing all idempotents and for which $a_1c_1b_mb_{m-1}\dots b_4 = 1$ in terms of the condition $(1)$ above. We shall make use of this criterion several times throughout the proof.

Let $f_a \in \Aut^S(\Lambda)$ be the automorphism that fixes all arrows except for $\alpha_1$ and $\alpha_2$, which are sent to $a \alpha_1$ and $a \alpha_2$ respectively for some $a \in k^*$. Let $\tau^{\Lambda} \in \Aut^S(\Lambda)$ be the automorphism that interchanges $e_1$ and $e_2$ (with all constants $b_i, a_i, c_i$ being equal to 1). We can now define a group homomorphism $F: k^* \times \Z/2\Z \to \Aut^S(\Lambda)/\Aut^S(\Lambda)\cap \Inn(\Lambda)$ by setting ${F((a,0)) = [f_a]}$ and ${F((a,1)) = [\tau^{\Lambda} \circ f_a]}$. It is easy to check that it is indeed a homomorphism. Namely, $\tau^{\Lambda} \circ f_a = f_a \circ \tau^{\Lambda}$ and $f_a \circ f_a' = f_{aa'}$ hold not only modulo inner automorphisms, but even in $\Aut^S(\Lambda)$. Injectivity of $F$ is also clear and follows directly from the criterion for inner automorphisms we have obtained earlier. Indeed, if $F((a,x))$ is inner, we have $x = 0$, because inner automorphisms fix all idempotents, and we are bound to have $a = 1$ by the condition $a_1c_1b_mb_{m-1}\dots b_4 = 1$ satisfied for inner automorphisms. Hence, $(a,x) = (1,0)$.

Thus, it remains it show that $F$ is surjective. Let $f \in \Aut^S(\Lambda)$ be an automorphism. We aim to show that $[f]$ belongs to $\Imm F$. Applying $\tau^{\Lambda}$ if needed, we can assume without loss of generality that $f$ fixes all $e_i$. Let $b_i, a_i, c_i$ be the elements of $k^*$ defining $f$ as in the discussion above (condition (1)). It is sufficient to find an element $a\in k^*$ such that $f_a \circ f$ is an inner automorphism. But this is again clear from the criterion. Indeed, simply set $a = (a_1c_1b_m \dots b_4)^{-1}$ and we are done.
  \end{proof} 

By $\Pic_0(\Lambda)$ we will denote the subgroup of $\TrPic(\Lambda)$ generated by $F \in \TrPic(\Lambda)$ such that $F(P_i) \cong P_i$ in $\K^b(\proj-\Lambda)$ for every $i=1, \dots, m$. One can see that $\Pic_0(\Lambda) \cong \kk^*$. Every element of $\Pic_0(\Lambda)$ is of the form $ _{\varphi}\Lambda \otimes_{\Lambda} -$, where $\varphi$ is an automorphism of $\Lambda$ which fixes every primitive idempotent $e_i$ and $_{\varphi}\Lambda$ denotes the $\Lambda$-bimodule which is regular regarded as a right module and has its left multiplication twisted by $\varphi$, i.e. $l \cdot t = \varphi(l)m$ for any $t \in {}_\varphi \Lambda$, $l \in \Lambda$. 
It is not difficult to see that $\Pic_0(\mathcal{R}_m) \cong \Pic_0(\Lambda_m) \cong k^*$ and, using the same arguments as in the proof of Theorem \ref{thm:pic}, $\Pic(\mathcal{R}_m) \cong \Pic(\Lambda_m) \cong \kk^* \times \Z/2\Z$ if $m>4$. However, this is not true for $\Pic(\mathcal{R}_4)$. One can show that $\Pic(\mathcal{R}_4) \cong \kk^* \times S_3$. This last fact will not be used in the proof of the main result.

\subsection{Spherical twists on \texorpdfstring{$\Lambda$}{Lg}} 
Note that the indecomposable projective modules $R_i = e_i\mathcal{R}$ over $\mathcal{R}$ are 0-spherical objects of $\D^b(\mathcal{R})$ (see \cite{ST} for the original definition and \cite{NV} for more details on this case). Namely, $\End^*_{D^b(\mathcal{R})}(R_i) \cong \kk[t]/(t^2)$ as graded $\kk-$algebras, where $\deg(t) = 0$, and $R_i$ is 0-Calabi-Yau for every $i=1,\dots,m$. Since $R_i$ is projective, the first condition simply means that $\End_{\mathcal{R}}(R_i) \cong \kk[t]/(t^2)$, which can be easily observed from the quiver of $\mathcal{R}$. The second condition means that there is a functorial isomorphism ${\Hom_{\D^b(\mathcal{R})}^*(R_i, -) \cong \Hom^*_{\D^b(\mathcal{R})}(-,R_i)^\vee}$, where $(-)^\vee$ denotes the $\kk$-linear dual. This follows from the fact that $\mathcal{R}$ is symmetric (see, for instance, \cite{R3}). Hence, to every $R_i$, as to any spherical object of a derived category, we can associate an autoequivalence of a particular form, namely, the spherical twist along $R_i$.

\begin{definition} Let $T_i \colon \D^b(\mathcal{R}) \to \D^b(\mathcal{R})$ be the {\bf spherical twist} functor associated to the spherical object $R_i = e_i\mathcal{R}$. In other words, $T_i$ is the autoeqivalence $\D^b(\mathcal{R})$ given by the formula 
$$T_i(X) = cone(R_i \otimes \RHom(R_i,X) \xrightarrow{ev} X).$$ 

 Here by $\RHom(X,-)$ we denote the standard Hom-complex functor taking values in the derived category of $\kk-$vector spaces and by $- \otimes X$ its left adjoint. 

\end{definition}

\begin{remark}
  In general, to define spherical twists as actual functors of triangulated categories, a dg-enhancement \cite{BK} is required (for details see, for instance, \cite{AL}). Here we use the fact that the bounded derived category of an abelian category always has a dg-enhancement.  
Equivalently, in our case we can define $T_i = X_i \otimes_{\mathcal{R}} -$, where $X_i = cone (R_i \otimes_{\kk} \Hom_\kk(R_i,k) \xrightarrow{m} \mathcal{R})$, ${m(ae_i\otimes e_ib) = ab}$. One can show that $X_i$ is a two-sided tilting complex, thus, $T_i$ defines an element of $\TrPic(\mathcal{R})$. 
\end{remark}

Let $F \colon \K^b(\proj-\Lambda) \to \K^b(\proj-\mathcal{R})$ be the equivalence we constructed in the proof of Lemma \ref{thm:equiv}. Let $\overline{F}$ be an equivalence quasi-inverse to $F$. Then we can define an isomorphism of groups ${F^* \colon \TrPic(\mathcal{R}) \to \TrPic(\Lambda)}$ sending $\varphi \in \TrPic(\mathcal{R})$ to $\overline{F} \circ \varphi \circ F  \in \TrPic(\Lambda)$. For $i=1, \dots, m$, we will denote the autoequivalences $F^*(T_i)$ by $t_i$ and refer to them as {\bf spherical twists on $\Lambda$}. 
\begin{lemma}\label{thm:twonLambda} The spherical twists $t_i$ on $\Lambda$ act on indecomposable projective $\Lambda$-modules in the following way. If $i \geq 4$, 
$$t_i(P_j) =
\begin{cases}
P_j, &\text{ if } j \neq i, i-1\\
P_{j-1},  &\text{ if } j=i\\
(P_i \xrightarrow{\beta_i} P_{i-1} \xrightarrow{soc} P_{i-1})[2],  &\text{ if } j = i-1
\end{cases}
$$ 
If $i=1,2$:

$$
t_i(P_j) =
\begin{cases}
P_j, & \text{ if } j = 3-i\\
P_j[1], & \text{ if } j=i\\
(P_i \xrightarrow{\delta_i\beta_m \dots \beta_{j+1}} P_j)[1], & \text{ if } j \geq 3
\end{cases}
$$ 

Finally, 

$$
t_3(P_j) =
\begin{cases}
(P_3 \xrightarrow{\alpha_1\delta_1\beta_m \dots \beta_{j+1}} P_j)[1], & \text{ if } j \neq 3\\
P_3[1], &\text{ if }  j=3\\
\end{cases}
$$ \end{lemma} 

\begin{proof} We need to check that $F(t_i(P_j)) = T_i(F(P_j)) = T_i(U_j)$ for every $j$ and $i$ ranging from $1$ to $m$, $t_i(P_j)$ as given above and $U_j$ as defined in the proof of Lemma \ref{thm:equiv}. First let $i \geq 4$. We have $T_i(U_j) = U_j$ for every $j \neq i, i-1$. Hence, indeed, $t_i(P_j) = P_j$ if $j \neq i, i-1$. Let $j = i$. It is easy to check that $T_i(U_i) = U_{i-1}$, so $t_i(P_i) = P_{i-1}$. It remains to consider the case $j = i-1$. We have to show that $$F((P_i \xrightarrow{\beta_i} P_{i-1} \xrightarrow{soc} P_{i-1})[2]) = T_i(U_{i-1})$$

The right-hand side is easy to compute. Here we get $$T_i(U_{i-1}) = (R_3 \xrightarrow{\gamma_3'} R_4 \xrightarrow{} \dots \xrightarrow{} R_{i-3} \xrightarrow{\begin{pmatrix} \gamma_{i-3}' \\ 0 \end{pmatrix}} R_{i-2}\oplus R_i \xrightarrow{\begin{pmatrix}\gamma_{i-2}' & \gamma_i \end{pmatrix}} R_{i-1})$$ 

On the other hand, $F((P_i \xrightarrow{} P_{i-1} \xrightarrow{soc} P_{i-1})[2])$ is the totalization of the twisted complex shown below, where the underlined module is in bidegree $(0,0)$ (see \cite{R2} for details on the construction and the justification). It is easy to see that this results in the same answer.
\begin{center}
\begin{tikzcd}
R_i \arrow[bend left=20,swap, dotted]{drrr}{\gamma_{i-1}}  &
 & \\
R_{i-1} \arrow[u, "\gamma_{i-1}'"] \arrow[r, "(-1)^{i} id"] &  R_{i-1}\arrow[rr, "(-1)^i \gamma_{i-1}'\gamma_{i-1}"] && R_{i-1} \\ 
R_{i-2} \arrow[u,"\gamma_{i-2}'"] \arrow[r, "(-1)^{i-1} id"] & R_{i-2}  \arrow[u,"\gamma_{i-2}'"] \arrow[rr, "0"]  &&  R_{i-2} \arrow[u,"\gamma_{i-2}'"] \\ 
\arrow[u, no head, no tail, dotted]  & \arrow[u, no head, no tail, dotted] &&  \arrow[u, no head, no tail, dotted] \\ R_{3} \arrow[u,"\gamma_{3}'"] \arrow[r, " id"] & R_{3}  \arrow[u,"\gamma_{3}'"] \arrow[rr, "0"]  &&  \underline{R_{3}} \arrow[u,"\gamma_{3}'"] 
\end{tikzcd}
\end{center}

Next let $i=1$ or $i=2$. We have $t_i(P_i) = P_i[1]$, $t_1(P_2) = P_2$, $t_2(P_1) = P_1$. Now consider the case $j \geq 3$. We have ${T_i(U_j) = (R_i \xrightarrow{\gamma_i'} R_3 \xrightarrow{\gamma_3'} R_4 \xrightarrow{} \dots \xrightarrow{\gamma_{j-1}'} R_j)[1]}$ and the application of $F$ to ${t_i(P_j) = (P_i \xrightarrow{\delta_i\beta_m \dots \beta_{j+1}} P_j)[1]}$ gives the same complex. 

Finally, let $i= 3$. Since $U_3 = R_3$ and $T_3(R_3) = R_3[1]$, indeed, $t_3(P_3) = P[1]$. Similarly we have $t_3(P_i) = (P_3 \xrightarrow{\alpha_i} P_i)[1]$ for $i=1,2$. Now let $j = 4, \dots, m$. We have ${T_3(U_j) = (R_3 \xrightarrow{soc} R_3 \xrightarrow{\gamma_3'} 
\dots \xrightarrow{\gamma_{j-1}'} R_j)[1]}$. Again, this is exactly the totalization of the bicomplex obtained when applying $F$ to ${(P_3 \xrightarrow{\alpha_1\delta_1\beta_m \dots \beta_{j+1}} P_j)[1]}$. 
  \end{proof}  

Later we will face the challenge to compute various compositions of the twists $t_i$'s modulo $\Pic_0(\Lambda)$ or, in other words, the tilting complexes corresponding to compositions of the twists. This is not so pleasant using just the definition or Lemma \ref{thm:twonLambda}, because in either case one would have to deal with bicomplexes. A more convenient way is to express each of the autoequivalences $t_i$ as a composition of tilting mutations (modulo $\Pic_0(\Lambda)$) and then use this to compute compositions of twists as "long" compositions of mutations. 

 \begin{lemma}\label{thm:twviamut}  \begin{enumerate}
     \item  
     $$t_i(\Lambda) \cong (\mu_i^+)^2(\Lambda), t_i^{-1}(\Lambda) = (\mu_{i-1}^-)^2(\Lambda), \, \text{if } i \geq 4.$$
      \item  $$t_3(\Lambda) \cong  \mu_4^- \circ \dots \mu_m^- \circ \mu_2^- \circ \mu_1^-[1](\Lambda)$$
$${t_3^{-1}(\Lambda) \cong  \mu_1^+ \circ  \mu_2^+ \circ \mu_m^+ \dots \mu_4^+ [-1](\Lambda)}$$
     \item $$t_k(\Lambda) \cong  \mu_{3-k}^- \circ \mu_{3}^- \circ \mu_{3-k}^+ \circ \mu_4^- \circ \dots \circ \mu_m^- \circ \mu_{3-k}^-[1](\Lambda)$$ $${t_k^{-1}(\Lambda) \cong  \mu_{3-k}^+ \circ \mu_m^+ \circ  \mu_{3-k}^- \circ \mu_{m-1}^+ \dots \circ \mu_3^+ \circ \mu_{3-k}^+[-1](\Lambda)}, \, \text{if } k = 1,2.$$
    
 \end{enumerate} \end{lemma}

\begin{proof}  The proof is by direct computation. We will only consider the ``positive'' half of the statement and leave the analogous calculations for $t_i^{-1}$'s to the reader.
\begin{enumerate} 
    \item The tilting complex on the left-hand side of $t_i(\Lambda) \cong (\mu_i^+)^2(\Lambda)$ was already computed in Lemma \ref{thm:twonLambda}, so it remains to compute $(\mu_i^+)^2(\Lambda)$. Whenever we have a sequence of mutations applied to $\Lambda$, each subsequent mutation yields a tilting complex and a modified Brauer tree depicting its endomorphism algebra, as discussed in the previous section. There is a natural one-to-one correspondence between the edges of this modified Brauer tree and the indecomposable direct summands of the corresponding tilting complex. Hence, a step-by-step computation of the tilting complex corresponding to a sequence of mutations may be displayed as a sequence of modified Brauer trees whose edges are labeled with complexes. Taking the direct sum of the complexes written on all edges of a modified Brauer tree gives a tilting complex whose endomorphism algebra is depicted by this tree. Every subsequent complex differs from the previous one in only one summand and every subsequent tree differs from the previous one in only one edge. Below we do this for $(\mu_i^+)^2$ and then for the remaining two cases. The curvy red arrows remind the ordering of the relevant edges. Depending of whether we apply $\mu^+$ or $\mu^-$, the arrow points at or away from the edge along which the edge of the next mutation "slides" in the modified Brauer tree. For $(\mu_i^+)^2$ we get the following sequence. 
     \begin{figure}[H]
  \includegraphics[scale=1.3]{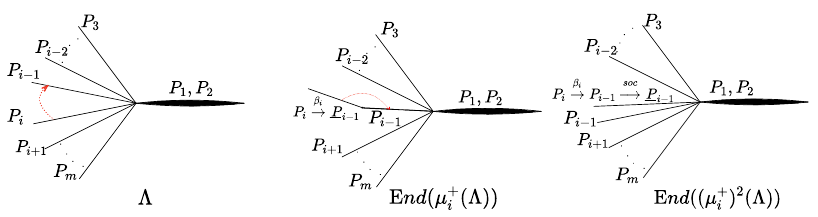}
\end{figure}
\noindent Here on the second step we obtained $P_i \xrightarrow{\beta_i} \underline{P_{i-1}}$ as the cone of the minimal left approximation of $P_i$ with respect to $\add(\bigoplus\nolimits_{j\neq i} P_j)$. On the final step we get $P_i \xrightarrow{\beta_i} P_{i-1} \xrightarrow{soc} \underline{P_{i-1}}$ as the cone of the minimal left $\add (\bigoplus\nolimits_{j\neq i} P_j)$-approximation of $P_i \xrightarrow{\beta_i} \underline{P_{i-1}}$. It is easy to see from Lemma \ref{thm:twonLambda} that $t_i(\Lambda) \cong \bigoplus_{j=1}^m t_i(P_j)$ is exactly the complex shown on the right in the picture above.

    \item Now we need to compute the complex $\mu_4^- \circ \dots \mu_m^- \circ \mu_2^- \circ \mu_1^-[1](\Lambda)$. It is convenient to group this sequence this into 2 steps, first applying $\mu_2^- \circ \mu_1^-$ and then the rest of the sequence. 
    
    \begin{figure}[H]
  \includegraphics[scale=1.3]{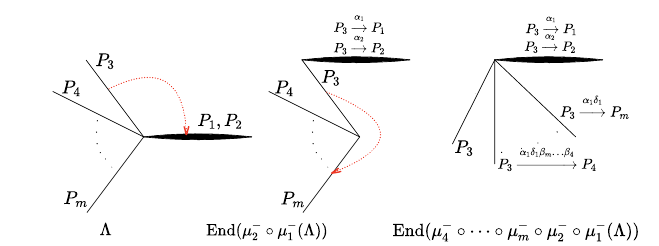}
\end{figure}

\noindent Same as in the previous case, by Lemma \ref{thm:twonLambda} this is exactly $t_3(\Lambda)$.
    \item The cases $k=1$ and $k=2$ are absolutely analogues, so we assume $k=1$ and compute ${\mu_{2}^- \circ \mu_{3}^- \circ \mu_{2}^+ \circ \mu_4^- \circ \dots \circ \mu_m^- \circ \mu_{2}^-[1](\Lambda)}$. 
     \begin{figure}[H] 
  \includegraphics[scale=1.2]{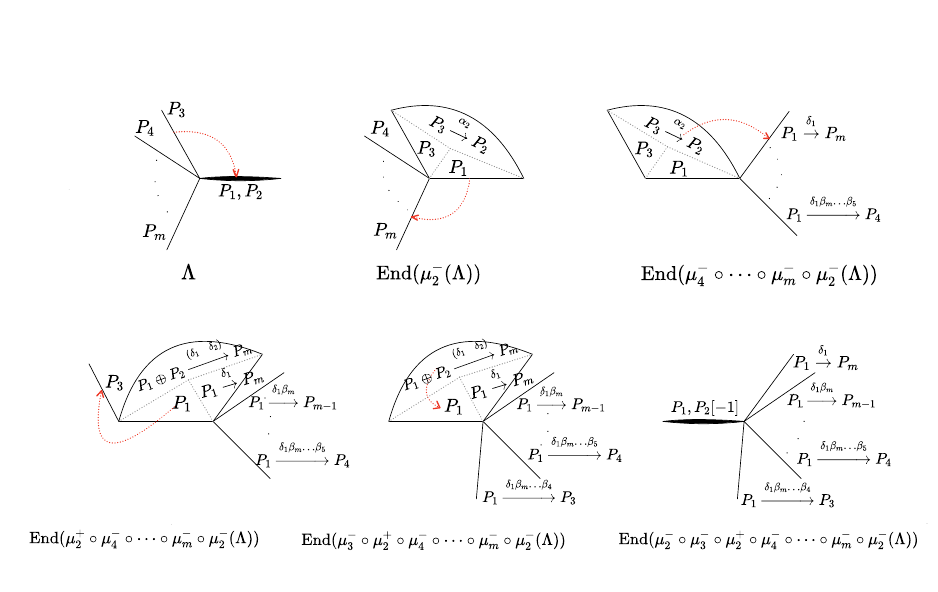} 
  \end{figure}
  \noindent Comparing the result to $t_1(\Lambda)$ obtained in Lemma \ref{thm:twonLambda}, we once again see that this is exactly what was claimed. 
\end{enumerate}
\end{proof} 

\begin{remark} When expressing compositions of several spherical twists as compositions of mutations and vice versa, we will have to carefully account for the labeling of summands. For example, while $(\mu_4^+)^2(\Lambda) \cong t_4(\Lambda)$, it is {\bf not} true that $(\mu_4^+)^4(\Lambda) \cong t_4 \circ t_4(\Lambda)$. In fact, $(\mu_4^+)^4$ does not even define an autoequivalence of $D^b(\Lambda)$, i.e. $\End((\mu_4^+)^4(\Lambda))$ is not isomorphic to $\Lambda$. The correct expression of $t_4 \circ t_4(\Lambda)$ via mutations is $(\mu_3^+)^2 \circ (\mu_4^+)^2(\Lambda)$, since $(\mu_4^+)^2$ interchanges the two edges $3$ and $4$. For the same reasons, $t_i^{-1}$ with $i \geq 4$ is $(\mu_{i-1}^-)^{2}$, and not $(\mu_i^-)^{2}$. \end{remark} 
The following technical lemma will be useful in our calculations later. It may be proved using induction on $k$ and expressing the composition of twists as a composition of mutations with the help of Lemma \ref{thm:twviamut}. We leave this to the reader. 

\begin{lemma}\label{twistseq1} Let $\mathcal{F}_k \colon D^b(\Lambda) \to D^b(\Lambda)$, $\mathcal{F}_k = (t_m \circ \dots \circ t_4 \circ t_3 \circ t_1 \circ t_2 \circ t_3)^k[-2k]$. Then $$
\mathcal{F}_k(P_1) = P_1, \mathcal{F}_k(P_2) = P_2, 
\mathcal{F}_k(P_i) = P_{k+i}, \text{ if } 3 \leq i \leq m-k$$ $$
\mathcal{F}_k(P_i) = (P_{i-m+k+2} \xrightarrow{\begin{pmatrix} \beta_{i-m+k+2} \dots \beta_4\alpha_1 \\ \beta_{i-m+k+2} \dots \beta_4\alpha_2  \end{pmatrix}} P_1 \oplus P_2)[1], \text{ if }  m-k+1 \leq i \leq m $$
\end{lemma}

\begin{remark}
  A convenient way to express the statement of Lemma \ref{twistseq1} above is to draw the star with a double edge and label its edges with indecomposable summands of the tilting complex $\mathcal{F}_k(\Lambda)$, as shown below (see the discussion in the proof of Lemma \ref{thm:twviamut}).  
 \begin{figure}[H]
  \includegraphics[scale=1.9]{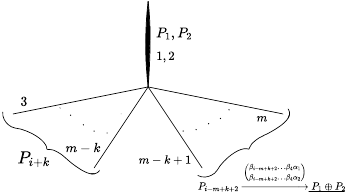}
\end{figure}

One should read the picture as follows. When there is a complex written under a brace around several edges, we substitute the labels of these edges for $i$ to get direct summands of the tilting complex. For instance, here we have summands $P_{i+k}$ for $i=3, \dots, m-k$, i.e. $P_{k+3}, \dots, P_{m}$. This tree with this particular labeling of edges is obtained if we express $\mathcal{F}_k$ as a composition of mutations via Lemma \ref{thm:twviamut}. 
  \end{remark}

\subsection{Some relations in \texorpdfstring{$\TrPic(\mathcal{R})$}{Lg}}

So far we have encountered elements of $\TrPic(\mathcal{R})$ of three types, namely, the spherical twists $T_i$, elements of $\Pic(\mathcal{R}) \cong \Out(\mathcal{R})$ and the shifts. Later we will show that in fact the whole group $\TrPic(\mathcal{R})$ is generated by such standard autoequivalences. In this section we will investigate which relations these generators are subject to.

All the information about relations between the spherical twists $\{T_i\}_{i=1}^m$ can be be derived from a somewhat non-trivial result on the faithfulness of braid group actions which is luckily at our disposal. 

\begin{theorem}[Corollary 1, \cite{NV}]\label{thm:faith} The subgroup of $\TrPic(\mathcal{R}_m)$ generated by the twists $\{T_i\}_{i=1}^m$ is isomorphic to the Artin group $B_{D_m}$ of type $D_m$. 
\end{theorem}

To establish some other relations in $\TrPic(\mathcal{R})$ we will use the fact that there is a canonical homomorphism from the derived Picard group to the stable Picard group, which we define below. 

\begin{definition}[Definition 3.8, \cite{RZ}]\label{thm:stable} Let $A$ be a finite-dimensional $\kk$-algebra. The {\bf stable Picard group} $\StPic(A)$ of $A$ is the group of isomorphism classes of projective-free $A^{op} \otimes A$-modules inducing stable equivalences of $A$, with the product of the classes of $M$ and $N$ defined to be the projective-free part of $M \otimes_A N$. 
 \end{definition}
 
 \begin{theorem}[Corollary 2.15, \cite{RZ}] There is a canonical homomorphism $S \colon \TrPic(A) \to \StPic(A)$, sending a two-sided tilting complex $X$ to $\Omega^n_{A\otimes A^{op}}C^n$, where $C$ is a bounded complex of $A^{op} \otimes A$-modules, such that it is isomorphic to $X$, $C^i = 0$ for $i > n$, $C^i$ are projective bimodules for $i < n$, and $C^n$ is projective as an $A$-module and projective as an $A^{op}-$module. Moreover, $S$ is injective on $\Pic_0(A)$.   \end{theorem} 
 
 Now let $\tau^{\mathcal{R}} \in \Pic(\mathcal{R})$ be the automorphism interchanging $e_1$ and $e_2$, $\gamma_1$ and $\gamma_2$, $\gamma_1'$ and $\gamma_2'$, and fixing all other idempotents and arrows. For $a \in k^*$ let $g_a \in \Pic_0(\mathcal{R})$ be the automorphism multiplying $\gamma_1'$, $\gamma_2'$ and $\gamma_3'$ by $a$ and fixing all idempotents and remaining arrows. It is easy to see that $F \circ \tau^{\Lambda} =  \tau^{\mathcal{R}} \circ F$ and $F \circ f_a =  g_a \circ F$, where $F \colon \K^b(\proj-\Lambda) \to \K^b(\proj-\mathcal{R})$ is the equivalence constructed in the proof of Lemma \ref{thm:equiv}.

\begin{lemma}\label{thm:zentrum} \begin{enumerate}
    \item   Let $C = T_m \circ T_{m-1} \circ \dots \circ T_1  \in \TrPic(\mathcal{R}_m)$. $$ C^{m-1} = g_{-1} [2m-3], \text{ if } m  \text{ is even}$$ 
$$ C^{m-1} = g_{-1} \circ \tau^{\mathcal{R}} [2m-3], \text{ if } m  \text{ is odd}$$ 
\item $$\tau^{\mathcal{R}} \circ T_i = T_i \circ \tau^{\mathcal{R}} \text{ for } i \neq 1,2$$
$$T_1 \circ \tau^{\mathcal{R}} =  \tau^{\mathcal{R}} \circ T_2$$
\end{enumerate}
\end{lemma} 

\begin{proof} \begin{enumerate}
    \item  First we establish that the claim holds modulo $\Pic_0(\mathcal{R_m})$. In other words, we will show that for any $m \geq 4$
$$C^{m-1}(P_i) = P_i[2m-3], \text{ for } i=3,\dots,m. $$
In addition, if $m$ is even, then
$$C^{m-1}(P_i) =  P_i[2m-3], \text{ for } i=1,2.$$
And if $m$ is odd, then 
$$C^{m-1}(P_i) =  P_{3-i}[2m-3], \text{ for } i=1,2.$$
Indeed, one can easily see that 
$$C(P_1) =  (P_2 \xrightarrow{\gamma_2'} P_3)[2]$$
$$C(P_2) = (P_1 \xrightarrow{\gamma_1'} P_3)[2]$$
$$C(P_i) = P_{i+1}, \text{ if } 3 \leq i \leq m-1$$ 
$$C(P_m) = (P_1 \oplus P_2 \xrightarrow{\begin{pmatrix} \gamma_1' & \gamma_2' \end{pmatrix}} P_3 \xrightarrow{\gamma_3'} P_4 \xrightarrow{\gamma_4'} \dots \xrightarrow{\gamma_{m-1}'} P_m)[m-1]$$
It is also clear that $$C^2(P_m) = T_1T_2T_3(P_1 \oplus P_2 \xrightarrow{\begin{pmatrix} \gamma_1' & \gamma_2' \end{pmatrix}} P_3)[m-1] = T_1T_2(P_3 \xrightarrow{\begin{pmatrix} -\gamma_1 \\ \gamma_2 \end{pmatrix}} P_1 \oplus P_2)[m] = P_3[m]$$
Thus, ${C^{n-1}(P_m) = C^{n-3}(P_3[m]) = C^{m-4}(P_4[m+1]) = \dots = C(P_{m-1}[2m-4]) = P_m[2m-3]}$ and for every $i = 3, \dots, m-1$ we have 
$$C^{m-1}(P_i) = C^{i-1} C^{m-i}(P_i) = C^{i-1}P_m[m-i] = C^{i-3}P_3[3m-i] = P_i[2m-3]. $$
Now let $i = 1$ or $2$. We have 
$$\hspace{-2em} \small{C^{n-1}(P_i) = C^{n-2}(P_{3-i} \xrightarrow{\gamma_{3-i}'} P_3)[2] =  C^{n-3}(P_i \xrightarrow{\gamma_i'} P_3 \xrightarrow{-\gamma_3'} P_4)[4] = \dots = C(P_{j} \xrightarrow{\gamma_j'} P_3 \xrightarrow{} \dots \xrightarrow{\pm \gamma_{m-1}'} P_m)[2m-4],}$$ 
where $j = i$ if $m$ is even and $j = 3-i$ if $m$ is odd. Finally, 
$$C(P_{j} \xrightarrow{\gamma_j'} P_3 \xrightarrow{} \dots \xrightarrow{\pm \gamma_{m-1}'} P_m)[2m-4] = T_1T_2(P_j[2m-4]) = T_{3-j}T_jP_j[2m-4] = P_j[2m-3].$$ 
Hence, indeed, $C^{m-1}(P_i) = P_i[2m-3]$ if $m$ is even and $C^{m-1}(P_i) = P_{3-i}[2m-3]$ if $m$ is odd. 
Now since the canonical homomorphism $S \colon \TrPic(\mathcal{R}_m) \to \StPic(\mathcal{R}_m)$ is injective on $\Pic_0(\mathcal{R}_m)$, it remains to show that it sends $C^{m-1} \circ g_{-1} \circ \tau^{\mathcal{R}} [-2m+3] \in \Pic_0(\mathcal{R}_m)$ (if $m$ is odd) or $C^{m-1} \circ g_{-1} [-2m+3] \in \Pic_0(\mathcal{R}_m)$ (if $m$ is even) to $\mathcal{R}_m$ in $\StPic(\mathcal{R}_m)$. Note that for every $i$ the twist $T_i$ is defined by the following two-sided tilting complex, where $\mathcal{R}_m$ is in degree $0$ and $P_i^*$ denotes the $k-$dual of $P_i$: 

$$ \dots \to 0 \to P_i \otimes_{\mathcal{R}_m} P_i^* \xrightarrow{ev} \mathcal{R} \to 0 \to \dots $$ 

Hence $S$ sends each $T_i$ to $\mathcal{R}_m$. Recall that by ${}_\varphi \mathcal{R}_m$ for $\varphi \in \Aut(\mathcal{R}_m)$ we denote the $\mathcal{R}_m$-bimodule which is regular regarded as a right module and has its left multiplication twisted by $\varphi$, i.e. $r \cdot t = \varphi(r)t$ for any $t \in {}_\varphi \mathcal{R}_m$, $r \in \mathcal{R}$. Then we have $S(g_{-1}) = {}_{g_{-1}}\mathcal{R}_m$ and ${S(g_{-1} \circ \tau^{\mathcal{R}}) = {}_{g_{-1}\circ \tau^{\mathcal{R}}}\mathcal{R}_m }$. Since $S([-2m+3]) = \Omega^{2m-3}_{\mathcal{R}_m \otimes \mathcal{R}^{op}_m} (\mathcal{R}_m)$, it is sufficient to show that $\Omega^{2m-3}_{\mathcal{R}_m \otimes \mathcal{R}^{op}_m} (\mathcal{R}_m) \cong {}_{g_{-1}}\mathcal{R}_m$ or ${\Omega^{2m-3}_{\mathcal{R}_m \otimes \mathcal{R}^{op}_m} (\mathcal{R}_m) \cong {}_{g_{-1}\circ \tau^{\mathcal{R}}}\mathcal{R}_m}$ in the stable category respectively if $m$ is even or odd.

By Remark 1 in \cite{GV} we have $\Omega^{2m-3}_{\Lambda_m \otimes \Lambda^{op}_m} (\Lambda_m) \cong {}_{f_{-1}} \Lambda_m$ if $m$ is even and $\Omega^{2m-3}_{\Lambda_m \otimes \Lambda^{op}_m} (\Lambda_m) \cong {}_{f_{-1}\circ \tau^{\Lambda}} \Lambda_m$ if $m$ is odd. By Corollary 2.15 in \cite{RZ}, the derived equivalence $F$ between $\Lambda_m$ and $\mathcal{R}_m$ induces a stable equivalence, i.e. there is a $\Lambda_m-\mathcal{R}_m$-bimodule $M$ and an $\mathcal{R}_m-\Lambda_m$-bimodule $N$ such that in the stable category we have $\Omega^{2m-3}_{\mathcal{R}_m \otimes \mathcal{R}^{op}_m} (\mathcal{R}_m) \cong M \otimes_{\Lambda_m} \Omega^{2m-3}_{\Lambda_m \otimes \Lambda^{op}_m} (\Lambda_m) \otimes_{\Lambda_m} N$. But since $F \circ \tau^{\Lambda} =  \tau^{\mathcal{R}} \circ F$ and $F \circ f_{-1} =  g_{-1} \circ F$, we have in the stable category

$$\Omega^{2m-3}_{\mathcal{R}_m \otimes \mathcal{R}^{op}_m} (\mathcal{R}_m) \cong M \otimes_{\Lambda_m} \Omega^{2m-3}_{\Lambda_m \otimes \Lambda^{op}_m} (\Lambda_m) \otimes_{\Lambda_m} N \cong M \otimes_{\Lambda_m} {}_{f_{-1}} \Lambda_m \otimes_{\Lambda_m} N \cong {}_{g_{-1}}\mathcal{R}_m $$
if $m$ is even, and 
$$\Omega^{2m-3}_{\mathcal{R}_m \otimes \mathcal{R}^{op}_m} (\mathcal{R}_m) \cong M \otimes_{\Lambda_m} \Omega^{2m-3}_{{\Lambda_m} \otimes {\Lambda^{op}_m}} (\Lambda) \otimes_{\Lambda_m} N \cong M \otimes_{\Lambda_m} {}_{f_{-1}\circ \tau^{\Lambda}} \Lambda_m \otimes_{\Lambda_m} N \cong  {}_{g_{-1}\circ \tau^{\mathcal{R}}}\mathcal{R}_m$$ 
if $m$ is odd. This finishes the proof. 
\item This is immediate modulo $\Pic_0$. We have
$$ \tau^\mathcal{R}(T_i(P_j)) = T_i(P_j) = T_i((\tau^\mathcal{R}(P_j)), \ j \geq 3, \ i \geq 3 $$
$$ \tau^\mathcal{R}(T_i(P_j)) = P_{3-j} =T_i(P_{3-j}) = T_i((\tau^\mathcal{R}(P_j)), \ j = 1,2, \ i \neq 3$$
$$ \tau^\mathcal{R}(T_3(P_j)) = (P_3 \xrightarrow{} P_{3-j})[1] = T_3((\tau^\mathcal{R}(P_j)), \ j = 1,2 $$ 
$$T_1(\tau^{\mathcal{R}}(P_1)) = P_2 = \tau^{\mathcal{R}}(T_2(P_1)),\  T_1(\tau^{\mathcal{R}}(P_2)) = P_1[1] = \tau^{\mathcal{R}}(T_2(P_2)) $$ $$ 
T_1(\tau(P_3)) = (P_1 \xrightarrow{} P_3)[1] = \tau^{\mathcal{R}}(T_2(P_3)), \ T_1(\tau^{\mathcal{R}}(P_j)) = P_j = \tau^{\mathcal{R}}(T_2(P_j)), \ j \geq 4 $$
Now we can again conclude that the required relations hold by using the same stable Picard group argument as above. Indeed, since the canonical homomorphism ${S \colon \TrPic(\mathcal{R}) \to \StPic(\mathcal{R})}$ is injective on $\Pic_0$, it suffices to show that $S(\tau^{\mathcal{R}} \circ T_i) = S( T_i \circ \tau^{\mathcal{R}})$ for $i \geq 3$ and $S(\tau^\mathcal{R} \circ T_1) = S(T_2 \circ \tau^\mathcal{R})$. But this is trivial, because $S$ sends every $T_i$ to $\mathcal{R}$ and $\tau^\mathcal{R}$ to ${}_{\tau^{\mathcal{R}}} \mathcal{R}$. 

\end{enumerate}
\end{proof} 

\section{Main result}\label{sec:mainres} 
 Recall that by $B_{D_m}$ we denote the Artin group of type $D_m$ (see Definition \ref{def:braid}). By $\{\sigma_i\}_{i=1}^m$ we will denote the set of standard generators of $B_{D_m}$, i.e. we will always assume that $\sigma_i$ satisfy braid relations of type $D_m$. Let $c =\sigma_1\dots\sigma_m \in B_{D_m}$. We define the following group by generators and relations 

$$ \widetilde{G}_{m} = \bigg\langle \sigma_i \in B_{D_{m}},\{\kappa_a\}_{a \in k^*}, \tau, s \bigg| \begin{array}{c}
\tau \sigma_i = \sigma_i \tau \ (i\neq 1,2); \ \tau \sigma_1 = \sigma_2 \tau; \  \tau^2 = e; \  \\  s \text{ and }  \kappa_a \ \forall a \in k^* \text{ commute with every generator}
\end{array} \bigg\rangle $$

\noindent Here we mean that $\kappa_a$ for $a \in k^*$ satisfy the same relations which hold for the corresponding elements of the multiplicative group of the field, i.e. $\kappa_a\kappa_b = \kappa_{ab}$ and $\kappa_1 = e$. Now depending on the parity of $m$ we add one more relation, derived in Lemma \ref{thm:zentrum}. Namely, let $G_m = \widetilde{G}_m/ \langle c^{m-1}s^{-2m+3} \kappa_{-1} \rangle$ if $m$ is even and $G_m = \widetilde{G}_m/ \langle c^{m-1}s^{-2m+3} \kappa_{-1}\tau \rangle$ if $m$ is odd.

\begin{theorem}\label{thm:main} The derived Picard group $\TrPic(\mathcal{R}_m)$ of $\mathcal{R}_m$ is isomorphic $G_m$.  \end{theorem}

\begin{proof} In order to simplify the notations, in the proof we denote the category $\K^b(\proj-\mathcal{R_m})$ by just $\K$. We begin by constructing a homomorphism $\Phi \colon G_m \to \TrPic(\mathcal{R})$. Let $\Phi(\sigma_i) = T_i$, $\Phi(s) = [1]$, $\Phi(\tau) = \tau^{\mathcal{R}}$ and $\Phi(\kappa_a) = g_a$. One has to check that $\Phi$ is a homomorphism first. Indeed, the spherical twists $\{T_i\}_{i=1}^m$ are known to satisfy braid relations of type $D$ (see \cite{ST} for details). Applying the stable Picard group argument we used in the proof of Theorem \ref{thm:zentrum}, it is easy to see $\{g_a\}_{a\in k^*}$ and the shift commute with each other, $\tau^\Lambda$ and the twists. The relation $(\tau^{\mathcal{R}})^2 = Id_{\K}$ is clear. Finally, relations added when passing from $\widetilde{G}_m$ to $G_m$ as well as $\tau^{\mathcal{R}} \circ T_i = T_i \circ \tau^{\mathcal{R}} $ for $i \neq 1,2$ and $T_1 \circ \tau^{\mathcal{R}} =  \tau^{\mathcal{R}} \circ T_2$ are satisfied in $\TrPic(\mathcal{R}_m)$ by Lemma \ref{thm:zentrum}.

Next we will show that $\Phi$ is injective. Note that every element in $G_m$ can be written in the form $\tau^b t \kappa_a s^d $, where $b = 1$ or $0$, $a \in k^*$, $d \in \Z$, $t \in B_{D_m}\leqslant G_m$. Suppose ${\Phi(\tau^b t \kappa_a s^d) = (\tau^{\mathcal{R}})^bT g_a [d] = Id_{K}}$, where $T = \Phi(t) \in B_{D_m} \leqslant \TrPic(\mathcal{R}_m)$ (by Theorem \ref{thm:faith}). Then  $(\tau^{\mathcal{R}})^bT$ commutes with every $T_i$. First let $b=0$. Then $T$ is in the center of the subgroup $B_{D_m}$ generated by the twists in $\TrPic(\mathcal{R}_m)$. The center $Z(B_{D_m})$ of $B_{D_m}$ is the infinite cyclic subgroup generated by $c^{m-1} = (\sigma_1 \dots \sigma_m)^{m-1}$ if $m$ is odd and by $c^{2m-2}$ if $m$ is even (see \cite{BS}). Recall that by $C$ we denote $T_m \circ \dots\ \circ T_1 \in \TrPic(\mathcal{R}_m)$. Let $m$ be even. Then 
$$ T = C^{(m-1)k} = (g_{-1})^k[k(2m-3)], \  (g_{-1})^k[k(2m-3)]g_a [d] = g_{(-1)^ka}[k(2m-3)+d] = Id_{\K}  $$
Thus, by  Theorem \ref{thm:faith} 
$$d=-k(2m-3), a = (-1)^k, t = c^{(m-1)k}.$$
Then $t \kappa_a s^c = c^{(m-1)k} \kappa_{(-1)^k}s^{-k(2m-3)} = e \in G_m$. Now let $m$ be odd. We have 
$$ T = C^{(m-1)2k} = (\tau^{\mathcal{R}})^{2k}(g_{-1})^{2k}[2k(2m-3)] = [2k(2m-3)], \ g_a [d+2k(2m-3)] = Id_{\K} $$
Hence, $a=1, d= -2k(2m-3), t = c^{2k(m-1)}$ and again $${t \kappa_a s^d = c^{2k(m-1)}[-2k(2m-3)] = e \in G_m}.$$ Now we consider the case $b=1$. If $\tau^{\mathcal{R}}T$ commutes with every $T_i$, we have 
$$ \tau^{\mathcal{R}} T T_i =  T_i\tau^{\mathcal{R}} T = \tau^{\mathcal{R}} T_i T, \ i \geq 3  $$
$$ \tau^{\mathcal{R}} T T_1 =  T_1\tau^{\mathcal{R}} T = \tau^{\mathcal{R}} T_2 T. $$
Thus, $T$ is an element of the Artin group $B_{D_m}$ such that it commutes with the generators $T_i$ for $i \geq 3$. Moreover, conjugation by $T$ interchanges $T_1$ and $T_2$. If $m$ is even, we get a contradiction, because the automorphism of $B_{D_{2k}}$ interchanging the two generators $T_1$ and $T_2$ is not inner even in the corresponding Coxeter group (see \cite{F}). Now suppose $m$ is odd. In this case $\tau^{\mathcal{R}}T = g_{-1}C^{m-1}[-2m+3]T$ commutes with every $T_i$, hence so does $C^{m-1}T$. Thus, we have $$C^{m-1}T = C^{(m-1)2k} = (\tau^{\mathcal{R}})^{2k}(g_{-1})^{2k}[2k(2m-3)] = [2k(2m-3)]$$
$$ \tau^{\mathcal{R}} T g_a [d] =  g_{-1}C^{m-1}T g_a [d-2m+3] =  g_{-1}g_{a}[2k(2m-3)+d-2m+3] = Id_{\K}$$
Thus, $$a=-1, d=-(2k-1)(2m-3), t = c^{(2k-1)(m-1)}$$ and we have $${\tau g_{-1} c^{(2k-1)(m-1)} s^{-(2k-1)(m-1)} = e \in G_m.}$$ 
 Finally, we need to show that $\Phi$ is surjective. It is sufficient to show that the derived Picard group $\TrPic(\Lambda)$ of $\Lambda$ is generated by $F^*(T_i)=t_i \ (i=1,\dots,m)$, $\Pic(\Lambda)$ and the shift. For this we adopt the general strategy of Zvonareva used in \cite{Alexandra}. The proof of this part is somewhat technically involved, so below we only describe the plan and then fill in the details in what follows. As we have noted before, this scheme of the proof is essentially lifted from \cite{Alexandra}. 
\begin{enumerate}
    \item[I.] For every modified Brauer tree $G$ we construct a particular {\bf standard sequence} of left mutations ${\mu_{G}^{st} = \mu_{i_d}^+ \circ \dots \circ \mu_{i_1}^+}$ such that $G$ is the modified Brauer tree of $\End_{\D^b(\Lambda)}(\mu_{G}^{st}(\Lambda))$. In this case $\mu_{G}^{st}$ naturally equips the graph $G$ with a labeling of edges, also referred to as standard (coming from the usual labeling of edges of $\Gamma=G_\Lambda$, the modified Brauer tree of $\Lambda$). Section \ref{sec:stmut} is devoted to this part of the proof. 
    \item[II.]  Recall that by $\Gamma$ we denote the modified Brauer tree of $\Lambda$ with the usual (standard in terms of I., for the empty sequence of mutations) labeling of edges, as in Definition \ref{def:Lambda}. Denote the group generated by $\{t_i\}_{i=1}^m$, $\Pic(\Lambda)$ and the shift $[1]$ by $\mathcal{P}$. Recall that the elements of $\TrPic(\Lambda)$ modulo $\Pic(\Lambda)$ are in correspondence with tilting complexes $T$ over $\Lambda$ with $\End_{\D^b(\Lambda)}(T) \cong \Lambda$ (see Section \ref{sec:prelim}). Thus, it is sufficient to show that for every tilting complex $T$ with $\End_{\D^b(\Lambda)}(T) \cong \Lambda$ there is an autoequivalence $F \in \mathcal{P}$ which sends $\Lambda$ to $T$. We can assume without loss of generality that $T$ is concentrated in non-positive degrees. Every representation-finite symmetric algebra is tilting-connected (Theorem \ref{thm:tiltconn}), thus there exists a sequence of left mutations such that ${\mu_{j_k}^+ \circ \dots \circ \mu_{j_1}^+(\Lambda) = T}$. Our final goal is to express $T$ as an image of a composition of twists modulo shifts and $\Pic(\Lambda)$. To do this, we will split this sequence of mutations into small understandable pieces. In turn, this will require us to carefully account for different labelings. More precisely, let $\mu^s = \mu_{j_s}^+ \circ \dots \circ \mu_{j_1}^+$ and $T^s = \mu^s(\Lambda)$ for $1 \leq s \leq k$. Let $G^{s}$ be the modified Brauer tree of $A^s = \End_{\D^b(\Lambda)}(T^s)$. There are two labelings of the edges of $G^s$ naturally obtained from $\Gamma$. Namely, the standard one discussed in I., emerging via $\mu^{st}_{G^s}$, and the one emerging via $\mu^s$. Let $\sigma_s \in S_m$ be the permutation of the labels on the edges of $G^s$ which sends the standard labeling to the one coming via $\mu^s$. Now let $\Gamma^{\sigma_s}$ be the modified Brauer tree obtained from $\Gamma$ by applying the permutation $\sigma_s$ to the labels on its edges (i.e. it is still $\Gamma$ as a graph, but its edges are labeled differently). In addition, by $\Lambda^{\sigma_s} \cong \Lambda$ we denote the corresponding algebra. Let $\mu_{G^s}^{st} = \mu_{i_d}^+ \circ \dots \circ \mu_{i_1}^+$ be the standard sequence of mutations for $G^s$. By $\widetilde{\mu}_{G^s}^{st}$ we denote the series of mutations $\mu_{\sigma_s(i_d)}^+ \circ \dots \circ \mu_{\sigma_s(i_1)}^+$. Observe that the modified Brauer tree of $\End_{\D^b(\Lambda)}(\widetilde{\mu}_{G^s}^{st}(\Lambda^{\sigma_s}))$ is $G^s$ with the labeling of edges induced by $\mu^s$. Denote $(\widetilde{\mu}_{G_s}^{st})^{-1} = \mu_{\sigma_s(i_1)}^- \circ \dots \mu_{\sigma_s(i_d)}^-$. Clearly, $\widetilde{\mu}_{G^s}^{st} \circ (\widetilde{\mu}_{G^s}^{st})^{-1}(A^s) = A^s$, so we can write:  
   $$T = \mu_{j_k}^+ \circ \dots \circ \mu_{j_1}^+(\Lambda) = \mu_{j_k}^+ \circ \widetilde{\mu}_{G^{k-1}}^{st} \circ (\widetilde{\mu}_{G^{k-1}}^{st})^{-1} \circ \mu_{j_{k-1}}^+ \circ \dots \circ \widetilde{\mu}_{G^2}^{st} \circ (\widetilde{\mu}_{G^2}^{st})^{-1} \circ \mu_{j_2}^+ \circ \widetilde{\mu}_{G^1}^{st} \circ (\widetilde{\mu}_{G^1}^{st})^{-1} \circ \mu_{j_1}^+(\Lambda) $$
   
  \item[III.] Now we can split this sequence of mutations into several pieces and prove that each of the them induces an autoequivalence which lies in $\mathcal{P}$. Indeed, note that $\End_{\D^b(\Lambda)}((\widetilde{\mu}_{G^1}^{st})^{-1} \circ \mu_{j_1}^+(\Lambda)) \cong \Lambda^{\sigma_1}$, $\End_{\D^b(\Lambda)}(\mu_{j_k}^+ \circ \widetilde{\mu}_{G^{k-1}}^{st}(\Lambda^{\sigma_{k-1}})) \cong \Lambda^{\sigma_k}$, $\End_{\D^b(\Lambda)} ((\widetilde{\mu}_{G^{s+1}}^{st})^{-1} \circ \mu_{j_{s+1}}^+ \circ \widetilde{\mu}_{G^s}^{st}(\Lambda^{\sigma_s})) \cong \Lambda^{\sigma_{s+1}}$, $s=1, \dots, k-2$. Then by Lemma 3 in \cite{Alexandra} one can see that it is sufficient to express each of the complexes $(\widetilde{\mu}_{G^1}^{st})^{-1} \circ \mu_{j_1}^+(\Lambda)$, $\mu_{j_k}^+ \circ \widetilde{\mu}_{G^{k-1}}^{st}(\Lambda^{\sigma_{k-1}})$ and $(\widetilde{\mu}_{G^{s+1}}^{st})^{-1} \circ \mu_{j_{s+1}}^+ \circ \widetilde{\mu}_{G^s}^{st}(\Lambda^{\sigma_s})$ as images of $\Lambda$ under some series of spherical twists $\{t_i\}_{i=1}^m$ modulo $\Pic(\Lambda)$ and the shift. This final step of the proof is fulfilled in Section   \ref{sec:stanyst}.
  
  %Indeed, the image of $\Lambda$ under an autoequivalence of $D^b(\Lambda)$ determines an element of $\TrPic(\Lambda)$ uniquely modulo $\Pic(\Lambda)$, hence, having expressed the complexes above as images of $\Lambda$ under elements of $\mathcal{P}$, we will have established that $T$ is also an image of $\Lambda$ under an element of $\mathcal{P}$. 
  
  \end{enumerate}
\end{proof}

\section{Standard series of mutations}\label{sec:stmut}
 As before, throughout this section $m$ denotes a fixed integer, $m \geq 4$. Recall that by $\Gamma = \Gamma_m$ we denote the star with $m-2$ regular edges and a double edge. The edges of $\Gamma$ are labeled with integers from $1$ to $m$ as in Definition \ref{def:Lambda}. By $\Lambda = \Lambda_m =  \bigoplus_{i=1}^m P_i$ we denote the corresponding algebra. Let $G$ be any modified Brauer tree with $m$ edges (the double edge is counted as two edges). The goal of this section is to construct a particular series of left mutations $\mu_G^{st} = \mu_{i_d}^+ \circ \dots \circ\mu_{i_1}^+$ for each $G$ in such a way that $G$ would be the modified Brauer tree corresponding to $\End_{\D^b(\Lambda)}(\mu_G^{st}(\Lambda))$. More precisely, we want the modified Brauer tree of $\End_{\D^b(\Lambda)}(\mu_G^{st}(\Lambda))$ to be isomorphic to $G$ as a graph together with a cyclic ordering of half-edges around each vertex (or as a graph together with an embedding into the place). In the case of Double Edge trees, we assume that graph isomorphisms send the double edge to the double edge. From now on this is what we always mean by ``the modified Brauer tree of $\End_{\D^b(\Lambda)}(\mu_G^{st}(\Lambda))$ is $G$''. We will refer to $\mu_G^{st}$ as {\bf the standard series of mutations} for $G$. Note that such a series $\mu_G^{st}$ naturally equips the graph $G$ with a labeling of edges emerging from the standard labeling of edges of $\Gamma$. This labeling on $G$ will also be referred to as {\bf standard}. Constructing the desired series $\mu_G^{st}$, we are also going to compute the tilting complexes $\mu_G^{st}(\Lambda)$ and describe the standard labeling explicitly.  
 
 \subsection{The Double Edge case}
 Suppose that $G$ is of the Double Edge type. We fix the following labeling of edges of $G$, which is going to coincide with the labeling induced by the standard sequence $\mu_G^{st}$ that is to be constructed. Let the double edge be labeled with $1, 2$ and let the end of the double edge that is not a leaf be the root of the tree. Now label the regular edges using the depth-first search {\bf starting from the edge followed by the double edge} in the cyclic ordering of edges around the root. See any of the trees in Figure \ref{fig:stex} for an example. Note that this labeling is consistent with our standard labeling for the Double Edge star $\Gamma$, introduced in Definition \ref{def:Lambda}. When this labeling is fixed, we will write $E(G)$ for the set of labels on the edges, i.e. $i \in E(G)$ is the edge of $G$ to which the number $i \in \{1, \dots, m\}$ is assigned. Now let ${\varphi_G \colon E(G) \to \{0, \dots, m-3\}}$ be the map sending each of the edges of $G$ to the length of the shortest path from this edge to the root. In particular, $\varphi_G(1) = 0$,$\varphi_G(2) = 0$, $\varphi_G(3) = 0$. We refer to $\varphi_G(i)$ as {\bf the level} of the edge $i$ in the tree $G$. In addition, if $i \in E(G)$ is an edge of $G$ with $\varphi_G(i) > 0$, let $\psi_G(i)$ be the unique edge adjacent to $i$ with $\varphi_G(\psi_G(i)) = \varphi_G(i) - 1$. For instance, if there is no edge following $i$ in the cyclic ordering on the same level as $i$, then we have $\psi_G(i) = i-1$.  
 
 \begin{lemma}\label{thm:steared} Let $\mu_G^{st} = (\mu^+_m)^{\varphi_G(m)} \circ \dots \circ (\mu^+_4)^{\varphi_G(4)}$. Then the modified Brauer tree of $\End_{\D^b(\Lambda)}(\mu_G^{st}(\Lambda))$ is $G$. Moreover, the labeling of edges of $G$ induced by $\mu_G^{st}$ coincides with the depth-first labeling described above. The indecomposable summands of the tilting complex $T^{st}_G = \mu_G^{st}(\Lambda)$ are as follows: 
 
 $$ (T_G^{st})_i = P_i \text{ if } \varphi_G(i) = 0, \quad 3 \leq i \leq m $$
$$ (T_G^{st})_i = \big(P_i \to P_{\psi_G(i)}\big)[\varphi_G(i)] \text{ if } \varphi_G(i) > 0, \quad 3 \leq i \leq m $$
$$(T_G^{st})_1 = P_{1}, (T_G^{st})_2 = P_{2}$$

We will say that $\mu_G^{st}$ is the standard sequence of mutations for $G$. 

  \end{lemma} 
  
  \begin{figure}[H]
  \centering
  \includegraphics[width=.9\linewidth]{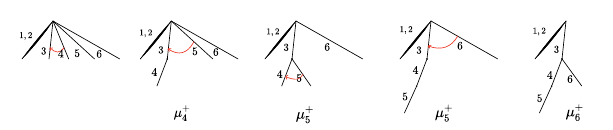}
  \caption{The standard sequence of mutations for the Double Edge tree $G$ (on the right,. with $\varphi_G(3) = 0, \varphi_G(4) = \varphi_G(6) = 1, \varphi_G(5) = 2$.}
  \label{fig:stex}
\end{figure}

Figure \ref{fig:stex} above shows an example of a standard sequence of mutations for a Double Edge tree. Lemma \ref{thm:steared} is easy to establish by induction. Moreover, since the standard sequence we constructed does not in any way involve the double edge, it is not difficult to see that the proof would be literally the same as the proof of the analogous statement for Brauer tree algebras in \cite{Alexandra} (see Lemma 6). 
 
 \subsection{The Triple Tree case}
 Suppose that $G$ is of the Triple Tree type and let $G_1, G_2, G_3$ be the three trees forming it. In addition, denote by $m_n = |E(G_n)|$ the number of edges in the corresponding tree, $n=1,2,3$. Then one has $m_1 + m_2 + m_2 = m-3$. We shall construct the series $\mu_G^{st}$ in several steps, computing the required tilting complex along the way: 
 
 \begin{enumerate}
     \item First apply $\mu^+_{2}$ to $\Lambda$. We get $\mu^+_2(\Lambda) = \bigoplus\limits_{i\neq 2}P_i \oplus (P_2 \xrightarrow{\delta_2} P_m)[1]$ and the modified Brauer tree corresponding to $\End_{\D^b(\Lambda)}(\mu^+_2(\Lambda))$ is as pictured below in Figure  \ref{fig:stmut1} (in the middle). 
     \item Now let $$\theta := \mu_{m_1+m_2+2}^+ \circ \dots \circ \mu_{m_1+3}^+ \circ  (\mu^+_{m_1+2})^2 \circ \dots \circ (\mu_3^+)^2$$
   $$Q_k := \bigg(P_k \xrightarrow{\begin{pmatrix} -\beta_k \dots \beta_{4} \alpha_1 \\ 
 \beta_k \dots \beta_{4} \alpha_2 \end{pmatrix}} P_1\oplus P_{2} \xrightarrow{\begin{pmatrix} \delta_1 & \delta_2 \end{pmatrix}} P_{m}\bigg) [2] \text{ for } k = 3, \dots, m_1+2$$ 
 $$L_t := (P_t \xrightarrow{\beta_t \dots \beta_4 \alpha_1} P_1)[1] \text{ for } t = m_1+3, \dots, m_1+m_2+2$$ 
 
 Then we have 
 
 $$\theta(\mu_2^+(\Lambda)) =  P_{m} \oplus P_1 \oplus (P_{2} \xrightarrow{\delta_2} P_{m})[1] \oplus \bigoplus\limits_{i=m_1+m_2+3}^{m-1} P_i \oplus \bigoplus\limits_{k=3}^{m_1+2} Q_k \oplus \bigoplus\limits_{t=m_1+3}^{m_1+m_2+2} L_t $$
 
 And the modified Brauer tree of $\End_{\D^b(\Lambda)}(\theta\circ \mu^+_2(\Lambda))$ is a Triple Tree formed by three stars with $m_1, m_2$ and $m_3$ edges respectively, as Figure \ref{fig:stmut1} below shows (on the right). 
 
  \begin{figure}[H]
  \includegraphics[scale=1.0]{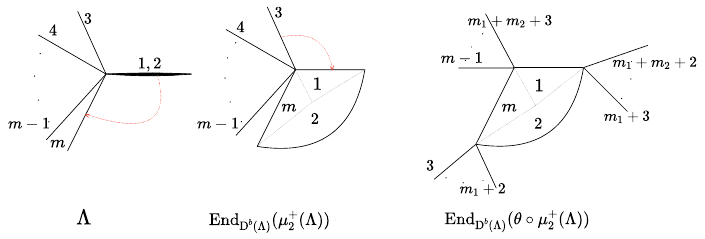}
  \caption{The modified Brauer trees of $\Lambda$, $\End_{\D^b(\Lambda)}(\mu^+_2(\Lambda))$  and  $\End_{\D^b(\Lambda)}(\theta\circ \mu^+_2(\Lambda))$ respectively.}
  \label{fig:stmut1}
\end{figure}

 \item Now what remains is to mutate each of the three stars into the desired trees $G_1$, $G_2$, $G_3$ in the same fashion as it has been just done in the Double Edge case.  With this aim in view, we first introduce an internal labeling of edges of each of the trees $G_i$ using the depth-first algorithm. In each of the trees $G_i$ there is a distinguished vertex which is one of the three vertices of the central triangle of $G$. We choose this vertex to be the root of the corresponding tree and label the edges of $G_1$, $G_2$ and $G_3$ with integers from $3$ to $m_1+2$, from $m_1+3$ to $m_1+m_2+2$, and from $m_1+m_2+3$ to $m-1$ respectively, same as in the corresponding star trees we obtained on the previous step. Similarly as for the Double Edge case, we also introduce the maps $\varphi_i \colon E(G_i) \to \{0, \dots, m_i - 1\}$ and $\psi_i \colon E(G_i) \setminus \varphi_i^{-1}(0) \to E(G_i)$ for each $i = 1, 2, 3$. In particular, $\varphi_1(3) = \varphi_2(m_1+3) = \varphi_3(m_1+m_2+3) = 0$. We have 
 
 $$\mu^{st}_{G_1} := (\mu_{m_1 + 2}^+)^{\varphi_1(m_1+2)} \circ \dots \circ (\mu_3^+)^{\varphi_1(3)} $$
 $$\mu^{st}_{G_2} := (\mu_{m_1 + m_2+2}^+)^{\varphi_2(m_1 + m_2+2)} \circ \dots \circ (\mu_{m_1+3}^+)^{\varphi_2(m_1+3)}$$ 
 $$\mu^{st}_{G_3} := (\mu_{m-1}^+)^{\varphi_3(m-1)} \circ \dots \circ (\mu_{m_1+m_2+3}^+)^{\varphi_3(m_1+m_2+3)}$$ 
 \end{enumerate}
 
 In the light of the discussion above, the following lemma is straightforward. 
 
 \begin{lemma} Let $\mu_G^{st} = \mu^{st}_{G_3} \circ \mu^{st}_{G_2} \circ \mu^{st}_{G_1} \circ \theta \circ \mu_2^+$. Then the modified Brauer tree of $\End_{\D^b(\Lambda)}(\mu_G^{st}(\Lambda))$ is $G$. The labeling of edges induced by $\mu_G^{st}$ is such that the edges forming the central triangle are labeled with $1, 2$ and $m$, and the labeling of edges inside each of the three trees is as described above, according to the depth-first algorithm. In addition, the indecomposable summands of the tilting complex $T^{st}_G = \mu_G^{st}(\Lambda)$ are as follows 
 
 $$(T^{st}_G)_{m} = P_{m}, \quad (T^{st}_G)_{2} =  (P_{2} \xrightarrow{\delta_2} P_{m})[1], \quad (T^{st}_G)_1 = P_1$$ 
  
$$ (T^{st}_G)_i = P_i \text{ if } \varphi_3(i) = 0 $$
$$ (T^{st}_G)_i = \bigg(P_i \to P_{\psi_3(i)}\bigg)[\varphi_3(i)] \text{ if } \varphi_3(i) \neq 0$$
where  $m_1 + m_2+3 \leq i \leq m - 1 \ \ (i \in E(G_3))$, 

$$ (T^{st}_G)_i = L_i, \text{ if } \varphi_2(i) = 0 $$
$$ (T^{st}_G)_i = \bigg(P_i \to P_{\psi_2(i)}\bigg)[\varphi_2(i)+1], \text{ if } \varphi_2(i) \neq 0 $$
where $m_1 + 3 \leq i \leq m_1 + m_2 + 2 \ \ (i \in E(G_2))$, and, finally, 

$$ (T^{st}_G)_i = Q_i, \text{ if } \varphi_1(i) = 0 $$
$$ (T^{st}_G)_i = \bigg(P_i \to P_{\psi_1(i)}\bigg)[\varphi_1(i)+2], \text{ if } \varphi_1(i) \neq 0 $$
where $3 \leq i \leq m_1+2 \ \ (i \in E(G_1))$. 

We will say that $\mu_G^{st}$ is the standard sequence of mutations for $G$. 

 \end{lemma}

\section{Autoequivalences of the form \texorpdfstring{$(\widetilde{\mu}_{\mu_j^+(G)}^{st})^{-1} \circ \mu_j^+ \circ \mu_G^{st}$}{Lg}}\label{sec:stanyst} 
 
 In this section we finish the proof of the main result, fulfilling step III of the plan sketched in Section \ref{sec:mainres}. Recall that, in the notation of Section \ref{sec:mainres}, the aim of this step is to express the complexes $(\widetilde{\mu}_{G^1}^{st})^{-1} \circ \mu_{j_1}^+(\Lambda)$, $\mu_{j_k}^+ \circ \widetilde{\mu}_{G^{k-1}}^{st}(\Lambda^{k-1})$ and $(\widetilde{\mu}_{G^{s+1}}^{st})^{-1} \circ \mu_{j_{s+1}}^+ \circ \widetilde{\mu}_{G^s}^{st}(\Lambda^s)$ as images of $\Lambda$ under some series of spherical twists $t_i$ modulo $\Pic(\Lambda)$ and the shift, for any indices $j_1, \dots, j_k \in \{1, \dots, m\}$.

We will begin with the most tedious part of this plan, namely, the complexes of the form $(\widetilde{\mu}_{G^{s+1}}^{st})^{-1} \circ \mu_{j_{s+1}}^+ \circ \widetilde{\mu}_{G^s}^{st}(\Lambda^{\sigma_s})$. First let us get rid of the excessive notation. Let $G$ be any modified Brauer tree. Recall that by $\mu_G^{st}$ we denote the standard sequence of mutations for $G$. By $\mu_j^+(G)$ we will denote the modified Brauer tree of $\End_{\D^b(\Lambda)}(\mu_j^+ \circ \mu_G^{st}(\Lambda))$. Let $\tau$ be the permutation one needs to apply to the standard labeling of edges of $\mu_j^+(G)$ to obtain be the labeling emerging via $\mu_j^+ \circ \mu_G^{st}$. As before, apply $\tau$ to the indices of mutations in $(\mu_{\mu_j^+(G)}^{st})^{-1}$ and denote the resulting series of mutations by $(\widetilde{\mu}_{\mu_j^+(G)}^{st})^{-1}$. It is sufficient to express the complexes of the form $(\widetilde{\mu}_{\mu_j^+(G)}^{st})^{-1} \circ \mu_j^+ \circ \mu_G^{st}(\Lambda)$ as images of $\Lambda$ under a series of spherical twists $t_i$ modulo $\Pic(\Lambda)$ and the shift, for any modified Brauer tree $G$ and any index $j = 1, \dots, m$. Whenever we write that two tilting complexes are equal, we mean that they are isomorphic modulo $\Pic(\Lambda)$. 

\subsection{Double Edge type}
First we are going to assume that $G$ is a Double Edge tree. Let $\varphi = \varphi_G, \psi = \psi_G$ be the maps introduced in Section \ref{sec:stmut}. With only a few exceptions, here we shall essentially repeat the results of \cite{Alexandra}. In each of the cases we will denote the autoequivalence $(\widetilde{\mu}_{\mu_j^+(G)}^{st})^{-1} \circ \mu_j^+ \circ \mu_G^{st}$ by $\mathcal{H}$ and the tilting complex $\mathcal{H}(\Lambda)$ by $\mathcal{T}$. By $\mathcal{T}_i$ we will always denote $\mathcal{H}(P_i)$. However, it will be often convenient to consider the tilting complex $\mathcal{T}$ with the labeling of summands coming from the series of mutations $(\widetilde{\mu}_{\mu_j^+(G)}^{st})^{-1} \circ \mu_j^+ \circ \mu_G^{st}$. In this case we will write $\mathcal{T}'_i$ instead of $\mathcal{T}_i$, meaning that $\mathcal{T}_i'$ is the summand of $\mathcal{T}$ corresponding to the edge of the modified Brauer tree of $\End(\mathcal{T})$ labeled with $i$. In addition, we will denote the tilting complex $\mu_j^+ \circ \mu_G^{st}(\Lambda)$ by $T$. 

\begin{enumerate}
\item Let $j$ be such that the mutation $\mu_j^+$ in $(\widetilde{\mu}_{\mu_j^+(G)}^{st})^{-1} \circ \mu_j^+ \circ \mu_G^{st}$ {\it does not involve irregular edges}. More precisely, (1) $j$ is not a double edge, i.e. $j \neq 1, 2$, and (2) $j$ is not followed by the double edge in the cyclic ordering, i.e. $j \neq 3$. This means that the series of mutations $(\widetilde{\mu}_{\mu_j^+(G)}^{st})^{-1} \circ \mu_j^+ \circ \mu_G^{st}$ does not ``touch'' irregular edges as well. Namely, for every $i$ such that there is a mutation $\mu_i^+$ in this composition, we have $\Hom_{D^b(\Lambda)}(P_i, P_1 \oplus P_2) = 0$. Thus, ${\mathcal{T}_1 = P_1, \mathcal{T}_2 = P_2}$ and $\mathcal{H}$ restricts to an autoequivalence of $\thick(\bigoplus_{k=3}^m P_k)$, which is equivalent to the bounded derived category of the Brauer star algebra with $m-2$ indecomposable projective modules. Moreover, observe that our construction of the standard series of mutations for a Double Edge tree is in an obvious way consistent with the construction of the standard series of mutations for a Brauer tree in \cite{Alexandra}. To conclude, in this case we can directly make use of the calculations already fulfilled in \cite{Alexandra}. Let us briefly repeat the results in our notations. 
\begin{enumerate}
    \item Suppose $j$ is pendant and there is an edge $l$ following $j$ in the cyclic ordering on the same level as $j$, i.e. $\varphi(j) = \varphi(l)$. Then $\mathcal{T} = \Lambda$ and there is nothing to express. 
    \item Suppose $j$ is pendant and there is no edge following $j$ in the cyclic ordering on the same level as $j$. In particular, $\psi(j) = j-1$. Then 
    $$\mathcal{T}_i = P_i, \quad i \neq j,j-1 $$ 
    $$\mathcal{T}_{j-1} = (P_j \xrightarrow{\beta_j} P_{j-1} \xrightarrow{soc} P_{j-1})[2]$$ 
    And we have $\mathcal{H}(\Lambda) = t_j(\Lambda)$. 
    \item Suppose $j$ is not pendant and there is an edge $l$ following $j$ in the cyclic ordering with $\varphi(j) = \varphi(l)$. Let $h$ be the other edge following $j$ in the cyclic ordering (with $\varphi(h) = \varphi(j)+1$). Then
    $$ \mathcal{T}_j = P_h$$
    $$\mathcal{T}_i = (P_h \xrightarrow{soc} P_h \xrightarrow{\beta_h \dots \beta_i} P_{i-1}), \quad j+1 \leq i \leq h$$   
    $$ \mathcal{T}_i = P_i \text{ otherwise}$$
   We have $\mathcal{H}(\Lambda) = (\mu_j^-)^2 \circ (\mu_{j+1}^-)^2 \circ \dots \circ (\mu_{h-1}^-)^2(\Lambda) = t_{j+1}^{-1} \circ \dots \circ t_h^{-1}(\Lambda)$. 
    \item Finally, suppose $j$ is not pendant and there is no edge on the same level as $j$, following $j$ in the cyclic ordering. In particular, $\psi(j) = j-1$. Let $l$ be the other edge following $j$ in the cyclic ordering, with $\varphi(l) = \varphi(j) +1$, and let $h$ be the edge following $l$ in the cyclic ordering with $\varphi(h) = \varphi(l)+1$. Then we have 
    
    $$ \mathcal{T}_i = P_i, \quad 3 \leq i \leq j-2 \text{ and } h+1 \leq i \leq m$$
   $$\mathcal{T}_{j-1} = (P_l \xrightarrow{\beta_l \dots \beta_j} P_{j-1} \xrightarrow{soc} P_{j-1})[2] $$
   $$\mathcal{T}_i = (P_l \xrightarrow{\begin{pmatrix} soc \\ -\beta_l\dots\beta_j \end{pmatrix}} P_{j-1} \oplus P_l \xrightarrow{\begin{pmatrix} 0 & -soc \\ \beta_l \dots \beta_{i+1} & 0 \end{pmatrix}} P_{j-1} \oplus P_i)[2], \quad j \leq i \leq l-1$$ 
   $$\mathcal{T}_i = (P_{i+1} \xrightarrow{\beta_{i+1} \dots \beta_{j}} P_{j-1} \xrightarrow{soc} P_{j-1})[2], \quad l \leq i \leq h-1$$ 
   $$ \mathcal{T}_h = P_{j-1}$$ 
   We have 
   $$\mathcal{H}(\Lambda) = (\mu_h^+)^2 \circ \dots \circ (\mu_{l+1}^+)^2 \circ (\mu_j^-)^2 \circ \dots \circ (\mu_{l-1}^-)^2 \circ (\mu_{l}^+)^2 \circ \dots \circ (\mu_j^+)^2(\Lambda) = t_h \circ \dots \circ t_{l+1} \circ t_j^{-1} \circ \dots \circ t_{l-1}^{-1} \circ t_l \circ \dots \circ t_j(\Lambda).$$ 
\end{enumerate}

\item Now what is left to consider are the remaining several cases that are unique to our context and cannot be derived from \cite{Alexandra}, because the mutation $\mu_j^+$ in $(\widetilde{\mu}_{\mu_j^+(G)}^{st})^{-1} \circ \mu_j^+ \circ \mu_G^{st}$ {\it involves irregular edges}. 

\begin{enumerate}

\item Suppose that $j=3$ and $3$ is pendant. We have 
 $$ T_3 = (P_{3} \xrightarrow{\begin{pmatrix}\alpha_1 \\ \alpha_2 \end{pmatrix}} P_1 \oplus P_{2})[1] $$ 
  \begin{figure}[H]
  \includegraphics[scale=1.2]{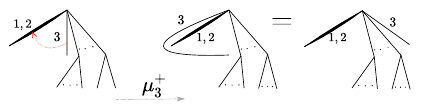}
  \label{fig:cases2}
\end{figure}
The tree $\mu_j^+(G)$ is again a Double Edge tree, the levels of whose edges are the same as in $G$. Hence we have  ${(\widetilde{\mu}^{st}_{\mu_{3}^+(\Gamma)})^{-1} = (\mu_{4}^-)^{\varphi(4)} \circ \dots \circ (\mu_m^-)^{\varphi(m)}}$ and 
$$\mathcal{T}_{i} = P_{i+1}, \quad 3 \leq i \leq m-1 $$
$$ \mathcal{T}_{1} = P_1,  \mathcal{T}_{2} = P_2$$
$$\mathcal{T}_{m} = (P_{3} \xrightarrow{\begin{pmatrix}\alpha_1 \\ \alpha_2 \end{pmatrix}} P_1 \oplus P_{2})[1]. $$
By Lemma \ref{twistseq1}, one can immediately see that $$\mathcal{H}(\Lambda) = \mathcal{F}_1(\Lambda) = t_m \circ \dots t_4 \circ t_3 \circ t_1 \circ t_2 \circ t_3[-2](\Lambda).$$ 
\item\label{difcase1} Suppose that $j=3$ and $j$ is not pendant. Let $h$ be the regular edge following $3$ in the cyclic ordering and let $l$ be the smallest integer such that $l > 3$ and $\varphi(l) = 0$ (if it exists). The complex $T_3$ is the cone of a morphism from $P_3$ to ${P_1\oplus P_2 \oplus (P_h \xrightarrow{\beta_h \dots \beta_4} P_3)[1]}$, i.e. $${T_3 = \mathcal{T}_3' = (P_h \xrightarrow{\begin{pmatrix} \beta_h \dots \beta_4 \alpha_1 \\ \beta_h \dots \beta_4 \alpha_2 \end{pmatrix}} P_1 \oplus P_2)[1]}.$$

  \begin{figure}[H]
  \centering
  \includegraphics[scale=1.3]{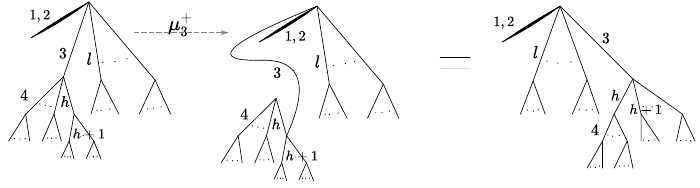}
  \label{fig:cases3}
\end{figure}

Now we need to apply $(\widetilde{\mu}^{st}_{\mu^+_3(\Gamma)})^{-1}$ to $T$. First we apply the part containing mutations with indices from $h+1$ to $l-1$. For every $i$ such that $h+1 \leq i \leq l-1$ we have $\varphi_{\mu_3^+(G)}(i) = \varphi(i)-1$. Hence $\mathcal{T}_i'$ is the cone of a morphism from $\mathcal{T}_3'$ to $(P_i \xrightarrow{\beta_i \dots \beta_{h+1}} P_h)[2]$ shifted by $-1$, i.e. ${\mathcal{T}_i' = (P_i \xrightarrow{\begin{pmatrix} \beta_i \dots \beta_4 \alpha_1 \\ \beta_i \dots \beta_4 \alpha_2 \end{pmatrix}} P_1\oplus P_2)[1]}$, for every $i$ such that $h+1 \leq i \leq l-1$. Then we apply the part of $(\widetilde{\mu}^{st}_{\mu^+_3(\Gamma)})^{-1}$ containing mutations with indices from $4$ to $h-1$. For every $i$ such that $4 \leq i \leq h-1$ we have $\varphi_{\mu_3^+(G)}(i) = \varphi(i)+1$. It is not difficult to see that $\mathcal{T}_i'$ is the cone of a morphism from $\mathcal{T}_3'$ to $(P_h \xrightarrow{\beta_h \dots \beta_{i+1}} P_i)[1]$, shifted by $-1$. More precisely, for all $i$ such that $ 4 \leq i \leq h-1$ we have $${\mathcal{T}_i = (P_h \xrightarrow{\begin{pmatrix} soc \\ -\beta_h \dots \beta_4 \alpha_1 \\ -\beta_h \dots \beta_4 \alpha_2 \end{pmatrix}} P_h \oplus P_1 \oplus P_{2} \xrightarrow{ \begin{pmatrix} \beta_h \dots \beta_{i+1} & \delta_1 \beta_m \dots \beta_{i+1} & \delta_2 \beta_m \dots \beta_{i+1}  \end{pmatrix}} P_i)[1]}.$$ Finally, since $\varphi_{\mu_3^+(G)}(h) = 1$, we apply $\mu_h^-$. The complex $\mathcal{T}_h'$ is the cone of a morphism from $(P_h \xrightarrow{\begin{pmatrix} \beta_h \dots \beta_4 \alpha_1 \\ \beta_h \dots \beta_4 \alpha_2 \end{pmatrix}} P_1 \oplus P_2)[1]$ to $(P_h \xrightarrow{\beta_h \dots \beta_4} P_3)[1]$, shifted by $-1$. 
To sum up, we have: 

$$\mathcal{T}_{i} = \mathcal{T}_{i+l-3}' = P_{i+l-3}, \quad 3 \leq i \leq m-l+3 $$
$$\mathcal{T}_{m-l+4} = \mathcal{T}_3'  = (P_h \xrightarrow{\begin{pmatrix} \beta_h \dots \beta_4 \alpha_1 \\ \beta_h \dots \beta_4 \alpha_2 \end{pmatrix}} P_1\oplus P_{2})[1]$$
$$\mathcal{T}_{m-l+5} = \mathcal{T}_h' = (P_h \xrightarrow{\begin{pmatrix} soc \\ -\beta_h\dots\beta_4\alpha_1 \\ -\beta_h \dots \beta_4 \alpha_2 \end{pmatrix}}P_h \oplus P_1 \oplus P_2 \xrightarrow{\begin{pmatrix} \beta_h \dots \beta_4 & \delta_1 \beta_m \dots \beta_4 & \delta_2\beta_m \dots \beta_4 \end{pmatrix}} P_3)[1]$$

$$\mathcal{T}_{i} = \mathcal{T}_{i-m+l-2}' =  $$

$$(P_h \xrightarrow{\begin{pmatrix} soc \\ -\beta_h\dots\beta_4\alpha_1 \\ -\beta_h \dots \beta_4 \alpha_2 \end{pmatrix}} P_h \oplus P_1 \oplus P_{2} \xrightarrow{\begin{pmatrix} \beta_h \dots \beta_{i-m+l-1} & \delta_1 \beta_m \dots \beta_{i-m+l-1} & \delta_2\beta_m \dots \beta_{i-m+l-1} \end{pmatrix}} P_{i-m+l-2})[1], $$ $$ m-l+6 \leq i \leq m-l+h+1. $$

$$\mathcal{T}_{i} = \mathcal{T}_{i-m+l-1}' = (P_{i-m+l-1} \xrightarrow{\begin{pmatrix} \beta_{i-m+l-1} \dots \beta_4 \alpha_1 \\ \beta_{i-m+l-1} \dots \beta_4 \alpha_2 \end{pmatrix}} P_1 \oplus P_{2})[1], \quad m-l+h+2 \leq i \leq m.$$ 

We have 
$\mathcal{H}(\Lambda) = t_{m-l+5}^{-1} \circ \dots \circ t_{m-l+h+1}^{-1} \circ (t_m \circ \dots \circ t_4 \circ t_3 \circ t_1 \circ t_2 \circ t_3)^{l-3}[-2(l-3)](\Lambda)$ (see Section \ref{sec:app} for details).
\item\label{difcase2} The cases $j=1$ and $j=2$ are identical, so we can assume without loss of generality that $j = 1$. Then $T_1 = \mathcal{T}_1' = (P_1 \xrightarrow{\delta_1\beta_m\dots\beta_{h+1}} P_h)[1]$, where $h$ is the edge which follows the double edge in the cyclic ordering. The tree $\mu_3^+(G)$ is now a Triple Tree, with central triangle formed by $1, 2$ and $h$. At least one of its three trees is empty and the other two trees are the subtrees of $G$ formed by edges from $3$ to $h-1$ and from $h+1$ to $m$ respectively. Now we have to pick an ordering of the trees in our configuration $\mu_1^+(G) = (G_1, G_2, G_3)$. We choose the empty tree to play the role of $G_1$. Let $\varphi_i$ be the level function of $G_i$, where $i=1,2,3$. Now apply $(\widetilde{\mu}^{st}_{\mu^+_1(\Gamma)})^{-1}$ to $T$. First we apply the part making the two (possibly) non-empty trees $G_2$ and $G_3$ into the stars. The levels of the edges in $\mu_1^+(\Gamma)$ are as follows 
    $$\varphi_2(i) = \varphi(i), \quad h+1 \leq i \leq m.$$
    $$\varphi_3(i) = \varphi(i) - 1, \quad 3 \leq i \leq h-1. $$
    Hence after this step we obtain a complex $\mathcal{T}''$ with $\mathcal{T}_i'' = P_i$ for $3\leq i \leq h-1$, ${\mathcal{T}_i'' = \mathcal{T}_i' = (P_i \xrightarrow{\beta_i \dots \beta_{h+1}} P_h)[1]}$ for $h+1 \leq i \leq m$ and $\mathcal{T}_h'' = \mathcal{T}_h' = P_h$. After that we apply $\mu_{3}^-\circ \dots \circ \mu_{h-1}^-$ and, finally, $\mu_{2}^-$ (see the picture below). 
    
    \begin{figure}[H]
    \centering
   \includegraphics[scale=1.2]{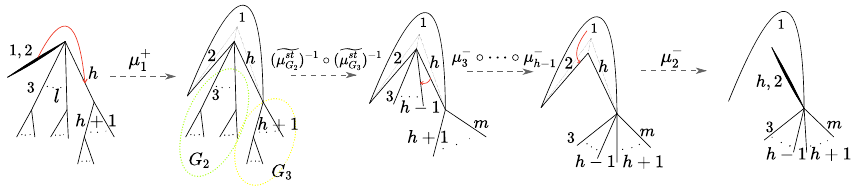}
  \label{fig:cases4}
\end{figure}
    Thus,
     $$\mathcal{T}_1 = \mathcal{T}_h' = P_h$$
         $$\mathcal{T}_2 = \mathcal{T}_2' = (P_1 \xrightarrow{\delta_1\beta_m \dots \beta_{h+1}} P_h \xrightarrow{\beta_h \beta_4 \alpha_2} P_2)[1].$$
         $$\mathcal{T}_i = \mathcal{T}_i' = (P_h \xrightarrow{\beta_h \dots \beta_{i+1}} P_i), \quad 3 \leq i \leq h-1.$$
         $$\mathcal{T}_i = \mathcal{T}_{i+1}' = (P_{i+1} \xrightarrow{\beta_{i+1} \dots \beta_{h+1}} P_h)[1], \quad h \leq i \leq m-1.$$
         $$\mathcal{T}_m = \mathcal{T}_1' = (P_1 \xrightarrow{\delta_1\beta_m\dots\beta_{h+1}} P_h)[1].$$
       
   We have $\mathcal{H}(\Lambda) = t_m\dots t_4 t_3 t_1 t_4^{-1} \dots t_h^{-1}[-1](\Lambda)$ (see Section \ref{sec:app} for details).

\end{enumerate}
\end{enumerate} 

\subsection{The Triple Tree type}
Now let $G$ be of the Triple Tree type and let $G_1, G_2$ and $G_3$ be the three trees forming it. We will need the following easy auxiliary lemma. 

\begin{lemma}{(\cite{Alexandra}, Lemma 1)}\label{equivlemma} Let $\Phi: D^b(\Lambda) \to D^b(\Lambda)$ be a triangular equivalence. Let $\mathcal{C}$ be a subcategory of $D^b(\Lambda)$ and let $X$ be an object of $D^b(\Lambda)$. Let $f \colon X \to C$ be a minimal left approximation of $X$ with respect to $\mathcal{C}$. Then $\Phi(f) \colon \Phi(X) \to \Phi(C)$ is a minimal left approximation of $\Phi(X)$ with respect to $\Phi(\mathcal{C})$. 
\end{lemma}

\begin{enumerate}
    \item Let $j$ be a regular edge of $G$ belonging to some tree $G_i$. Suppose that all edges following $j$ in the cyclic ordering also belong to the same tree $G_i$. Then, clearly, we can assume without loss of generality that the other two trees $G_d$, $d \neq i$ are empty. Then we have 
    $$\mu^{st}_G =  \mu^{st}_{G_i} \circ (\mu_{m-1}^+)^{3-i} \circ \dots \circ (\mu_3^+)^{3-i} \circ \mu_2^+. $$ 
    $$ (\widetilde{\mu}^{st}_{\mu^+_j(G)})^{-1} =    \mu_2^- \circ (\mu_{m-1}^-)^{3-i} \circ \dots \circ (\mu_3^-)^{3-i} \circ (\widetilde{\mu}^{st}_{\mu^+_j(G_i)})^{-1}.$$ 
    \noindent Applying Lemma \ref{equivlemma} to the  autoequivalence $\Phi$ of $D^b(\Lambda)$ induced by $\mu^{st}_{G_i} \circ \mu_j^+ \circ (\widetilde{\mu}^{st}_{\mu^+_j(G_i)})^{-1}$, one may assume that $i = 3$, hence $\mathcal{T} =    \mu_2^- \circ (\widetilde{\mu}^{st}_{\mu^+_j(G_i)})^{-1} \circ  \mu_j^+ \circ \mu^{st}_{G_i} \circ \mu_2^+(\Lambda)$. Finally, observe that $\mu_2^+$ commutes with $\Phi$, hence we can conclude that the resulting tilting complex $\mathcal{T}$ is the same as in the corresponding Double Edge case, so there is nothing to prove. 
    
    \item Now let $j$ be a regular edge of $G$ belonging to some tree $G_i$, such that one of the edges following $j$ in the cyclic ordering is a central edge ($1$, $2$ or $m$). 
    
    \begin{enumerate}
        \item First suppose $j \in E(G_i)$ is pendant. If in addition $i = 3$ or $i=2$, it is clear that $\mathcal{T} = \Lambda$, since $\mu_{\mu_j^+(G)}^{st} = \mu_j^+ \circ \mu_{G}^{st}$. Thus, it remains to consider the case $i=1$, $j=3$. We have 
         $$(T^{st}_G)_{3} = Q_3 = \bigg(P_3 \xrightarrow{\begin{pmatrix} -\alpha_1 \\ 
 \alpha_2 \end{pmatrix}} P_1\oplus P_{2} \xrightarrow{\begin{pmatrix} \delta_1 & \delta_2 \end{pmatrix}} P_{m}\bigg) [2]$$ $$(T^{st}_G)_{m} =P_{m}.$$
      
     Hence $$T_3 = \bigg(P_3 \xrightarrow{\begin{pmatrix} -\alpha_1 \\ 
 \alpha_2 \end{pmatrix}} P_1\oplus P_{2} \xrightarrow{\begin{pmatrix} \delta_1 & \delta_2 \end{pmatrix}} P_{m}  \xrightarrow{soc} P_{m}\bigg) [3] .$$
 
 Now we apply $(\widetilde{\mu}_{\mu_j^+(G)}^{st})^{-1}$. Observe that $(\widetilde{\mu}_{\mu_j^+(G)}^{st})^{-1}$ is just $(\mu_{G}^{st})^{-1}$ without $\mu_3^-$. Therefore,
   $$\mathcal{T}_{m-1} = \mathcal{T}_3' =  \bigg(P_3 \xrightarrow{\begin{pmatrix} -\alpha_1 \\ 
 \alpha_2 \end{pmatrix}} P_1\oplus P_{2} \xrightarrow{\begin{pmatrix} \delta_1 & \delta_2 \end{pmatrix}} P_{m}  \xrightarrow{soc} P_{m}\bigg) [3].  $$
 $$\mathcal{T}_i = \mathcal{T}_{i+1}' = P_{i+1}, \quad 3 \leq i \leq m-2.$$
 $$\mathcal{T}_m = \mathcal{T}'_m = P_m, \mathcal{T}_1 = \mathcal{T}_1' = P_1, \mathcal{T}_2 = \mathcal{T}_2' = P_2.$$ 
We have $\mathcal{H}(\Lambda) = t_m^2 \circ t_{m-1} \circ \dots \circ t_4 \circ t_3 \circ t_1 \circ t_2 \circ t_3[-2](\Lambda)$ (see Lemma \ref{twistseq1}).  
 
 \item Now suppose $j$ is not pendant.  
 
 \begin{enumerate}
     \item Let $i=3$, hence $j = m_1+m_2+3$. Let $h$ be the regular edge following $j$ in the cyclic ordering and let $l+1 > m_1+m_2+3$ be the smallest index such that $\varphi_3(l+1) = 0$. If such $l$ does not exist, it means that the subtree `under' the edge $m_1 + m_2+3$ is the whole $G_3$, in which case a few steps in what follows should be skipped. Since this only simplifies the computations, we are not going to consider this case separately.
     
     The graph $\mu_{m_1+m_2+3}^+(G) = (G_1', G_2', G_3')$ is again of the Triple Tree type. 
     \begin{center}
        
      \begin{figure}[H]
      \centering
  \includegraphics[scale=1.1]{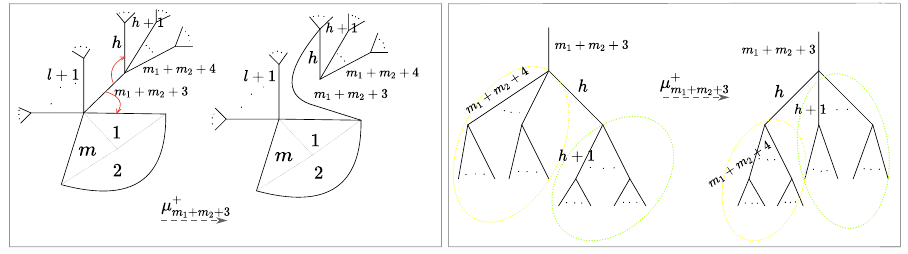}
  \label{fig:cases5}
\end{figure}
\end{center}
     We order the three trees forming it in the following way. Let $G_3'$ be the tree formed by the edges with indices from $l+1$ to $m-1$ and $G_2'$ be the tree containing the edge $3$. As before, we denote the level functions for $G_1'$, $G_2'$ and $G_3'$ by $\varphi_1', \varphi_2'$ and $\varphi_3'$ respectively. We have $T_{m_1+m_2+3} = (P_h \xrightarrow{\beta_h \dots \beta_4 \alpha_1} P_1)[1]$. The levels of edges in the new tree $G_2'$ are as follows: 
$$\varphi_2'(m_1+m_2+3) = 0, \varphi_2'(h) = 1 $$
$$\varphi_2'(i) = \varphi_3(i)+1, \quad m_1+m_2+4 \leq i \leq h-1$$
$$\varphi_2'(i) = \varphi_3(i)- 1, \quad h+1 \leq i \leq l $$
$$\varphi_2'(i) = \varphi_2(i), m_1 + 3 \leq i \leq m_1+m_2+2 $$
 Now we apply $(\widetilde{\mu}^{st}_{\mu^+_{m_1+m_2+3}(G)})^{-1}$ to $T$ in several steps. First we apply the part mutating the three trees $G_1', G_2'$ and $G_3'$ into stars, obtaining a complex $\mathcal{T}''$. We have $$\mathcal{T}'' = \mu_h^- \circ (\mu_{h-1}^-)^{\varphi_3(h-1)+1} \circ \dots \circ (\mu_{4}^-)^{\varphi_3(4)+1} \circ (\mu_{l}^-)^{\varphi_3(t)-1} \circ \dots \circ (\mu_{h+1}^-)^{\varphi_3(h+1)-1} \circ (\mu^{st}_{G_3'})^{-1} \circ (\mu^{st}_{G_1'})^{-1} (T). $$
 $$\mathcal{T}_i'' = (P_i \xrightarrow{\beta_i \dots \beta_4 \alpha_1} P_1)[1], \quad h+1 \leq i \leq l,$$
 $$\mathcal{T}_i'' = \bigg(P_h \xrightarrow{\begin{pmatrix} soc \\ -\beta_h \dots \beta_4 \alpha_1 \end{pmatrix}} P_h \oplus P_1 \xrightarrow{\begin{pmatrix} \beta_h\dots\beta_{i+1} & \delta_1 \dots \beta_{i+1}  \end{pmatrix}} P_{i}\bigg)[1], \quad  m_1+m_2+4 \leq i \leq h-1, $$ 
 $$\mathcal{T}_i'' = \mathcal{T}_i' = P_i, \quad l+1 \leq i \leq m-1 \text{ and } 3 \leq i \leq m_1+2.$$
 The complex $\mathcal{T}_h''$ is the cone of a morphism of ${(P_h \xrightarrow{\beta_h \dots \beta_4 \alpha_1} P_1)[1]}$ to ${(P_h \xrightarrow{\beta_h \dots \beta_{m_1 +m_2+4}} P_{m_1+m_2+3})[1]}$, shifted by $-1$. That is, 
$$\mathcal{T}''_h = \bigg(P_h \xrightarrow{\begin{pmatrix} soc \\ -\beta_h \dots \beta_4 \alpha_1 \end{pmatrix}} P_h \oplus P_1 \xrightarrow{\begin{pmatrix} \beta_h\dots\beta_{m_1+m_2+4} & \delta_1 \dots \beta_{m_1+m_2+4}  \end{pmatrix}} P_{m_1+m_2+3}\bigg)[1]. $$
Applying the rest of $(\widetilde{\mu}^{st}_{\mu^+_{m_1+m_2+3}(G)})^{-1}$, we get 
$$\mathcal{T}_{i}= \mathcal{T}_{i-1}' = \bigg(P_h \xrightarrow{soc} P_h \xrightarrow{\beta_h\dots\beta_{i}} P_{i-1}\bigg), \quad m_1+m_2+5 \leq i \leq h, $$
$$\mathcal{T}_{m_1+m_2+3} = P_h, $$
$$\mathcal{T}_{m_1+m_2+4} = \mathcal{T}_h' = (P_h \xrightarrow{soc} P_h \xrightarrow{\beta_4\dots\beta_{m_1+m_2+4}} P_{m_1+m_2+3}) $$
$$\mathcal{T}_i = \mathcal{T}_i' = P_i, \text{ otherwise.}$$
$$\mathcal{H}(\Lambda) = t_{m_1+m_2+4}^{-1} \circ \dots \circ t_{h}^{-1}(\Lambda).$$
The next two cases are very similar to the one just discussed. 
\item Let $i=2$, thus, $j=m_1+3$. Let $h$ be the regular edge following $j$ in the cyclic ordering and let $l+1 > m_1+3$ be the smallest index such that $\varphi_2(l+1) = 0$. Proceeding in the same way as in the previous case, we have 

$$T_{m_1+3} = (P_h \xrightarrow{ \begin{pmatrix} -\beta_h \dots \beta_4 \alpha_1 \\ \beta_h \dots \beta_4 \alpha_2 \end{pmatrix}} P_1\oplus P_{2} \xrightarrow{ \begin{pmatrix} \delta_1 & \delta_2 \end{pmatrix}} P_{m})[2] = Q_h $$

After we apply the part of $(\widetilde{\mu}^{st}_{\mu^+_{m_1+3}(G)})^{-1}$ to $T$ mutating the new tree $G_1'$ into a star, we obtain a complex $\mathcal{T}''$: 

$$\mathcal{T}''_i = Q_i, \quad h+1 \leq i \leq l.$$
$$\mathcal{T}''_i = \bigg(P_h \xrightarrow{\begin{pmatrix} soc \\ -\beta_h\dots\beta_4\alpha_1 \\ -\beta_h \dots \beta_4\alpha_2 \end{pmatrix}}P_h \oplus P_1 \oplus P_{2} \xrightarrow{ \begin{pmatrix} \beta_h\dots\beta_{i+1} && \delta_1\dots \beta_{i+1} && -\delta_2\dots \beta_{i+1} \\ \beta_h \dots \alpha_1\delta_1  && \delta_1 && -\delta_2 \end{pmatrix}} P_i \oplus P_{m} \bigg)[2], $$ $$\quad m_1 +4 \leq i \leq h-1.$$
$$\mathcal{T}''_h = \bigg(P_h \xrightarrow{\begin{pmatrix} soc \\ -\beta_h\dots\beta_4\alpha_1\\ -\beta_h\dots\beta_4\alpha_2 \end{pmatrix}}P_h \oplus P_1 \oplus P_{2} \xrightarrow{ \begin{pmatrix}  \beta_h\dots\beta_{m_1+4} && \delta_1\dots \beta_{m_1+4} && -\delta_2\dots \beta_{m_1+4} \\ \beta_h \dots \alpha_1\delta_1 && \delta_1 && -\delta_2 \end{pmatrix}} P_{m_1+3} \oplus P_{m} \bigg)[2]. $$

Applying the rest of $(\widetilde{\mu}^{st}_{\mu^+_{m_1+3}(G)})^{-1}$, we get 
$$\mathcal{T}_{m_1+3} = \mathcal{T}_{m_1+3}' = P_h, $$
$$\mathcal{T}_i = \mathcal{T}_{i-1}' = (P_h \xrightarrow{soc} P_h \xrightarrow{\beta_4 \dots \beta_i} P_{i-1}), \quad m_1+5 \leq i \leq h,$$
$$\mathcal{T}_{m_1+4}= \mathcal{T}_h' = (P_h \xrightarrow{soc} P_h \xrightarrow{\beta_4 \dots \beta_{m_1+4}} P_{m_1+3}),$$
$$\mathcal{T}_i = \mathcal{T}_i' = P_i \text{ otherwise.} $$

We have $\mathcal{H}(\Lambda) = t_{m_1+4}^{-1} \circ \dots \circ t_h^{-1}(\Lambda)$.

 \item\label{difcase3} Let $i=1$, thus, $j=3$. The complex $T_3 = \mathcal{T}_3'$ is the cone of a morphism from $Q_3$ to $(P_h \xrightarrow{\beta_h \dots \beta_4} P_3)[3]\oplus P_{m}$, namely, 
$$ T_3 = \mathcal{T}_3' = (P_h \xrightarrow{\begin{pmatrix} -\beta_h \dots \beta_4\alpha_1 \\ \beta_h \dots \beta_4\alpha_2 \end{pmatrix}} P_1\oplus P_2 \xrightarrow{\begin{pmatrix} \delta_1 & \delta_2 \end{pmatrix}} P_{m} \xrightarrow{soc} P_{m})[3]. $$
Then proceeding as in the previous two cases, we get: 
$$\mathcal{T}_i = \mathcal{T}_{i+l-2}' = P_{i+l-2}, \quad 3 \leq i \leq m-l+1. $$ 
$$\mathcal{T}_{m-l+2} = \mathcal{T}_3' = (P_h \xrightarrow{\begin{pmatrix} -\beta_h \dots \beta_4\alpha_1 \\ \beta_h \dots \beta_4\alpha_2 \end{pmatrix}} P_1\oplus P_2 \xrightarrow{\begin{pmatrix} \delta_1 & \delta_2 \end{pmatrix}} P_{m} \xrightarrow{soc} P_{m})[3]. $$
$$\mathcal{T}_{m-l+3} = \mathcal{T}_h' = (P_h \xrightarrow{\begin{pmatrix} -soc \\ -\beta_h \dots_4\alpha_1 \\ \beta_h \dots \beta_4 \alpha_2 \end{pmatrix}} P_h \oplus P_1 \oplus P_{2} \xrightarrow{\begin{pmatrix} -\beta_h \dots \beta_4 & -\delta_1 \beta_m \dots \beta_4 & -\delta_2 \beta_m \dots \beta_4 \\ 0 & \delta_1 & \delta_2 \end{pmatrix}} P_3 \oplus P_{m} \xrightarrow{soc} P_{m})[3].$$ 
$$\mathcal{T}_i= \mathcal{T}_{i-m+l}' =  $$ $\kern-10.5em \small{(P_h \xrightarrow{\begin{pmatrix} -soc \\ -\beta_h \dots_4\alpha_1 \\ \beta_h \dots \beta_4 \alpha_2 \end{pmatrix}} P_h\oplus P_1 \oplus P_{2} \xrightarrow{\begin{pmatrix} -\beta_h \dots \beta_{i-m+l+1} & -\delta_1 \beta_m \dots \beta_{i-m+l+1} & -\delta_2 \beta_m \dots \beta_{i-m+l+1} \\ 0 & \delta_1 & \delta_2 \end{pmatrix}} P_{i-m+l} \oplus P_{m} \xrightarrow{soc} P_{m}) [3],} $
	$$m-l+4 \leq i \leq m-l+h-1. $$
$$\mathcal{T}_i = \mathcal{T}_{i-m+l+1}' = (P_{i-m+l+1} \xrightarrow{\begin{pmatrix} -\beta_{i-m+l+1} \dots \beta_4\alpha_1 \\ \beta_{i-m+l+1} \dots \beta_4\alpha_2  \end{pmatrix}} P_1\oplus P_{2} \xrightarrow{\begin{pmatrix} \delta_1 & \delta_2 \end{pmatrix}} P_{m} \xrightarrow{soc} P_{m})[3], \quad m-l+h \leq i \leq m-1 $$
$$\mathcal{T}_i = \mathcal{T}_i' = P_i, \text{ otherwise.}$$
$$\mathcal{H}(\Lambda) = t_{m-l+3}^{-1} \circ \dots \circ t_{m-l+h-1}^{-1} \circ t_m \circ  \dots  \circ t_{m-l+3} \circ (t_m \circ \dots \circ t_4 \circ t_3 \circ t_1 \circ t_2 \circ t_3)^{l-2}[-2(l-2)](\Lambda).$$ 

\noindent See Section \ref{sec:app} for details. 

 \end{enumerate}
 \end{enumerate}
 \item Now let $j$ be one of the central edges $1,2$ or $m$. First we suppose that the tree following $j$ in the cyclic ordering is empty. Equivalently, $\mu_j^+(G)$ is a Double Edge tree. 
 
 \begin{enumerate}
     \item\label{difcase4}  Let $j = 1$, $G_2 = \varnothing$. 
       \begin{figure}[H]
  \includegraphics[scale=1.0]{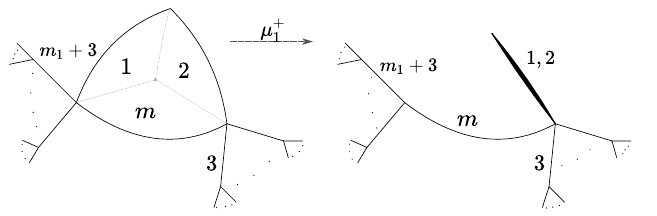}
  \label{fig:cases6}
\end{figure}
    Let $\varphi$ be the level function for the Double Edge tree $\mu_1^+(G)$.  The levels of the edges in $\mu_{1}^+(G)$  are as follows
       $$\varphi(i) = \varphi_1(i), \quad 3 \leq i \leq m_1+2, $$ 
       $$ \varphi(i) = \varphi_3(i)+1, \quad m_1 + 3 \leq i \leq m-1,$$
      $$\varphi(m) = 0$$ 
      We have 
      $(\widetilde{\mu}_{\mu_1^+(G)}^{st})^{-1} =  (\mu_{m_1+3}^-)^{\varphi(m_1+3)} \circ \dots \circ (\mu_{m-1}^-)^{\varphi(m-1)} \circ (\mu_3^-)^{\varphi(3)}\circ \dots \circ (\mu_{m_1+2}^-)^{\varphi(m_1+2)}$, thus, 
       $$\mathcal{T}_3 = \mathcal{T}_{m}' = P_{m} $$
         $$\mathcal{T}_i = \mathcal{T}_{i+m_1-1}' = (P_{m} \xrightarrow{\beta_m \dots \beta_{i+m_1}} P_{i+m_1-1}), \quad 4 \leq i \leq m-m_1. $$
      $$\mathcal{T}_i = \mathcal{T}_{i-m+m_1+2}' =  Q_{i-m+m_1+2} =  $$ $$\kern-4.5em (P_{i-m+m_1+2} \xrightarrow{\begin{pmatrix} -\beta_{i-m+m_1+2} \dots \beta_4\alpha_1 \\ \beta_{i-m+m_1+2} \dots \beta_4\alpha_2 \end{pmatrix}} P_1 \oplus P_2 \xrightarrow{\begin{pmatrix}\delta_1 & \delta_2 \end{pmatrix}} P_m)[2], \quad m-m_1+1 \leq i \leq m.$$
      $$\mathcal{T}_2 = \mathcal{T}_{2}' = (P_{2} \xrightarrow{\delta_2 } P_{m})[1] $$ 
      $$\mathcal{T}_1 = \mathcal{T}_1' = (P_1 \xrightarrow{\delta_1} P_{m})[1] $$
  $$\mathcal{H}(\Lambda) = t_3^{-1} \circ t_4^{-1} \circ \dots t_{m-m_1}^{-1} \circ (t_m \circ \dots \circ t_4 t_3 \circ t_1 \circ t_2 \circ t_3)^{m_1}[-2m_1+1](\Lambda).$$
     The next two cases are very similar to the one just examined, so we are going to discuss them briefly.  
     \item\label{caselast} Let  $j = 2, G_1 = \varnothing$. 
         We have $T_2 = (P_{2} \xrightarrow{\delta_2} P_{m} \xrightarrow{\beta_m \dots \beta_4 \alpha_1} P_{1})[2].$ Now let $\varphi$ be the level function of $\mu_2^+(G)$. Then 
      ${\varphi(i) = \varphi_3(i), m_2 +3 \leq i \leq m-1}$, ${ \varphi(i) = \varphi_2(i)+1,  3 \leq i \leq m_2+2}$, ${\varphi(1) = 0}$.
      $$\mathcal{T}_1 = \mathcal{T}_{m}' = P_{m}.$$ 
      $$\mathcal{T}_2 = \mathcal{T}_{2}' = (P_{2} \xrightarrow{\delta_2} P_{m} \xrightarrow{\beta_m \dots \beta_4 \alpha_1} P_{1})[2].$$
      $$\mathcal{T}_3 = \mathcal{T}_{1}' = P_{1}. $$
      $$\mathcal{T}_i = \mathcal{T}_{i-1}' = P_{i-1}, \quad 4 \leq i \leq m. $$
   We have $\mathcal{H}(\Lambda) = t_1^{-1}\circ t_3^{-1}\circ t_1^{-1}\circ t_4^{-1} \circ \dots \circ t_m^{-1}[2](\Lambda)$ (see Section \ref{sec:app} for details).
      \item Let $j=m, G_3 = \varnothing$.  We have $\mathcal{T}_m = P_2[1]$. Again, let $\varphi$ be the level function of $\mu_m^+(G)$. Then $\varphi(i) = \varphi_2(i), m_1+3 \leq i \leq m-1 $, $\varphi(i) = \varphi_1(i)+1,  3\leq i \leq m_1+2$, $\varphi(2) = 0$.
      $$\mathcal{T}_i = \mathcal{T}_{i-1}' = L_{i-1} = (P_{i-1} \xrightarrow{\beta_{i-1} \dots \beta_4 \alpha_1} P_1)[1],  4 \leq i \leq m $$
      $$\mathcal{T}_3 = \mathcal{T}_{2}' = (P_{2} \xrightarrow{\delta_2} P_{m})[1] $$ 
      $$ \mathcal{T}_2 = \mathcal{T}_{m}' =  P_2[1], \mathcal{T}_1 = \mathcal{T}_{1}' = P_1$$ 
      
      We have $\mathcal{H}(\Lambda) = t_2 \circ (t_m \circ \dots \circ t_4 \circ t_3 \circ t_1 \circ t_2 \circ t_3)^{m-3}[-2(m-3)](\Lambda)$ (see Lemma \ref{twistseq1}). 
 \end{enumerate}
 
 \item Finally, it remains to consider the cases when $j$ is a central edge ($1, 2$ or $m$) and the tree following it in the cyclic ordering is {\it not} empty. In other words, $j$ is central and $\mu_j^+(G)$ is again of the Triple Tree type. 
 
 \begin{enumerate}
     \item\label{difcase5} Let $j =1, G_2 \neq \varnothing$. Let $h$ be the edge of $G_2$ which follows $1$ in the cyclic ordering.
    $$(T^{st}_G)_1 = P_1, (T^{st}_G)_{m} = P_{m}   $$
    $$(T^{st}_G)_{h} = (P_{h} \xrightarrow{\beta_h \dots \beta_4\alpha_1} P_1)[1] $$
   Thus, 
     $$T_1 = \bigg(P_{h} \oplus P_1 \xrightarrow{\begin{pmatrix} \beta_{h}\dots \beta_4 \alpha_1 & id \\ 0 & \delta_1  \end{pmatrix}} P_1 \oplus P_{m} \bigg)[1] \sim \bigg(P_{h} \xrightarrow{\beta_h \dots \beta_4 \alpha_1 \delta_1} P_{m} \bigg)[1]$$ 
     \begin{center}

      \begin{figure}[H]
  \includegraphics[scale=1.0]{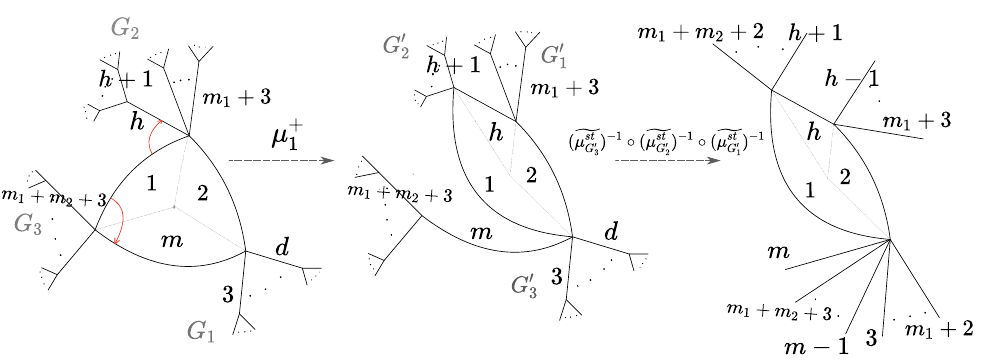}
  \label{fig:cases7}
  \end{figure}
   \end{center}
     
We denote the three trees forming $\mu_1^+(G)$ by $G_1', G_2'$ and $G_3'$, as pictured above. Namely, the tree $G_3'$ has edges with labels $3, \dots, m_1+2$ and $m_1+m_2+3, \dots, m$, the tree $G_2'$ has edges with labels $h+1, \dots, m_1+m_2+2$, and the tree $G_1'$ has edges with labels $m_1 + 3, \dots, h-1$. Edges $1, 2$ and $h$ form the central triangle of $\mu_1^+(G)$. Denote by $\varphi_1', \varphi_2'$ and $\varphi_3'$ the level functions of $G_1'$, $G_2'$ and $G_3'$ respectively. Then
     $$\varphi_3'(i) = \varphi_1(i), \quad 3 \leq i \leq m_1+2 $$
     $$\varphi_3'(m) = 0 $$
     $$\varphi_3'(i) = \varphi_3(i) + 1, \quad m_1+m_2+3 \leq i \leq m-1  $$
     $$ \varphi_1'(i) = \varphi_2(i), \quad m_1+3 \leq i \leq h-1 \quad (i \in G_1') $$
     $$ \varphi_2'(i) = \varphi_2(i)-1, \quad h+1 \leq i \leq m_1 +m_2 + 2 \quad (i \in G_2')  $$
As always, first we apply to $T$ the part of $(\widetilde{\mu}_{\mu_1^+(\Gamma)}^{st})^{-1}$ mutating the three trees $G'_1$, $G'_2$, and $G'_3$ into stars. Denote the resulting complex by $\mathcal{T}'' = (\widetilde{\mu}_{G_3'}^{st})^{-1} \circ(\widetilde{\mu}_{G_2'}^{st})^{-1} \circ (\widetilde{\mu}_{G_1'}^{st})^{-1}(T) $.
We have 
 $$\mathcal{T}_i'' = L_i, \quad i \in G_1', \text{ i.e. } m_1 + 3 \leq i \leq h-1. $$
   $$ \mathcal{T}_i'' = (P_i \xrightarrow{\beta_i \dots \beta_{h+1}} P_h)[2], \quad i \in G_2', \text{ i.e. } h+1 \leq i \leq m_1 + m_2 +2. $$
   $$ \mathcal{T}_i'' = \mathcal{T}_i' = Q_i, \quad 3 \leq i \leq m_1+2.$$
   $$ \mathcal{T}_i'' = \mathcal{T}_i' =  (P_{m} \xrightarrow{\beta_m \dots \beta_{i+1}} P_i), \quad m_1+m_2+3 \leq i \leq m-1.$$
   $$ \mathcal{T}_{m}'' = \mathcal{T}_m' = P_{m}.$$
   $$\mathcal{T}_1'' = T_1, \mathcal{T}_{2}'' = T_{2}, \mathcal{T}_h'' = T_h. $$
      Next we apply $\big((\mu_{m_1+3}^-)^2\circ \dots\circ (\mu_{h-1}^-)^2\big)\circ (\mu_{h+1}^- \circ \dots \circ \mu_{m_1+m_2+2}^-)$ and get 
    $$\mathcal{T}_i' = (P_i \xrightarrow{\beta_i \dots \beta_4 \alpha_1 \delta_1} P_{m})[1], \quad h+1 \leq i \leq m_1+m_2+2. $$  
        $$\mathcal{T}_i' = (P_{m} \xrightarrow{\beta_m \dots \beta_{i+1}} P_{i}), \quad m_1+3 \leq i \leq h-1.$$
    Finally, we apply $\mu_h^-$. 
    $$\mathcal{T}_1 = \mathcal{T}_h' = (P_2 \xrightarrow{ \begin{pmatrix} \delta_2\beta_m \dots \beta_{h+1} \\ -\delta_2 \end{pmatrix}} P_h \oplus P_m \xrightarrow{ \begin{pmatrix} \beta_h \dots \beta_4 \alpha_1 & \beta_m \dots \beta_4 \alpha_1 \end{pmatrix}} P_1)[1].$$
    $$ \mathcal{T}_2 = \mathcal{T}_1' = (P_h \xrightarrow{\beta_h \dots \beta_4 \alpha_1 \delta_1} P_m)[2].$$
    $$\mathcal{T}_i = \mathcal{T}_{i+m_1}' = (P_m \xrightarrow{\beta_m \dots \beta_{i+m_1+1}} P_{i+m_1}), \quad 3 \leq i \leq h-1-m_1.$$
    $$\mathcal{T}_i = \mathcal{T}_{i+m_1+1}' = (P_{i+m_1+1} \xrightarrow{\beta_{i+m_1+1} \dots \beta_4 \alpha_1 \delta_1} P_m)[1], \quad h-m_1 \leq i \leq m_2+1.$$ 
    $$\mathcal{T}_{m_2+2} = \mathcal{T}_m' = P_m $$
    $$\mathcal{T}_i = \mathcal{T}_{i+m_1}' = (P_m \xrightarrow{\beta_m \dots \beta_{i+m_1+1}} P_{i+m_1}), \quad m_2 +3 \leq i \leq m-m_1-1.$$
      $$ \mathcal{T}_i = \mathcal{T}_{i-m+m_1+3}' =  Q_{i-m+m_1+3} = $$ $$ (P_{i-m+m_1+3} \xrightarrow{\begin{pmatrix} -\beta_{i-m+m_1+3} \dots \beta_4 \alpha_1 \\ \beta_{i-m+m_1+3} \dots \beta_4 \alpha_2\end{pmatrix}} P_1 \oplus P_2 \xrightarrow{\begin{pmatrix} \delta_1 & \delta_2 \end{pmatrix}} P_m)[2], \quad m-m_1 \leq i \leq m-1.$$
   $$\mathcal{T}_m = \mathcal{T}_{2}' = (P_2 \xrightarrow{\delta_2} P_m)[1]$$
   $$\mathcal{H}(\Lambda) = t_2t_{m_2+2}\dots t_4t_3t_2^{-1}t_4^{-1} \dots t_{h-m_1}^{-1} t_m \dots t_4t_3t_2t_4^{-1}\dots t_{m-m_1}^{-1}  (t_m \dots t_4 t_3t_1t_2t_3)^{m_1}[-2m_1-1](\Lambda)$$
    See Section \ref{sec:app} for details. The remaining two cases are going to be similar to the one we have just considered.  
    \item\label{difcase6} Let $j=2, G_1 \neq \varnothing$. Let $d$ be the edge of $G_1$, following $2$ in the cyclic ordering.
    
    \begin{center}
     \begin{figure}[H]
 \includegraphics[scale=1.0]{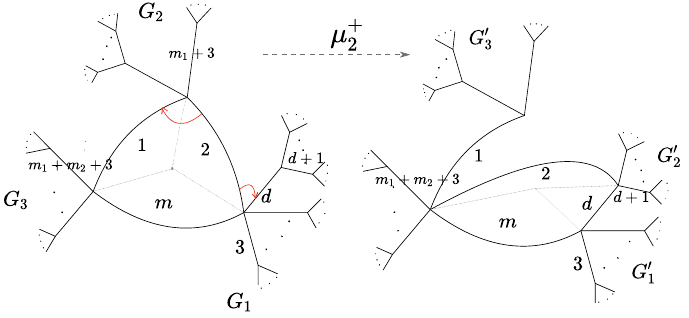}
  \label{fig:cases8}
  \end{figure}
    \end{center}
    
   First we calculate $T_{2} = \mathcal{T}_{2}'$. Since  $(T^{st}_{G})_{2} = (P_{2} \xrightarrow{\delta_2} P_{m})[1]$,  $(T^{st}_{G})_d = Q_d$ and $(T^{st}_{G})_1 =  P_1$, we have $T_{2} = \mathcal{T}_{2}' = (P_d \xrightarrow{\beta_d \dots \beta_4\alpha_1} P_1 \xrightarrow{soc} P_1)[2]$. 
     The levels of edges in the new trees forming $\mu_{2}^+(G)$ are as follows. 
     $$\varphi_3'(i) = \varphi_3(i), \quad m_1+m_2+3 \leq i \leq m-1. $$
     $$\varphi_3'(1) = 0 $$
     $$\varphi_3'(i) = \varphi_2(i)+1, \quad m_1 + 3 \leq i \leq m_1+m_2+2. $$
     $$\varphi_2'(i) = \varphi_1(i) - 1, \quad d+1\leq i \leq m_1+2.$$
     $$\varphi_1'(i) = \varphi_1(i), \quad 3\leq i\leq d-1.$$
     
     First we apply to $T$ the part of $(\widetilde{\mu}_{\mu_2^+(\Gamma)}^{st})^{-1}$ mutating the three trees $G'_1$, $G'_2$, and $G'_3$ into stars. As in the previous case, denote the corresponding complex by $\mathcal{T}''$. We immediately have
     $$\mathcal{T}_i'' = \mathcal{T}_i' = P_i, \quad m_1+m_2+3 \leq i \leq m-1. $$
     $$\mathcal{T}_i'' = Q_i, \quad 3 \leq i \leq d-1. $$
     $$\mathcal{T}_i'' = (P_i \xrightarrow{\beta_i \dots \beta_{d+1}} P_d)[3], \quad d+1 \leq i \leq m_1+2.$$
    
Now we compute $\mathcal{T}_i''$ for $m_1+3 \leq i \leq m_1 + m_2+2$. The complex $\mathcal{T}_{i}''$ is the cone of the morphism $(0, id)$ from $T_1 = P_1$ to $(P_i \xrightarrow{\beta_i \dots \beta_4\alpha_1} P_1)[1]$, shifted by $-1$, hence 
$$\mathcal{T}_i'' = P_i, \quad m_1+3 \leq i \leq m_1+m_2+2 $$
Now we shall apply $(\mu_{d+1}^-) \circ \dots \circ (\mu_{m_1+2}^-)$. The complex
$\mathcal{T}_i'$ is the cone of the morphism $(0, id, 0, 0)$ from ${\mathcal{T}_{2}'' = (P_d \xrightarrow{\beta_d \dots \beta_4\alpha_1} P_1 \xrightarrow{soc} P_1)[2]}$ to ${\mathcal{T}_i'' = (P_i \xrightarrow{\beta_i \dots \beta_{d+1}} P_d)[3]}$ for ${d+1 \leq i \leq m_1+2}$. Thus,
$$\mathcal{T}_i' = (P_i \xrightarrow{\beta_i \dots \beta_4 \alpha_1} P_1 \xrightarrow{soc} P_1)[2], \quad d+1 \leq i \leq m_1+2. $$
Next apply $(\mu_3^-)^2 \circ \dots \circ (\mu_{d-1}^-)^2$. Let $3 \leq i \leq d-1$. After one  mutation $\mu_i^{-}$ in this sequence, the summand labeled $i$ will be isomorphic to the cone of the morphism $(0,id,id)$ from $Q_d$ to $Q_i$, shifted by $-1$. This is just $(P_d \xrightarrow{\beta_d \dots \beta_{i+1}} P_i)[2]$. Applying $\mu_i^-$ again, we see that $\mathcal{T}_i'$ is the cone of the morphism $(soc,\delta_1\beta_m\dots\beta_{i+1},0)$ from $(P_d \xrightarrow{\beta_d \dots \beta_{4}\alpha_1} P_1 \xrightarrow{soc} P_1)[2]$ to $(P_d \xrightarrow{\beta_d \dots \beta_{i+1}} P_i)[2]$, shifted by $-1$. 
Hence 
$$\mathcal{T}_i' = (P_d \xrightarrow{\begin{pmatrix} soc \\ -\beta_d \dots \beta_{4}\alpha_1 \end{pmatrix}} P_d \oplus P_1 \xrightarrow{\begin{pmatrix} 0 & -soc  \\ \beta_d \dots \beta_{i+1} &  \delta_1\beta_m \dots \beta_{i+1}  \end{pmatrix}} P_1\oplus P_i)[2], \quad 3 \leq i \leq d-1. $$

Finally, apply $\mu_d^-$. Recall that $T_{m}=P_{m}$, $T_d = Q_d$. Hence $${\mathcal{T}_d' = (P_d \xrightarrow{\begin{pmatrix} -\beta_d \dots \beta_4 \alpha_1 \\ \beta_d \dots \beta_4 \alpha_1 \end{pmatrix}} P_2\oplus P_1)[1]}.$$ Summarising what we obtained above: 

$$\mathcal{T}_1 = \mathcal{T}_2' = (P_d \xrightarrow{\beta_d \dots \beta_4 \alpha_1} P_1 \xrightarrow{soc} P_1)[2].$$ 
$$\mathcal{T}_2 = \mathcal{T}_d' = (P_d \xrightarrow{\begin{pmatrix} -\beta_d \dots \beta_4 \alpha_1 \\ \beta_d \dots \beta_4 \alpha_1 \end{pmatrix}} P_2\oplus P_1)[1].$$ 
$$\mathcal{T}_i = \mathcal{T}_i' = (P_d \xrightarrow{\begin{pmatrix} soc \\ -\beta_d \dots \beta_{4}\alpha_1 \end{pmatrix}} P_d \oplus P_1 \xrightarrow{\begin{pmatrix} 0 & -soc  \\ \beta_d \dots \beta_{i+1} &  \delta_1\beta_m \dots \beta_{i+1}  \end{pmatrix}} P_1\oplus P_i)[2], \quad 3 \leq i \leq d-1.$$
$$\mathcal{T}_i = \mathcal{T}_{i+1}' = (P_{i+1} \xrightarrow{\beta_{i+1} \dots \beta_4 \alpha_1} P_1 \xrightarrow{soc} P_1)[2], \quad d \leq i \leq m_1+1.$$ 
$$\mathcal{T}_{m_1+2} = \mathcal{T}_1' = P_1.$$ 
$$\mathcal{T}_i = \mathcal{T}_i' = P_i, \quad m_1 +3 \leq i \leq m$$
We have $\mathcal{H}(\Lambda) = t_1\circ t_{m_1+2} \circ \dots \circ t_4 \circ t_3 \circ t_1^{-1} \circ t_4^{-1}\circ \dots \circ t_d^{-1}(\Lambda)$ (see Section \ref{sec:app} for details).
\item\label{difcase7}  Let $j=m, G_3 \neq \varnothing$.  Let $l$ be the edge of $G_1$, following $m$ in the cyclic ordering.
   
       We have $\mathcal{T}_{m}' = T_{m} = (P_{2} \xrightarrow{\delta_2 \beta_m \dots \beta_{l+1}} P_l)[1]$. The levels of edges in the new trees are as follows. 
        $$\varphi_2'(i) = \varphi_3(i)-1, \quad l+1 \leq i \leq m-1.$$
         $$\varphi_1'(i) = \varphi_3(i), \quad m_1+m_2+3 \leq i \leq l-1. $$
          $$\varphi_3'(i) = \varphi_2(i), \quad m_1 + 3 \leq i \leq m_1+m_2+2. $$
          $$\varphi'_3(i) = \varphi_1(i)+1, \quad 3 \leq i \leq m_1+2.$$
          $$\varphi_3'(2) = 0.$$
          
          \begin{center}
           \begin{figure}[H]
  \includegraphics[scale=1.0]{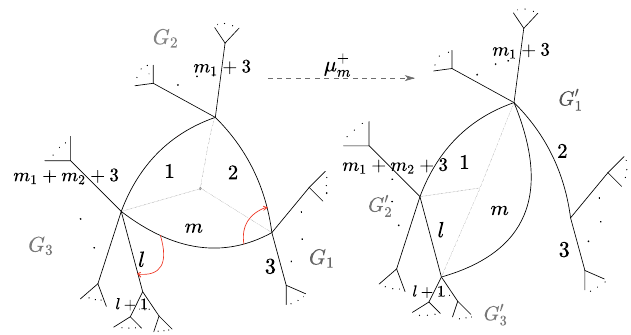}
  \label{fig:cases9}
  \end{figure}
      
          \end{center}
            Applying to $T$ the part of $(\widetilde{\mu}_{\mu_m^+(\Gamma)}^{st})^{-1}$ which mutates the three trees $G'_1$, $G'_2$, and $G'_3$ into stars, we get 
            $$\mathcal{T}_i'' = P_i, \quad m_1+m_2+3 \leq i \leq l-1.$$      $$\mathcal{T}_i'' = (P_i \xrightarrow{\beta_i \dots \beta_{l+1}} P_l)[1], \quad l+1 \leq i \leq m-1. $$   
             $$\mathcal{T}_i' = \mathcal{T}_i'' = (P_i \xrightarrow{\beta_i \dots \beta_4 \alpha_1} P_1)[1], \quad m_1 + 3 \leq i \leq m_1+m_2+2. $$ 
          $$\mathcal{T}_i' = \mathcal{T}_i'' = (P_i \xrightarrow{\beta_i \dots \beta_4 \alpha_1} P_1)[1], \quad 3 \leq i \leq m_1+2.$$
          $$\mathcal{T}_2' = \mathcal{T}_{2}'' = (P_{2} \xrightarrow{\delta_2} P_{m})[1].$$
          Then, applying the rest of $(\widetilde{\mu}_{\mu_m^+(\Gamma)}^{st})^{-1}$, i.e. ${\mu_m^-\circ(\mu_{l+1}^-)^2 \circ \dots \circ (\mu_{m-1}^-)^2 \circ \mu_{2}^{-}\circ \mu_{3}^- \circ \dots \circ \mu_{m_1+m_2+2}^-}$, we get:
          $$\mathcal{T}_1 = \mathcal{T}_1' = P_1. $$
          $$\mathcal{T}_2 = \mathcal{T}_m' = P_2.$$ 
          $$\mathcal{T}_i = \mathcal{T}_{i+l-2}' = (P_1\oplus P_2 \xrightarrow{ \begin{pmatrix} \delta_1\beta_m \dots \beta_{i+l-1} \\ \delta_1\beta_m \dots \beta_{i+l-1} \end{pmatrix}} P_{i+l-2}), \quad 3 \leq i \leq m-l+1.$$ 
          $$\mathcal{T}_{m-l+2} = \mathcal{T}_2' = (P_1\oplus P_2 \xrightarrow{ \begin{pmatrix} \delta_1 & \delta_2 \end{pmatrix}} P_m).$$
          $$\mathcal{T}_i = \mathcal{T}_{i+l-2}' = P_{i-m+l}, \quad m-l+3 \leq i \leq m.$$ 
          We have $\mathcal{H}(\Lambda) = t_2 \circ t_1 \circ (t_m \circ \dots \circ t_4 \circ t_3 \circ t_1 \circ t_2 \circ t_3)^{l-2}[-2l+2](\Lambda)$ (see Section \ref{sec:app} for details).
 \end{enumerate}
 
    \end{enumerate}

\subsection{Finishing the proof}
Recall that, in the notation of Section \ref{sec:mainres}, to finish the proof of Theorem \ref{thm:main}, it remains to express the complexes $(\widetilde{\mu}_{G_1}^{st})^{-1} \circ \mu_{j_1}^+(\Lambda)$, $\mu_{j_k}^+ \circ \widetilde{\mu}_{G_{k-1}}^{st}(\Lambda^{k-1})$ as images of $\Lambda$ under series of spherical twists $t_i$ (modulo $\Pic(\Lambda)$ and the shift)

.  

 \begin{enumerate}
     \item Fix any $j_1$. Let $\mathcal{H}^1$ denote the autoequivalence of $D^b(\Lambda)$  induced by $(\widetilde{\mu}_{G_1}^{st})^{-1} \circ \mu_{j_1}^+$ and let $\mathcal{T}^1$ be the tilting complex $\mathcal{H}^1(\Lambda)$. We denote $\mathcal{H}^1(P_i)$ by $\mathcal{T}_i^1$. It is easy to see that if $j_1 \neq 3$, then $\mathcal{T}^1 = \Lambda$. If $j_1 = 3$, then the sequence $(\widetilde{\mu}_{G_1}^{st})^{-1}$ is empty, so
     
     $$ \mathcal{T}_i^1 = P_{i+1}, \quad 3 \leq i \leq m-1. $$
     $$ \mathcal{T}_m^1 = (P_{3} \xrightarrow{\begin{pmatrix}\alpha_1 \\ \alpha_2 \end{pmatrix}} P_1 \oplus P_{2})[1]. $$
     
    By Lemma \ref{twistseq1} we have $\mathcal{H}^1(\Lambda) = t_m \circ \dots \circ t_4 \circ t_3 \circ t_1 \circ t_2 \circ t_3[-2](\Lambda)$. 
     
     \item Next we need to deal with complexes of the form $\mu_{j_k}^+ \circ \widetilde{\mu}_{G_{k-1}}^{st}(\Lambda^{k-1})$. We can assume without loss of generality that the labeling of summands of $\Lambda^{k-1}$ is standard and write simply $\mu_{j_k}^+ \circ \mu_{G_{k-1}}^{st}(\Lambda)$. Let $\mathcal{H}^k$ autoequivalence of $D^b(\Lambda)$ induced by $\mu_{j_k}^+ \circ \mu_{G_{k-1}}^{st}$. The fact that $\End_{D^b(\Lambda)}(\mu_{j_k}^+ \circ \mu_{G_{k-1}}^{st}(\Lambda)) \cong \Lambda$ leaves us with few options. Namely, there are two possibilities. 
     
     \begin{enumerate}
         \item $\mu_{G_{k-1}}^{st} = \mu_i^+$, where $4 \leq i \leq m$. Then $j_k = i$ and $\mathcal{H}^k(\Lambda) = (\mu_i^+)^2(\Lambda) = t_i(\Lambda)$.
         \item $\mu_{G_{k-1}}^{st} = \mu_2^+$. Then $j_k = 2$ and we have $ = \mu_{j_k}^+ \circ \mu_{G_{k-1}}^{st} = (\mu_2^+)^2$, hence $\mathcal{H}^k$ acts on indecomposable summands in the following way.
         
         $$\mathcal{H}^k(P_1) = P_m, \mathcal{H}^k(P_2) = (P_2 \xrightarrow{\delta_2} P_m \xrightarrow{\beta_m\dots\beta_4\alpha_1} P_1)[2] $$
         $$\mathcal{H}^k(P_3) = P_1$$ 
         $$\mathcal{H}^k(P_i) = P_{i-1}, \quad 4 \leq i \leq m$$
         
        We have already seen this autoequivalence in case \ref{caselast}. More precisely, ${\mathcal{H}^k(\Lambda) = t_1^{-1} \circ t_3^{-1} \circ t_1^{-1} \circ t_4^{-1} \dots t_m^{-1}[2](\Lambda)}$. 
        
     \end{enumerate}

 \end{enumerate}
 
 \section{Appendix}\label{sec:app} 
In Section \ref{sec:stanyst} we computed all tilting complexes $\mathcal{H}(\Lambda)$, $\mathcal{H}^1(\Lambda)$, $\mathcal{H}^k(\Lambda)$ of the form $(\widetilde{\mu}_{\mu_j^+(G)}^{st})^{-1} \circ \mu_j^+ \circ \mu_G^{st}(\Lambda)$, $(\widetilde{\mu}_{G}^{st})^{-1} \circ \mu_{j}^+(\Lambda)$, and $\mu_{j}^+ \circ \mu_{G_{k}}^{st}(\Lambda)$ respectively. In each case we stated how such a complex is expressed as an image of $\Lambda$ under a composition of spherical twists $\{t_i\}_{i}^m$, $\Pic(\Lambda)$ and shifts. Many expressions are straightforward to obtain, but some require explanations. In the Appendix we sketch step-by-step computations that justify these difficult cases.

The computations will be displayed as sequences of modified Brauer trees whose edges are labeled with indecomposable direct summands of tilting complexes obtained at every step. We already used this technique in the proof of Lemma \ref{thm:twviamut}, where complexes of the form $t_i^{\pm1}(\Lambda)$ were expressed as images of $\Lambda$ under compositions of mutations and shifts. However, now we will work with bigger ``chunks'' of mutations and all modified Brauer trees that we show will be Double Edge stars. In each of the cases we start with the modified Brauer tree of $\Lambda$ whose edges are labeled with indecomposable projective $\Lambda-$modules. Now let $\mathcal{C} = \bigoplus_{i=1}^m C_i$ be a tilting complex over with $\End_{D^b(\Lambda)}(\mathcal{C}) \cong \Lambda$ and suppose we have the Double Edge star whose edges are labeled with indecomoposable complexes $C_i$ (the double edge is labeled with two complexes). Suppose that we want to compute $t_i^{\pm1}(\mathcal{C})$. By Lemma \ref{thm:twviamut} we can express $t_i^{\pm1}$ as a composition of mutations and shifts. In turn, a sequence of mutations yields a sequence of modified Brauer trees whose edges are labeled with direct summands of tilting complexes, as explained in the proof of Lemma \ref{thm:twviamut}. Since $t_i^{\pm 1}$ is an autoequivalence of $\D^b(\Lambda)$, the last modified Brauer tree in this sequence is again the Double Edge star. In the computations below we skip all intermediate steps where the resulting modified Brauer tree is not the star and only show the trees corresponding to compositions of twists.

When there are several consequent twists in the sequence acting similarly, we often group them into one step. When there are several consequent edges in a Double Edge star which should be labeled with complexes that differ only in one index (e.g. $P_i \to P_1$ for several $i$'s), we draw a brace around these edges and write the complex only once below the brace. To know which summand of the tilting complex is assigned to a particular edge one needs to take the integer written on this edge and substitute it in place of $i$ in the complex written below the brace. For instance, in the first picture below on the left (case 2b) we have edges labeled with $P_l$, $P_{l+1}, \dots, P_m$ and edges labeled with $({P_3 \xrightarrow{} \underline{P_1 \oplus P_2}), \dots, (P_{l-1} \xrightarrow{} \underline{P_1 \oplus P_2})}$. See also the discussion after Lemma \ref{twistseq1}. To make the pictures more compact we omit the differentials in our complexes. Nevertheless, they can always be easily recovered from the provided data. The resulting tilting complexes with differentials are listed in Section \ref{sec:stanyst}. All computations here are modulo $\Pic(\Lambda)$ (see Theorem \ref{thm:pic}).

\begin{enumerate}
    \item[\underline{Case \ref{difcase1}}.]\label{app1} $\mathcal{H}(\Lambda) = t_{m-l+5}^{-1} \circ \dots \circ t_{m-l+h+1}^{-1} \circ (t_m \circ \dots \circ t_4 \circ t_3 \circ t_1 \circ t_2 \circ t_3)^{l-3}[-2(l-3)](\Lambda)$
    
      \begin{figure}[H]
 \includegraphics[scale=1.1]{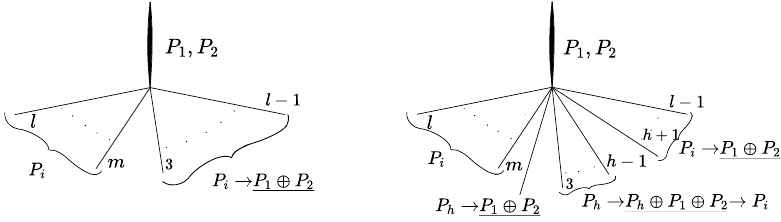}
 \tiny{ \caption*{\kern-7em Step 1.\small{$(t_m \circ \dots \circ t_4 \circ t_3 \circ t_1 \circ t_2 \circ t_3)^{l-3}[-2(l-3)]$, by Lemma \ref{twistseq1}} \ \ \ \ \ \  Step 2. \small{$t_{m-l+5}^{-1} \circ \dots \circ t_{m-l+h+1}^{-1}$} }}
  \label{fig:app1}
  \end{figure}
  \item[\underline{Case \ref{difcase2}}.]\label{app2}
  
  $\mathcal{H}(\Lambda) = t_m \circ \dots t_4 \circ t_3 \circ t_1 \circ t_4^{-1}  \circ \dots \circ t_h^{-1}[-1](\Lambda)$
  
  \begin{figure}[H]
 \includegraphics[scale=0.85]{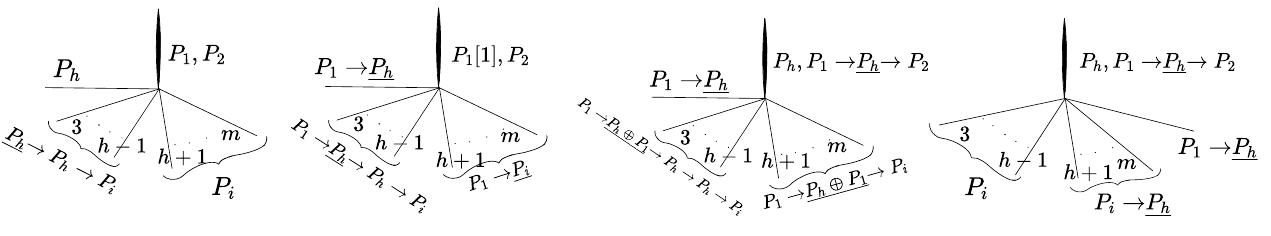}
 \tiny{ \caption*{\kern-4em Step 1.$t_4^{-1} \circ \dots \circ t_h^{-1}$ \ \ \ \ \ \ \ \ \quad \quad Step 2. $ t_1$ \ \ \ \ \ \quad \quad \quad \quad  Step 3. $ t_3[-1]$ \ \ \ \ \ \ \ \ \quad \quad \quad \quad   Step 4. $t_m \circ \dots \circ t_4$ }}
  \label{fig:app2}
  \end{figure}
   
   \item[\underline{Case \ref{difcase3}}]
   $\mathcal{H}(\Lambda) = t_{m-l+3}^{-1} \circ \dots \circ t_{m-l+h-1}^{-1} \circ t_m \circ  \dots  \circ t_{m-l+3} \circ (t_m \circ \dots \circ t_4 \circ t_3 \circ t_1 \circ t_2 \circ t_3)^{l-2}[-2(l-2)](\Lambda)$ 
   
   \begin{figure}[H]
 \includegraphics[scale=0.95]{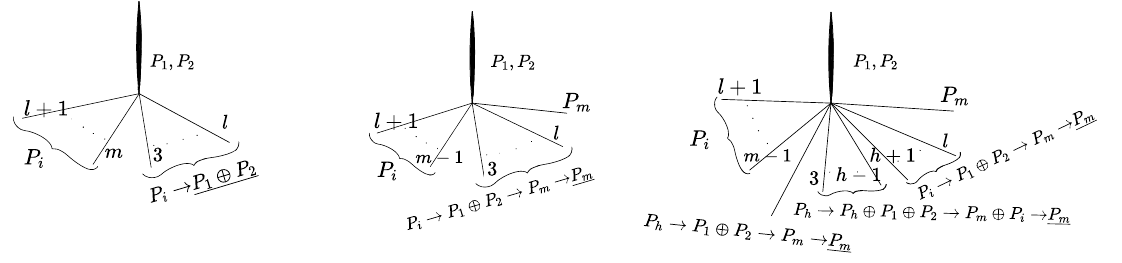}
 \caption*{\kern-4em \footnotesize{Step 1. $(t_m \circ \dots \circ t_4 \circ t_3 \circ t_1 \circ t_2 \circ t_3)^{l-2}[-2(l-2)]$, by Lemma \ref{twistseq1} \ \ \ \ Step 2. $t_m \circ  \dots  \circ t_{m-l+3}$ \ \ \ \ \  Step 3. $t_{m-l+3}^{-1} \circ \dots \circ t_{m-l+h-1}^{-1}$ }}
  \label{fig:app3}
  \end{figure}
   
      \item[\underline{Case \ref{difcase4}}] $\mathcal{H}(\Lambda) = t_3^{-1} \circ t_4^{-1} \circ \dots t_{m-m_1}^{-1} \circ (t_m \circ \dots \circ t_4 t_3 \circ t_1 \circ t_2 \circ t_3)^{m_1}[-2m_1+1](\Lambda)$.
      
         \begin{figure}[H]
 \includegraphics[scale=0.95]{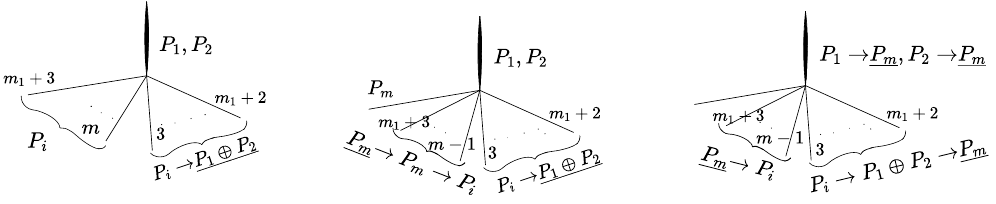}
 \tiny{ \caption*{\kern-6em \footnotesize{ Step 1.$(t_m \circ \dots \circ t_4 t_3 \circ t_1 \circ t_2 \circ t_3)^{m_1}[-2m_1]$, by Lemma \ref{twistseq1} \ \ \ \quad \quad \quad \quad   Step 2. $ t_4^{-1} \circ \dots t_{m-m_1}^{-1}$ \ \ \ \quad \quad \quad \quad  Step 3. $ t_3^{-1}[1]$ }}}
  \label{fig:app4}
  \end{figure}
  
 \item[\underline{Case \ref{caselast}}]
 $\mathcal{H}(\Lambda) = t_1^{-1}\circ t_3^{-1}\circ t_1^{-1}\circ t_4^{-1} \circ \dots \circ t_m^{-1}[2](\Lambda)$
  \begin{figure}[H]
 \includegraphics[scale=0.95]{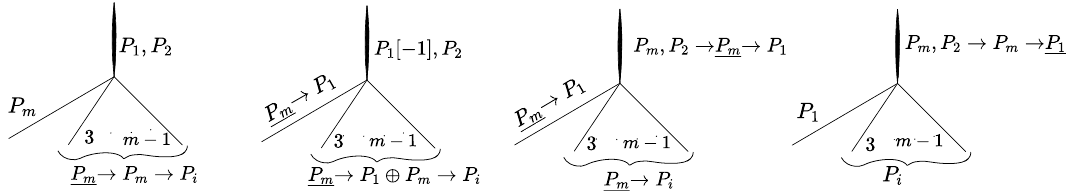}
 \tiny{ \caption*{\kern-5em Step 1. $ t_4^{-1} \circ \dots \circ t_m^{-1}$ \ \ \ \quad \quad \quad \quad  Step 2. $ t_1^{-1}$ \ \ \ \quad \quad \quad \quad  Step 3. $ t_3^{-1}[1]$  \ \ \ \quad \quad \quad \quad \quad  Step 4. $t_1^{-1}[1]$ }}
  \label{fig:applast}
  \end{figure}

\item[\underline{Case \ref{difcase5}}]
    $\mathcal{H}(\Lambda) = t_2t_{m_2+2}\dots t_4t_3t_2^{-1}t_4^{-1} \dots t_{h-m_1}^{-1} t_m \dots t_4t_3t_2t_4^{-1}\dots t_{m-m_1}^{-1}  (t_m \dots t_4 t_3t_1t_2t_3)^{m_1}[-2m_1-1](\Lambda)$
    
    \begin{figure}[H]
 \includegraphics[scale=0.95]{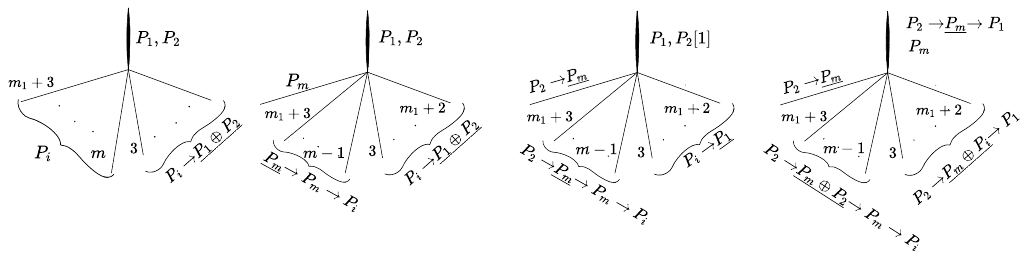}
 \tiny{ \caption*{\kern-4.5em \small{Step 1. $(t_m \circ \dots \circ t_4 \circ t_3 \circ t_1 \circ t_2 \circ t_3)^{m_1}[-2m_1]$  \quad  \quad  Step 2. $t_4^{-1}\circ \dots \circ t_{m-m_1}^{-1}$  \quad \quad   Step 3. $t_2$  \quad \quad \quad  Step 4. $t_3[-1]$ }}}
  \label{fig:app61}
  \end{figure}
  
    \begin{figure}[H]
 \includegraphics[scale=0.95]{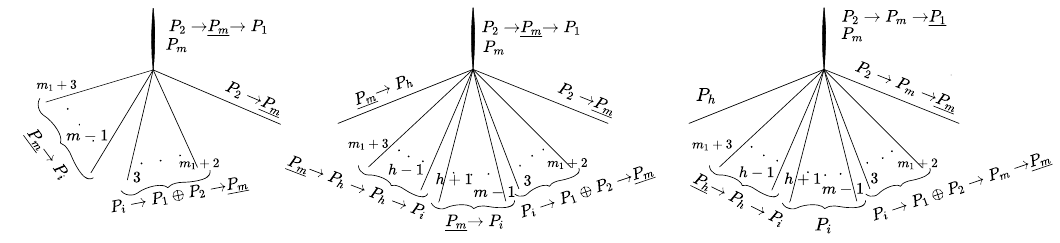}
 \tiny{ \caption*{\kern-5em Step 5. $t_m \circ \dots \circ t_4$  \quad \quad \quad \quad \quad \quad  Step 6. $t_4^{-1} \dots t_{h-m_1}^{-1}$  \quad \quad \quad \quad \quad \quad  Step 7. $t_2^{-1}[-1]$  }}
  \label{fig:app62}
  \end{figure}
  
      \begin{figure}[H]
 \includegraphics[scale=0.97]{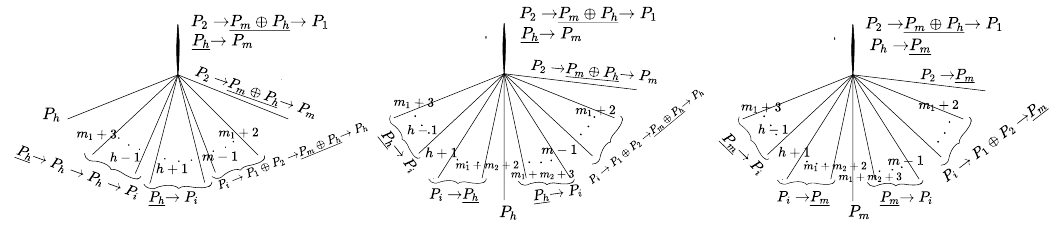}
 \tiny{ \caption*{\kern-5em Step 8. $t_3[-1]$  \quad \quad \quad \quad \quad \quad  Step 9. $t_{m_2+2}\dots t_4$  \quad \quad \quad \quad \quad \quad  Step 10. $t_2$  }}
  \label{fig:app63}
  \end{figure}
    
\item[\underline{Case \ref{difcase6}}]
  
  $\mathcal{H}(\Lambda) = t_1 \circ t_{m_1+2} \circ \dots \circ t_4 \circ t_3 \circ t_1^{-1} \circ t_4^{-1}  \circ \dots  \circ t_d^{-1}(\Lambda)$
  
  \begin{figure}[H]
 \includegraphics[scale=0.95]{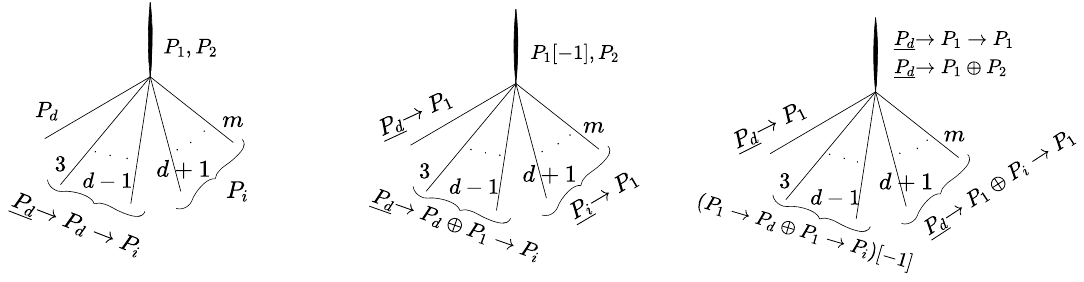}
 \tiny{ \caption*{\kern-2em Step 1.$ t_4^{-1} \circ \dots \circ t_d^{-1}$ \ \ \ \quad \quad \quad \quad \quad \quad \quad Step 2. $ t_1^{-1}$ \ \ \ \quad \quad \quad \quad \quad \quad \quad \quad \quad \quad \quad Step 3. $ t_3[-1]$  \ \ \ \quad \quad \quad \quad \quad  }}
  \label{fig:app51}
  \end{figure}
  \begin{figure}[H]
 \includegraphics[scale=0.9]{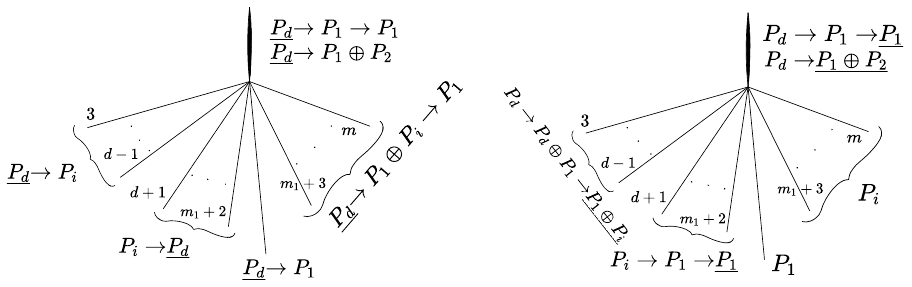}
 \tiny{ \caption*{\kern-5em Step 4.$ t_4^{-1} \circ \dots \circ t_d^{-1}$ \ \ \ \quad \quad \quad \quad \quad \quad \quad \quad \quad \quad \quad \quad \quad \quad \quad Step 5. $ t_1[1]$ \ \ \ \quad \quad \quad \quad }}
  \label{fig:app52}
  \end{figure}
  
  \item[\underline{Case \ref{difcase7}}]
  $\mathcal{H}(\Lambda) = t_2 \circ t_1 \circ (t_m \circ \dots \circ t_4 \circ t_3 \circ t_1 \circ t_2 \circ t_3)^{l-2}[-2l+2](\Lambda)$
   \begin{figure}[H]
 \includegraphics[scale=1.0]{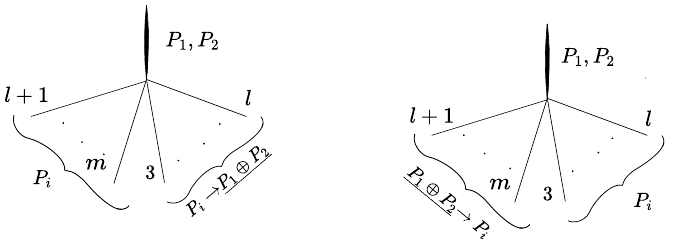}
 \tiny{ \caption*{\kern-5em Step 1.$(t_m \circ \dots \circ t_4 \circ t_3 \circ t_1 \circ t_2 \circ t_3)^{l-2}[-2l+4] $ \ \ \ \quad \quad \quad \quad \quad \quad \quad \quad  Step 2. $ t_2 \circ t_1 [-2]$}}
  \label{fig:app7}
  \end{figure}

\end{enumerate}

\end{document}